\documentclass[10pt]{amsart}
\usepackage[utf8]{inputenc}
\usepackage[T1]{fontenc}
\usepackage[english]{babel}
\title[Local limit theorems in relatively hyperbolic groups II]{Local limit theorems in relatively hyperbolic groups II : the non-spectrally degenerate case}
\author{Matthieu Dussaule}
\date{}

\usepackage{comment}
\usepackage{amsmath}
\usepackage{amssymb}
\usepackage{dsfont}
\usepackage{amsfonts}
\usepackage{amsthm}
\usepackage{tikz}
\usepackage{graphicx}
\usepackage{fancyhdr}
\usepackage{caption}
\usepackage{enumerate}
\usepackage{mathtools}
\usepackage[colorlinks=true,pdfstartview=FitH,bookmarks=false]{hyperref}
\hypersetup{urlcolor=black,linkcolor=black,citecolor=black,colorlinks=true}

\newcommand{\vertiii}[1]{{\left\vert\kern-0.25ex\left\vert\kern-0.25ex\left\vert #1 
    \right\vert\kern-0.25ex\right\vert\kern-0.25ex\right\vert}}

\newcommand\N{\mathbb{N}}
\newcommand\Z{\mathbb{Z}}

\newcommand\R{\mathbb{R}}
\newcommand\C{\mathbb{C}}

\theoremstyle{plain}
\newtheorem{definition}{Definition}[section]
\newtheorem{proposition}[definition]{Proposition}
\newtheorem{corollary}[definition]{Corollary}
\newtheorem{theorem}[definition]{Theorem}
\newtheorem{lemma}[definition]{Lemma}
\newtheorem*{prop*}{Proposition}
\newtheorem*{lem*}{Lemma}

\theoremstyle{remark}
\newtheorem{remark}{Remark}[section]

\newtheorem*{rem*}{Remark}

\DeclareMathOperator{\Cay}{Cay}
\DeclareMathOperator{\spec}{spec}

\usepackage{etoolbox}
\apptocmd{\sloppy}{\hbadness 10000\relax}{}{}
\apptocmd{\sloppy}{\vbadness 10000\relax}{}{}

\begin{document}

\begin{abstract}
    This is the second of a series of two papers dealing with local limit theorems in relatively hyperbolic groups.
    In this second paper, we restrict our attention to non-spectrally degenerate random walks, which were introduced in \cite{DussauleGekhtman} and we prove precise asymptotics of the probability $p_n(e,e)$ of going back to the origin at time $n$.
    We combine techniques adapted from thermodynamic formalism with the rough estimates of the Green function given by the first paper to show that $p_n(e,e)\sim CR^{-n}n^{-3/2}$, where $R$ is the spectral radius of the random walk.
    This generalizes results of W.~Woess for free products \cite{Woess2} and results of Gou\"ezel for hyperbolic groups \cite{Gouezel1}.
\end{abstract}

\maketitle

\section{Introduction}
Consider a finitely generated group $\Gamma$ and a probability measure $\mu$ on $\Gamma$.
We define the $\mu$-random walk on $\Gamma$, starting at $\gamma \in \Gamma$, as
$X_n^{\gamma}=\gamma g_1...g_n$,
where $(g_k)$ are independent random variables of law $\mu$ in $\Gamma$.
The law of $X_n^{\gamma}$ is denoted by $p_n(\gamma,\gamma')$.
It is given by the convolution powers $\mu^{*n}$ of the measure $\mu$.

We say that $\mu$ is admissible if its support generates $\Gamma$ as a semigroup.
We say that $\mu$ is symmetric if $\mu(\gamma)=\mu(\gamma^{-1})$.
Finally, if $\mu$ is admissible, we say that the random walk is aperiodic if $p_n(e,e)>0$ for large enough $n$.
The local limit problem consists in finding asymptotics of $p_n(e,e)$ when $n$ goes to infinity.
In many situations, if the $\mu$-random walk is aperiodic, one can prove a local limit theorem of the form
\begin{equation}\label{locallimittheoremgeneralform}
p_n(e,e)\sim C R^{-n}n^{-\alpha},
\end{equation}
where $C>0$ is a constant, $R\geq 1$ and $\alpha \in \R$.
In such a case, $\alpha$ is called the critical exponent of the random walk.

If $\Gamma=\Z^d$ and $\mu$ is finitely supported and aperiodic, then classical Fourier computations show that
$p_n(e,e)\sim Cn^{-d/2}$ if the random walk is centered and $p_n(e,e)\sim CR^{-n}n^{-d/2}$ with $R>1$ if the random walk is non-centered.
If $\Gamma$ is a non-elementary Gromov-hyperbolic group and $\mu$ is finitely supported, symmetric and aperiodic, then one has
$p_n(e,e)\sim CR^{-n}n^{-3/2}$, for $R>1$, see \cite{Gouezel1} and references therein.

Free products are a great source of examples for various local limits, see for example \cite{Cartwright1}, \cite{Cartwright2}, \cite{CandelleroGilch}.
W.~Woess proved in \cite{Woess2} that for a special class of nearest neighbor random walks on free products, called "typical case" in \cite{Woess}, one has a local limit of the form~(\ref{locallimittheoremgeneralform}), with $\alpha=3/2$.
This "typical case" should be considered informally as a situation where the random walk only sees the underlying tree structure of the free product, and not what happens inside the free factors.
So in some sense, this coefficient $3/2$ is consistent with the hyperbolic case \cite{Gouezel1}.
Our main goal in this series of two papers is to extend W.~Woess' results to any relatively hyperbolic group.

\medskip
In the first paper, we introduced the notion of spectral positive-recurrence and proved a weaker form of~(\ref{locallimittheoremgeneralform}) under this assumptions, namely that there exists $C$ such that
$C^{-1}R^{-n}n^{-3/2}\leq p_n(e,e)\leq CR^{-n}n^{-3/2}$.
In this second paper, we prove a precise local limit theorem like~(\ref{locallimittheoremgeneralform}), with $\alpha=3/2$, for non-spectrally degenerate measures on relatively hyperbolic groups.
As it was proved in the first paper, non-spectrally degenerate random walks are spectrally positive-recurrent, so our assumptions here are stronger, but we prove a more precise result.
We insist on the fact that our methods in both papers are very different and that this paper is not an enhanced version of the first one, as it uses the results of the first paper.

We will give more details on relatively hyperbolic groups and on (non-)spectrally degenerate measures in Section~\ref{Sectionbackground}.
Recall for now that a finitely generated group $\Gamma$ is relatively hyperbolic if it acts via a geometrically finite action on a proper geodesic Gromov hyperbolic space $X$.
Denote by $\Omega$ the collection of maximal parabolic subgroups, which are the stabilizers of the parabolic limit points for this action
and let $\Omega_0$ be a set of representatives of conjugacy classes of elements of $\Omega$.
Such a set $\Omega_0$ is finite.

Let $\mu$ be a probability measure on a relatively hyperbolic group $\Gamma$.
Denote by $R_\mu$ its spectral radius, that is the radius of convergence of the Green function $G(x,y|r)$, defined as
$$G(x,y|r)=\sum_{n\geq 0}p_n(x,y)r^n.$$
This radius of convergence is independent of $x,y$.
Let $\mathcal{H}\in \Omega_0$ be a parabolic subgroup.
Denote by $p_{\mathcal{H}}$ the first return kernel to $\mathcal{H}$ associated to the measure $R_\mu\mu$.
Say that a probability measure $\mu$ is spectrally degenerate along $\mathcal{H}\in \Omega_0$ if the spectral radius of $p_{\mathcal{H}}$ is 1.
Say that $\mu$ is non-spectrally degenerate if for every $\mathcal{H}\in \Omega_0$, it is not spectrally degenerate along $\mathcal{H}$.
This definition is independent of the choice of $\Omega_0$.
It was introduced in \cite{DussauleGekhtman} and appeared to be crucial in the study of the stability of the Martin boundary of relatively hyperbolic groups.
Our main goal is to prove the following.

\begin{theorem}\label{maintheorem}
Let $\Gamma$ be a non-elementary relatively hyperbolic group.
Let $\mu$ be a finitely supported, admissible and symmetric probability measure on $\Gamma$.
Assume that the corresponding random walk is aperiodic and non-spectrally degenerate along parabolic subgroups.
Then for every $\gamma,\gamma'\in \Gamma$ there exists $C_{\gamma,\gamma'}>0$ such that
$$p_n(\gamma,\gamma')\sim C_{\gamma,\gamma'}R_{\mu}^{-n}n^{-3/2}.$$
If $\mu$ is admissible but the $\mu$-random walk is not aperiodic, similar asymptotics hold for $p_{2n}(\gamma,\gamma')$ if the distance between $\gamma$ and $\gamma'$ is even and for $p_{2n+1}(\gamma,\gamma')$ if this distance is odd.
\end{theorem}

This generalize both W.~Woess's results \cite{Woess2} on free products and known results on hyperbolic groups (see \cite{GerlWoess}, \cite{Lalley}, \cite{GouezelLalley} and \cite{Gouezel1}).
As a corollary, we also get the following.

\begin{corollary}\label{maincorollary}
Let $\Gamma$ be a non-elementary relatively hyperbolic group.
Let $\mu$ be a finitely supported, admissible and symmetric probability measure on $\Gamma$.
Assume that the corresponding random walk is aperiodic and non-spectrally degenerate along parabolic subgroups.
Denote by $q_n(x,y)$ the probability that the first visit in positive time at $y$ starting at $x$ is at time $n$.
Then,
$$q_n(\gamma,\gamma')\sim C_{\gamma,\gamma'}R_{\mu}^{-n}n^{-3/2}.$$
\end{corollary}

In \cite{Gerl}, P.~Gerl conjectured that if a local limit of the form $p_n(e,e)\sim C R^{-n}n^{-\alpha}$ holds for a finitely supported random walk, then $\alpha$ is a group invariant.
This conjuecture was disproved by
D.~Cartwright in \cite{Cartwright2}.
He gave examples of local limit theorems on $\Z^d*\Z^d$, with $\alpha=d/2$ and examples on the same groups with $\alpha=3/2$.
Actually, one can get a critical exponent of the form $d/2$ only when $d\geq 5$, otherwise $\alpha$ is always $3/2$.
There are some computations to explain why in \cite{Cartwright1} (see also \cite{Woess}).
In \cite[Proposition~6.1]{DussauleGekhtman}, we gave a geometric explanation of this fact and proved
that if a parabolic subgroup $\mathcal{H}$ is virtually abelian of rank $d\leq 4$, the random walk cannot be spectrally degenerate along $\mathcal{H}$.
As a particular case, we thus get the following corollary, for Kleinian groups.

\begin{theorem}\label{maintheoremKleinian}
Let $\Gamma$ be the fundamental group of a geometrically finite hyperbolic manifold of dimension $n\leq 5$.
Let $\mu$ be a finitely supported, admissible and symmetric probability measure on $\Gamma$.
Assume that the $\mu$-random walk is aperiodic.
Then for every $\gamma,\gamma'\in \Gamma$ there exists $C_{\gamma,\gamma'}>0$ such that
$$p_n(\gamma,\gamma')\sim C_{\gamma,\gamma'}R_{\mu}^{-n}n^{-3/2}.$$
If the $\mu$-random walk is not aperiodic, similar asymptotics hold for $p_{2n}(\gamma,\gamma')$ if the distance between $\gamma$ and $\gamma'$ is even and for $p_{2n+1}(\gamma,\gamma')$ if it is odd.
\end{theorem}

\medskip
Let us now give some details on the proofs.
We will have the same approach as S.~Gou\"ezel and S.~Lalley in \cite{GouezelLalley} and \cite{Gouezel1} and we begin by explaining their work.

The first step in both papers is to get an asymptotic differential equation satisfied by the Green function.
In all this paper, we will use the following notations : if two functions $f$ and $g$ satisfy that there exists some constant $C\geq 0$ such that $f\leq C g$, then we write $f\lesssim g$.
Also, if $f\lesssim g$ and $g\lesssim f$, then we write $f\asymp g$.
Whenever we need to be specific about the constant, or about its dependence over some parameters, we will write the full inequalities to avoid being unclear.
In \cite{GouezelLalley} and \cite{Gouezel1}, the authors prove that
\begin{equation}\label{hyperbolicroughtequadiff}
    \frac{d^2}{dr^2}G(e,e|r)\asymp \left (\frac{d}{dr}G(e,e|r)\right )^3,
\end{equation}
the implicit constant not depending on $r$.
Integrating these inequalities yields
$$\left (\frac{d}{dr}G(e,e|r)\right )^{-2}-\left (\frac{d}{dr}G(e,e|R_{\mu})\right )^{-2}\asymp R_{\mu}-r,$$
so that, assuming $\frac{d}{dr}G(e,e|R_{\mu})=+\infty$ (which is proved in \cite{GouezelLalley} and \cite{Gouezel1}),
one gets
$$\frac{d}{dr}G(e,e|r)\asymp \frac{1}{\sqrt{R_{\mu}-r}}.$$
The rigorous way to proceed is to transform these a priori estimates~(\ref{hyperbolicroughtequadiff}) into an equivalent when $r$ tends to $R_{\mu}$, that is,
\begin{equation}\label{hyperbolicpreciseequadiff}
    \frac{d^2}{dr^2}G(e,e|r)\underset{r\rightarrow R_{\mu}}{\sim}C\left (\frac{d}{dr}G(e,e|r)\right )^3.
\end{equation}
Once this is established, one can prove that
$$\frac{d}{dr}G(e,e|r)\underset{r\rightarrow R_{\mu}}{\sim} \frac{C'}{\sqrt{R_{\mu}-r}}.$$
Finally, one can get asymptotics of $p_n(e,e)$ from asymptotics of $\frac{d}{dr}G(e,e|r)$ using Tauberian theorems and spectral theory.
To go from~(\ref{hyperbolicroughtequadiff}) to~(\ref{hyperbolicpreciseequadiff}), S.~Lalley and S.~Gou\"ezel use thermodynamic formalism.
Precisely, they use Cannon's result and choose a finite automaton that encodes shortlex geodesics in the hyperbolic group $\Gamma$.
They then define some H\"older functions depending on $r$ on the path space of this automaton, using the Martin kernel $K$ defined as a quotient of the Green function.
Precisely,
$K(\gamma,\gamma')=G(\gamma,\gamma')/G(e,\gamma')$, where $e$ is the neutral element of the group.
To prove that this Martin kernel is H\"older continuous, they use the strong Ancona inequalities (see Section~\ref{SectionAnconainequalities} for more details).
They then use the Ruelle-Perron-Frobenius theorem to derive asymptotic properties of this Martin kernel, when $r$ tends to $R_\mu$, which in turn leads to~(\ref{hyperbolicpreciseequadiff}).
We will give more details on thermodynamic formalism in Section~\ref{Sectionthermodynamicformalism}.

\medskip
We will adapt their proofs to the relatively hyperbolic case.
The first step is given by the results proved in the first paper.
Precisely, \cite[Theorem~1.5]{DussauleLLT1} shows that~(\ref{hyperbolicroughtequadiff}) holds again in our situation.

In the present paper, we use thermodynamic formalism to derive from these a priori estimates some precise equivalent (see Theorem~\ref{preciseequadiff}).
There will be several difficulties here.
First, we do not have a finite automaton encoding geodesics.
Anyway, geodesics are not so much interesting for our purpose.
Indeed, Ancona inequalities that are used in \cite{GouezelLalley} and \cite{Gouezel1} to prove H\"older continuity do not hold along geodesics, but along relative geodesics in relatively hyperbolic groups.
On the other hand, we proved in the first paper that there exists an automaton with finite set of vertices and countable set of edges that encodes relative geodesics, see precisely \cite[Theorem~4.2]{DussauleLLT1}.
We will use instead this automaton.
However, the associated path space will not be finite but countable.
We will thus have to use thermodynamic formalism for countable Markov shifts, which is more delicate than thermodynamic formalism for Markov shifts of finite type.
For example, there are situations where Ruelle-Perron-Frobenius theorem does not hold for countable shifts.
We will thus prove that the H\"older continuous function introduced in \cite{GouezelLalley} and \cite{Gouezel1} is positive recurrent (using the terminology of O.~Sarig in \cite{Sarig1}), which will be sufficient to mimic some of the arguments of S.~Lalley and S.~Gou\"ezel.

Another difficulty will be that the family of transfer operators $(\mathcal{L}_r)_{r\leq R_\mu}$ we introduce will not vary continuously in $r$ for the operator norm.
However, looking carefully at the proofs of \cite{GouezelLalley} and \cite{Gouezel1}, one only needs continuity of the spectral data associated to this family of operators.
We will use perturbations results due to G.~Keller and C.~Liverani \cite{KellerLiverani} to prove this sort of continuity.

Finally, the final step (getting the local limit theorem from the asymptotics of the Green function) is a combination of Tauberian theorems and spectral theory.
We will be able to use directly the results of \cite{GouezelLalley} and so we have nothing to prove there to conclude.
We will also deduce Corollary~\ref{maincorollary} from Theorem~\ref{maintheorem} using directly results of \cite{GouezelLalley}.

\subsection{Organization of the paper}

In Section~\ref{Sectionbackground}, we compile the tools and results we will need in the following.
We first give more details about relatively hyperbolic groups
and we give a proper statement on the existence of an automaton encoding relative geodesics.
We then give more details on spectrally degenerate measures and state some technical results about the Green function proved in the first paper \cite{DussauleLLT1}.
We finally recall weak and strong relative Ancona inequalities that were proved in \cite{DussauleGekhtman} and that will be used all along the paper.

In Section~\ref{Sectionthermodynamicformalism}, we review results of O.~Sarig on thermodynamic formalism for countable shifts.
This will give a general framework for the following.
We also recall perturbation theorems of G.~Keller and C.~Liverani  (see \cite{KellerLiverani}) that we will use later to obtain precise asymptotics of spectral data of a transfer operator $\mathcal{L}_r$ associated with a suitable potential $\varphi_r$.

In Section~\ref{Sectionapplicationthermodynamicformalism},
we use thermodynamic formalism to obtain a precise asymptotic of the first derivative of the Green function in terms of the dominant eigenvalue of $\mathcal{L}_r$.
This part is similar to \cite[Section~3]{Gouezel1}, although several changes have to be made.
The main difference is that our automaton that encodes relative geodesics has an infinite set of edges, so that the corresponding Markov shift has a countable number of states.
This yields several difficulties, related to the lack of compactness of the path space of the Markov shift.
The major difficulty is to find continuity of $r\mapsto \mathcal{L}_r$.
We prove a weak form of continuity, using the perturbation theorem of G.~Keller and C.~Liverani cited above.
Another difficulty comes from the fact that our automaton is not strongly connected and might have several maximal components.
We thus also have to prove that the spectral data of the transfer operator does not depend on the components of the Markov shift.
This problem already occurs in \cite{Gouezel1} and we use the same global strategy to solve it, which is itself based on the work of D.~Calegari and K.~Fujiwara, see \cite{CalegariFujiwara}.
However, we again have new difficulties here, coming from lack of compactness on the one hand and coming from the lack of knowledge on the Martin boundary on the other hand.

In Section~\ref{Sectionproofofpreciseequadiff}, we study the second derivative of the Green function $G''(e,e|r)$.
This allows us to
transform the rough a priori estimates~(\ref{hyperbolicroughtequadiff}) into a precise asymptotic, proving~(\ref{hyperbolicpreciseequadiff}).
We still follow the strategy of \cite{Gouezel1}, but again, due to lack of compactness, many of the arguments have to be changed.
All this section is very technical and we encourage the reader to first get familiar with the hyperbolic case in \cite[Section~3.6]{Gouezel1}.

To conclude, in Section~\ref{SectionFromGreentoLLT}, we recall the results of \cite{Gouezel1} and \cite{GouezelLalley} which explain how to get asymptotics of the convolution powers $\mu^{*n}$ of $\mu$ from asymptotics of the Green function.
We end there the proof of our main theorem.

\subsection{Acknowledgements}\label{Sectionacknowledgements}
The author thanks S.~Gou\"ezel for his advice and explanations on \cite{Gouezel1} and \cite{GouezelLalley}.
He also thanks
I.~Gekhtman and L.~Potyagailo for many helpful conversations about relatively hyperbolic groups.

\section{Some background}\label{Sectionbackground}
\subsection{Relatively hyperbolic groups}
We first recall definitions and basic properties of relatively hyperbolic groups.
More details are given in the first paper \cite{DussauleLLT1}.
Consider a finitely generated group $\Gamma$ acting discretely and by isometries on a proper and geodesic hyperbolic space $(X,d)$.
We denote the limit set of $\Gamma$ by $\Lambda \Gamma$, that is the set of accumulation points in the Gromov boundary $\partial X$ of an orbit $\Gamma \cdot o$, $o\in X$.
A point $\xi\in \Lambda \Gamma$ is called conical if there is a sequence $(\gamma_{n})$ of $\Gamma$ and distinct points $\xi_1,\xi_2$ in $\Lambda \Gamma$ such that
$\gamma_{n}\xi$ converges to $\xi_1$ and $\gamma_{n}\zeta$ converges to $\xi_2$ for all $\zeta\neq \xi$ in $\Lambda \Gamma$.
A point $\xi\in \Lambda \Gamma$ is called parabolic if its stabilizer in $\Gamma$ is infinite, fixes exactly $\xi$ in $\Lambda\Gamma$ and contains no loxodromic element.
A parabolic limit point $\xi$ in $\Lambda \Gamma$ is called bounded parabolic if is stabilizer in $\Gamma$ is infinite and acts cocompactly on $\Lambda \Gamma \setminus \{\xi\}$.
Say that the action is geometrically finite if the limit set only consists of conical limit points and bounded parabolic limit points.

Then, say that $\Gamma$ is relatively hyperbolic with respect to $\Omega$ if it acts geometrically finitely on such a hyperbolic space $(X,d)$ such that the stabilizers of the parabolic limit points are exactly the elements of $\Omega$.
In this situation, $\Gamma$ is said to be non-elementary if its limit set is infinite.

One might choose different spaces $X$ on which $\Gamma$ can act geometrically finitely.
However, different choices of $X$ give rise to equivariantly homeomorphic limit sets $\Lambda \Gamma$.
We call this limit set the Bowditch boundary of $\Gamma$ and we denote it by $\partial_B\Gamma$.

\medskip
Let $\Gamma$ be a relatively hyperbolic group, let $\Omega$ be the collection of parabolic subgroup and let $\Omega_0$ be a finite set of representatives of conjugacy classes of elements of $\Omega$.
Fix a finite generating set $S$ for $\Gamma$.
Denote by $\hat{\Gamma}$ the Cayley graph associated with the infinite generating set consisting of the union of $S$ and of all parabolic subgroups $\mathcal{H}\in \Omega_0$.
Endowed with the graph distance, that we write $\hat{d}$, the graph $\hat{\Gamma}$ is hyperbolic.

A relative geodesic is a geodesic in the graph $\hat{\Gamma}$.
A relative quasi-geodesic is a path of \textit{adjacent} vertices in $\hat{\Gamma}$, which is a quasi-geodesic for the distance $\hat{d}$.
We  say that a path is without backtracking if once it has left a coset $\gamma \mathcal{H}$, for $\mathcal{H}\in \Omega_0$, it never goes back to it.
Relative geodesic and relative quasi-geodesic satisfy the following property, called the BCP property.
for all $\lambda,c$, there exists a constant $C_{\lambda,c}$ such that for every pair $(\alpha_1,\alpha_2)$ of relative $(\lambda,c)$-quasi geodesic paths without backtracking, starting and ending at the same point in $\Gamma$, the following holds
\begin{enumerate}
\item if $\alpha_1$ travels more than $C_{\lambda,c}$ in a coset, then $\alpha_2$ enters this coset,
\item if $\alpha_1$ and $\alpha_2$ enter the same coset, the two entering points and the two exit points are $C_{\lambda,c}$-close to each other in $\Cay(\Gamma,S)$.
\end{enumerate}

\medskip
We will both need to study geodesics in $\Cay(\Gamma,S)$ and relative geodesics in the following.
We will use the following terminology.
Let $\alpha$ be a geodesic in $\Cay(\Gamma,S)$ and let $\eta_1,\eta_2\geq 0$.
A point $\gamma$ on $\alpha$ is called an $(\eta_1,\eta_2)$-transition point if for any coset $\gamma_0\mathcal{H}$ of a parabolic subgroup, the part of $\alpha$ consisting of points at distance at most $\eta_2$ from $\gamma$ is not contained in the $\eta_1$-neighborhood of $\gamma_0\mathcal{H}$.

Transition points are of great importance in relatively hyperbolic groups.
They stay close to points on relative geodesics in the following sense.

\begin{lemma}\label{projectiontransitionpoints}\cite[Proposition~8.13]{Hruska}
Fix a generating set $S$.
For every large enough $\eta_1,\eta_2>0$, there exists $r\geq0$ such that the following holds.
Let $\alpha$ be a geodesic in $\Cay(\Gamma,S)$
and let $\hat{\alpha}$ be a relative geodesic path with the same endpoints as $\alpha$.
Then, for the distance in $\Cay(\Gamma,S)$, any $(\eta_1,\eta_2)$-transition point on $\alpha$ is within $r$ of a point on $\hat{\alpha}$
and conversely, any point on $\hat{\alpha}$ is within $r$ of an $(\eta_1,\eta_2)$-transition point on $\alpha$.
\end{lemma}

The reason for introducing both definitions in the present paper, despite this lemma, is as follows.
On the one hand, a lot of known results in literature are stated using transition points.
On the other hand, it will be more convenient to use relative geodesics because of our relative automatic structure that we now introduce.

\subsection{Relatively automatic groups}
Hyperbolic groups are known to be strongly automatic, meaning that for every such generating set $S$, there exists a finite directed graph $\mathcal{G}=(V,E,v_*)$ encoding geodesics.
It is quite implicit in Farb's work \cite{Farbthesis}, \cite{Farb}, although not formally stated, that relatively hyperbolic groups are relatively automatic in the following sense.

Let $\Gamma$ be a finitely generated group and let $\Omega$ be a collection of subgroups invariant by conjugacy and such that there is a finite set $\Omega_0$ of conjugacy classes representatives of subgroups in $\Omega$.

\begin{definition}\label{definitionautomaticstructure}
A relative automatic structure for $\Gamma$ with respect to the collection of subgroups $\Omega_0$ and with respect to some finite generating set $S$ is a directed graph $\mathcal{G}=(V,E,v_*)$ with distinguished vertex $v_*$ called the starting vertex, where the set of vertices $V$ is finite and with a labelling map $\phi:E\rightarrow S\cup \bigcup_{\mathcal{H}\in \Omega_0}\mathcal{H}$ such that the following holds.
If $\omega=e_1,...,e_n$ is a path of adjacent edges in $\mathcal{G}$, define $\phi(e_1,...,e_n)=\phi(e_1)...\phi(e_n)\in \Gamma$.
Then,
\begin{itemize}
    \item no edge ends at $v_*$, except the trivial edge starting and ending at $v_*$,
    \item every vertex $v\in V$ can be reached from $v_*$ in $\mathcal{G}$,
    \item for every path $\omega=e_1,...,e_n$, the path $e,\phi(e_1),\phi(e_1e_2),...,\phi(\gamma)$ in $\Gamma$ is a relative geodesic from $e$ to $\phi(\gamma)$, that is the image of $e,\phi(e_1),\phi(e_1e_2),...,\phi(\gamma)$ in $\hat{\Gamma}$ is a geodesic for the metric $\hat{d}$,
    \item the extended map $\phi$ is a bijection between paths in $\mathcal{G}$ starting at $v_*$ and elements of $\Gamma$.
\end{itemize}
\end{definition}

Note that the union $S\cup \bigcup_{\mathcal{H}\in \Omega_0}\mathcal{H}$ is not required to be a disjoint union.
Actually, the intersection of two distinct subgroups $\mathcal{H},\mathcal{H}'\in \Omega_0$ can be non-empty.
Also note that we require the vertex set $V$ to be finite.
However, the set of edges is infinite, except if the parabolic subgroups $\mathcal{H}$ are finite (in which case the group $\Gamma$ is hyperbolic).

If there exists a relative automatic structure for $\Gamma$ with respect to $\Omega_0$ and $S$, we say that $\Gamma$ is automatic relative to $\Omega_0$ and $S$.
The following was proved in the first paper.

\begin{theorem}\cite[Theorem~4.2]{DussauleLLT1}\label{thmcodingrelhypgroups}
Let $\Gamma$ be a relatively hyperbolic group and let $\Omega_0$ be a finite set of representatives of conjugacy classes of the maximal parabolic subgroups.
For every symmetric finite generating set $S$ of $\Gamma$,
$\Gamma$ is automatic relative to $\Omega_0$ and $S$.
\end{theorem}

Along the proof of this theorem, a lot of technical lemmas about relative geodesics were proved in \cite{DussauleLLT1}.
We will use some of them repeatedly in this paper, so we restate them for convenience.
We use the same notations as above and so we fix a relatively hyperbolic group $\Gamma$ and a finite set $\Omega_0$ of conjugacy classes of parabolic subgroups.

\begin{lemma}\cite[Lemma~4.9]{DussauleLLT1}\label{Proposition315Osininfinite}
For every $(\lambda,c)$ and $K\geq 0$, there exists $C\geq 0$ such that the following holds.
Let $(x_1,...,x_n,...)$ and $(x_1',...,x_m',...)$ be two infinite relative geodesics such that
$d(x_1,x'_1)\leq K$ and $x_n$ and $x_m'$ converge to the same conical limit point $\xi$.
Then for every $j\geq 1$, there exists $i_j$ such that $d(x_j,x_{i_j}')\leq C$.
\end{lemma}

\begin{lemma}\cite[Lemma~4.16]{DussauleLLT1}\label{lemmarelativetripod}
Let $(e,\gamma_1,...,\gamma_n)$ and $(e,\gamma_1',...,\gamma_m')$ be two relative geodesics. 
Assume that the nearest point projection of $\gamma'_m$ on $(e,\gamma_1,...,\gamma_n)$ is at $\gamma_l$.
If there are several such nearest point projection, choose the closest to $\gamma_n$.
Then, any relative geodesic from $\gamma_m'$ to $\gamma_n$ passes within a bounded distance (for the distance $d$) of $\gamma_{l}$.
Moreover, if $\gamma_l\neq e$, then any relative geodesic from $e$ to $\gamma_m'$ passes within a bounded distance of $\gamma_{l-1}$.
\end{lemma}

\subsection{Spectrally degenerate measures}
We now consider a group $\Gamma$, hyperbolic relative to a collection of peripheral subgroups $\Omega$ and we fix a finite collection $\Omega_0=\{\mathcal{H}_1,...,\mathcal{H}_N\}$ of representatives of conjugacy classes of $\Omega$.
We assume that $\Gamma$ is non-elementary.
Let $\mu$ be a probability measure on $\Gamma$, $R_{\mu}$ the spectral radius of the $\mu$-random walk and $G(\gamma,\gamma'|r)$ the associated Green function, evaluated at $r$, for $r\in [0,R_{\mu}]$.
If $\gamma=\gamma'$, we simply use the notation $G(r)=G(\gamma,\gamma|r)=G(e,e|r)$.

\medskip
We denote by $p_k$ the first return transition kernel to $\mathcal{H}_k$.
Namely, if $h,h'\in \mathcal{H}_k$, then
$p_k(h,h')$ is the probability that the $\mu$-random walk, starting at $h$, eventually comes back to $\mathcal{H}_k$ and that its first return to $\mathcal{H}_k$ is at $h'$.
In other words,
$$p_k(h,h')=\mathbb{P}_h(\exists n\geq 1, X_n=h',X_1,...,X_{n-1}\notin \mathcal{H}_k).$$
More generally, for $r\in [0,R_{\mu}]$, we denote by $p_{k,r}$ the first return transition kernel to $\mathcal{H}_k$ for $r\mu$.
Precisely, if $h,h'\in \mathcal{H}_k$, then
$$p_{k,r}(h,h')=\sum_{n\geq 1}\sum_{\underset{\notin \mathcal{H}_k}{\gamma_1,...,\gamma_{n-1}}}r^n\mu(h^{-1}\gamma_1)\mu(\gamma_1^{-1}\gamma_2)...\mu(\gamma_{n-2}^{-1}\gamma_{n-1})\mu(\gamma_{n-1}^{-1}h').$$

We then denote by $p_{k,r}^{(n)}$ the convolution powers of this transition kernel, by $G_{k,r}(h,h'|t)$ the associated Green function, evaluated at $t$ and by $R_k(r)$ the associated spectral radius, that is, the radius of convergence of $t\mapsto G_{k,r}(h,h'|t)$.
For simplicity, write $R_k=R_k(R_{\mu})$.
As for the initial Green function, if $h=h'$, we will simply use the notation $G_{k,r}(t)=G_{k,r}(h,h|t)=G_{k,r}(e,e|t)$.

According to \cite[Lemma~3.4]{DussauleLLT1}, for any $r\in [0,R_{\mu}]$, for any $k\in \{1,...,N\}$,
$$G_{k,r}(h,h'|1)=G(h,h'|r).$$
Also, since $\Gamma$ is non-elementary, it contains a free group and hence is non-amenable.
It follows from a result of Guivarc'h (see \cite[p.~85, remark~b)]{Guivarch}) that $G(R_{\mu})<+\infty$.
Thus, $G_{k,R_{\mu}}(1)<+\infty$.
In particular, $R_k\geq 1$.

\begin{definition}
We say that $\mu$ (or equivalently the random walk) is spectrally degenerate along $\mathcal{H}_k$ if $R_k=1$.
We say it is non-spectrally degenerate if for every $k$, $R_k>1$.
\end{definition}

This definition was introduced in \cite{DussauleGekhtman} to study the homeomorphism type of the Martin boundary at the spectral radius.
As explained in \cite[Section~3.3]{DussauleLLT1}, it should be thought as a notion of spectral gap between the spectral radius of the random walk on the whole group and the spectral radii of the induced walks on the the parabolic subgroups.

\medskip
Let us now recall some consequences of spectral degenerescence proved in \cite{DussauleLLT1}.
We introduce the following notations.
We write
$$I^{(k)}(r)=\sum_{\gamma^{(1)},...,\gamma^{(k)}\in \Gamma}G(\gamma,\gamma^{(1)}|r)G(\gamma^{(1)},\gamma^{(2)}|r)...G(\gamma^{(k-1)},\gamma^{(k)}|r)G(\gamma^{(k)},\gamma'|r).$$
Then, $I^{(k)}(r)$ is related to the $k$th derivative of the Green function.
For a precise statement, we refer to \cite[Lemma~3.2]{DussauleLLT1}.
For instance, we have the following.

\begin{lemma}\cite[Lemma~3.1]{DussauleLLT1}\label{lemmafirstderivative}
For every $\gamma,\gamma'\in \Gamma$, for every $r\in [0,R_{\mu}]$, we have
$$\frac{d}{dr}(rG(\gamma_1,\gamma_2|r))=\sum_{\gamma\in \Gamma}G(\gamma_1,\gamma|r)G(\gamma,\gamma_2|r).$$
\end{lemma}

If $\mathcal{H}$ is a parabolic subgroup, we also write
$$I^{(k)}_{\mathcal{H}}(r)=\sum_{\gamma^{(1)},...,\gamma^{(k)}\in \mathcal{H}}G(\gamma,\gamma^{(1)}|r)G(\gamma^{(1)},\gamma^{(2)}|r)...G(\gamma^{(k-1)},\gamma^{(k)}|r)G(\gamma^{(k)},\gamma'|r).$$
Once $\Omega_0=\{\mathcal{H}_1,...,\mathcal{H}_N\}$ is fixed, we also write $I^{(k)}_{j}(r)=I^{(k)}_{\mathcal{H}_j}(r)$ for simplicity.

One of the main result of the first paper is that whenever $I^{(2)}_{\mathcal{H}}(r)<+\infty$ for every parabolic subgroup $\mathcal{H}\in \Omega_0$, then
$$I^{(2)}(r)\asymp \left ( I^{(1)}(r)\right )^3.$$
This is a consequence of the following result.

\begin{proposition}\cite[Proposition~5.6]{DussauleLLT1}\label{roughequadiff}
For every $r\in [0,R_{\mu})$, we have
$$\frac{I^{(2)}(r)}{\left (I^{(1)}(r)\right )^3}\asymp 1+\sum_{j}I^{(2)}_{j}(r).$$
In particular, if $\mu$ is non-spectrally degenerate, then
$$I^{(2)}(r)\asymp \left ( I^{(1)}(r)\right )^3.$$
\end{proposition}

The following results were also proved in the first paper.
Note that we do not need to assume that $\mu$ is non-spectrally degenerate.

\begin{lemma}\cite[Lemma~5.4]{DussauleLLT1}\label{finitesumalongspheres}
There exists some uniform $C\geq 0$ such that for every $r\in [0,R_{\mu}]$, for every $m$,
$$\sum_{\gamma\in \hat{S}_m}H(e,\gamma|r)\leq C.$$
\end{lemma}

\begin{corollary}\cite[Corollary~5.5]{DussauleLLT1}\label{derivativeparabolicGreenfinite}
For every parabolic subgroup $\mathcal{H}\in \Omega_0$ and every $r\in [0,R_{\mu}]$, we have
$I^{(1)}_{\mathcal{H}}(r)<+\infty$.
\end{corollary}

Finally, we will also use the following.

\begin{proposition}\cite[Proposition~5.8]{DussauleLLT1}\label{relationR_kG'(R)}
If $\mu$ is non-spectrally degenerate, then
$$\frac{d}{dr}_{|r=R_\mu}G(e,e|R_{\mu})=+\infty.$$
\end{proposition}

\subsection{Relative Ancona inequalities}\label{SectionAnconainequalities}
We consider a finitely generated group $\Gamma$, hyperbolic relative to a collection of subgroups $\Omega$.
Let $\Omega_0$ be a finite set of representatives of conjugacy classes.
We also consider a probability measure $\mu$ on $\Gamma$, whose finite support generates $\Gamma$ as a semigroup and denote by $R_{\mu}$ its spectral radius.
As soon as $\Gamma$ is non-elementary, it is non-amenable, so that $R_{\mu}>1$ according to Kesten's results \cite{Kesten}.

In the case where $\Gamma$ is hyperbolic, A.~Ancona proved that the Green function $G$ is roughly multiplicative along geodesics.
Precisely, there exists $C\geq 1$ such that if $x,y,z$ are elements along a geodesic in this order, then
\begin{equation}\label{weakAncona}
    \frac{1}{C}G(x,y)G(y,z)\leq G(x,z)\leq C G(x,y)G(y,z).
\end{equation}
See \cite{Ancona} for more details.
The proof also works for the Green function evaluated at $r$, when $r<R_{\mu}$.
Actually, the lower bound is always true, so that the content of Ancona inequalities really is
$$G(x,z)\leq C G(x,y)G(y,z).$$
In \cite{Gouezel1}, S.~Gouëzel proved that these inequalities still hold at $r=R_{\mu}$ when the measure is symmetric.
He also gave a strengthened version of them.
Namely, if $x,x',y,y'$ are four points such that geodesics $[x,y]$ from $x$ to $y$ and $[x',y']$ from $x'$ to $y'$ fellow travel for a time at least $n$, then for $r\in [1,R_{\mu}]$,
\begin{equation}\label{strongAncona}
    \left | \frac{G(x,y|r)G(x',y'|r)}{G(x',y|r)G(x,y'|r)} -1\right |\leq C\rho^n,
\end{equation}
where $C\geq 0$ and $0<\rho<1$ (see \cite[Theorem~2.9]{Gouezel1}).
This strengthened version of Ancona inequalities was proved at $r<R_{\mu}$ in \cite{IzumiNeshveyevOkayasu}.
It was also already proved at the spectral radius by S.~Gouëzel and S.~Lalley in the case of co-compact Fuchsian groups (see \cite[Theorem~4.6]{GouezelLalley}).
It allowed the authors to get Hölder regularity for Martin kernels on the Martin boundary and then to use thermodynamic formalism to deduce local limit theorems in hyperbolic groups in \cite{Gouezel1} and \cite{GouezelLalley}.

\medskip
Back to relatively hyperbolic groups.
Inequalities similar to~(\ref{weakAncona}) were obtained by I.~Gekhtman, V.~Gerasimov, L.~Potyagailo and W.-Y.~Yang in \cite{GGPY}.
Recall that if $\alpha$ is a geodesic in the Cayley graph $\Cay(\Gamma,S)$, a point on $\alpha$ is called a transition point if it is not deep in a parabolic subgroup.
It is proved in \cite{GGPY} that if $x,y,z$ are elements on a geodesic in this order and if $y$ is a transition point on this geodesic, then~(\ref{weakAncona}) holds.
The proof actually works for $G(\cdot,\cdot|r)$ whenever $r<R_{\mu}$.


As explained, because of our automatic structure, it will be more convenient for us to work with relative geodesics rather than transition points on actual geodesics.
We fix a generating set $S$ and consider the Cayley graph and the graph $\hat{\Gamma}$ associated with $S$.

\begin{definition}
Let $\Gamma$ be a relatively hyperbolic group and let $\mu$ be a probability measure on $\Gamma$ with Green function $G$ and specral radius $R_{\mu}$.
If $r\in [1,R_{\mu}]$, say that $\mu$ satisfies the weak $r$-relative Ancona inequalities if there exists $C\geq 0$ (which depends on $r$) such that
for every $x,y,z\in \Gamma$ such that their images in $\hat{\Gamma}$ lie in this order on a relative geodesic,
$$\frac{1}{C}G(x,y|r)G(y,z|r)\leq G(x,z|r)\leq C G(x,y|r)G(y,z|r).$$
Say that $\mu$ satisfies the weak relative Ancona inequalities up to the spectral radius if it satisfies the $r$-relative Ancona inequalities for every $r\in [1,R_{\mu}]$
with a constant $C$ not depending on $r$.
\end{definition}

Lemma~\ref{projectiontransitionpoints}, together with \cite[Corollary~5.10]{GerasimovPotyagailo} and \cite[Theorem~5.2]{GGPY} show that if the support of $\mu$ is finite and generates $\Gamma$ as a semigroup, then $\mu$ satisfies the weak 1-relative Ancona inequalities.
We will also need the following enhanced version of relative Ancona inequalities.

\begin{definition}
We say that two relative geodesic $[x,y]$ and $[x',y']$ $c$-fellow travel for a time $n$, for some $c\geq 0$, if there exist distinct points $\gamma_1,...,\gamma_n$ which are at distance in $\Cay(\Gamma,S)$ at most $c$ from points on $[x,y]$ and points on $[x',y']$.
\end{definition}

\begin{definition}
Let $\Gamma$ be a relatively hyperbolic group and let $\mu$ be a probability measure on $\Gamma$ with Green function $G$ and spectral radius $R_{\mu}$.
If $r\in [1,R_{\mu}]$, say that $\mu$ satisfies the strong $r$-relative Ancona inequalities if for every $c\geq 0$, there exist $C\geq 0$ and $0<\rho<1$ such that
if $x,x',y,y'$ are four points such that relative geodesics $[x,y]$ from $x$ to $y$ and $[x',y']$ from $x'$ to $y'$ $c$-fellow travel for a time at least $n$, then
$$\left | \frac{G(x,y|r)G(x',y'|r)}{G(x',y|r)G(x,y'|r)} -1\right |\leq C\rho^n.$$
Say that $\mu$ satisfies the strong relative Ancona inequalities up to the spectral radius if it satisfies the strong $r$-relative Ancona inequalities for every $r\in [1,R_{\mu}]$ with constants $C$ and $\rho$ not depending on $r$.
\end{definition}

\begin{remark}
If $\mu$ satisfies the strong or weak relative Ancona inequalities, then the reflected measure $\check{\mu}$, defined by $\check{\mu}(\gamma)=\mu(\gamma^{-1})$ also satisfies them.
Indeed, if $\check{G}$ is the Green function of the reflected measure, then $\check{G}(x,y)=G(y,x)$.
\end{remark}

The following is proved in \cite{DussauleGekhtman}.

\begin{theorem}
Let $\Gamma$ be a non-elementary relatively hyperbolic group and let $\mu$ be a symmetric probability measure on $\Gamma$ whose finite support generates $\Gamma$.
Then $\mu$ satisfies both the weak and strong relative Ancona inequalities up to the spectral radius.
\end{theorem}

Actually, these inequalities are stated in \cite{DussauleGekhtman} using the Floyd distance, which is a suitable rescaling of the distance in $\Cay(\Gamma,S)$.
However, \cite[Corollary~5.10]{GerasimovPotyagailo} relates the Floyd distance with transition points and Lemma~\ref{projectiontransitionpoints} relates transition points with points on a relative geodesic.
We deduce the above theorem combining these two results with \cite[Theorem~1.6]{DussauleGekhtman}.

\section{Thermodynamic formalism}\label{Sectionthermodynamicformalism}

\subsection{Transfer operators with countably many symbols}\label{Sectionthermodynamicformalismsetup}
We follow here the terminology of O.~Sarig and recall some facts proved in \cite{Sarig1} and \cite{Sarig2}.
We consider a countable set $\Sigma$ (the set of symbols), and a matrix $A=(a_{s,s'})_{s,s'\in \Sigma}$, with entries zeroes and ones (the transition matrix).
We then define
$$\Sigma_A^*=\{x=(x_1,...,x_n), x_i\in \Sigma, n\geq 0,\forall i, a_{x_i,x_{i+1}}=1\}$$
and
$$\partial \Sigma_A^*=\{x=(x_1,...,x_n,...), x_i\in \Sigma, \forall i, a_{x_i,x_{i+1}}=1\}.$$
Note that in the definition of $\Sigma_A^*$, $n$ can be 0, so that the empty sequence, that we denote by $\emptyset$ is in $\Sigma_A^*$.
We also define $$\overline{\Sigma}_A=\Sigma_A^*\cup \partial \Sigma_A^*.$$
If $s_1,...,s_k\in \Sigma$, we define the cylinder
$[s_1,...,s_k]$ as $\{x\in \overline{\Sigma}_A,x_1=s_1,...,x_k=s_k\}$.

Let $T:\overline{\Sigma}_A\rightarrow \overline{\Sigma}_A$ be given by
$T((x_1,...,x_n))=(x_2,...,x_n)$ if $(x_1,...,x_n)\in \Sigma_A^*$
and $T((x_1,...,x_n,...)=(x_2,...,x_n,...)$ if $(x_1,...,x_n,...)\in \partial \Sigma_A^*$.
We call $T$ the shift map and we call the pair $(\overline{\Sigma}_A,T)$ a Markov shift.

We say that the Markov shift is irreducible if for every $s,s'\in \Sigma$, there exists $N_{a,b}$ such that
there exists $x\in \overline{\Sigma}_A$ with $x_1=s$ and $x'\in \overline{\Sigma}_A$ with $x'_1=s'$ such that
$T^{N_{a,b}}x=x'$.
In other words, one can reach any cylinder $[s']$ from any cylinder $[s]$ with a finite number of iterations of the shift. 
We say it is topologically mixing if for every $s,s'\in \Sigma$, there exists $N_{a,b}$ such that
for every $n\geq N_{a,b}$,
there exists $x\in \overline{\Sigma}_A$ with $x_1=s$ and $x'\in \overline{\Sigma}_A$ with $x'_1=s'$ such that
$T^{n}x=x'$.

In \cite{Sarig1}, everything is stated only using $\partial \Sigma_A^*$.
However, up to considering a cemetery symbol $x_{\dag}$, we can consider finite sequences $(x_1,...,x_n)$ in $\Sigma_A^*$ as infinite ones, of the form $(x_1,...,x_n,x_{\dag},...,x_{\dag},...)$.
Thus, we can apply the terminology and results of~\cite{Sarig1} to $\overline{\Sigma}_A$.
Also, for technical reasons, it will be convenient to assume that the empty sequence is not a preimage of itself by the shift.
This can be done for example using a second cemetery symbol.

\medskip
We also define a metric on $\overline{\Sigma}_A$, setting
$d(x,y)=2^{-n}$, where $n$ is the first time that the two sequences $x$ and $y$ differ.

If $\varphi: \overline{\Sigma}_A\rightarrow \R$ is a function,
define
$$V_n(\varphi)=\sup \{|\varphi(x)-\varphi(y)|,x_1=y_1,...,x_n=y_n\}.$$
For $\rho\in (0,1)$, such a function $\varphi$ is called $\rho$-locally H\"older continuous if it satisfies
$$\exists C,\forall n\geq 1,V_n(\varphi)\leq C\rho^n.$$
It is called locally H\"older continuous if it is $\rho$-locally H\"older continuous for some $\rho$.
Notice that nothing is required for $V_0(\varphi)$ and in particular, $\varphi$ can be unbounded.
We can always change the metric $d$ on $\overline{\Sigma}_A$, defining $d_{\rho}(x,y)=\rho^{n}$, where $n$ is the first time that the two sequences $x$ and $y$ differ (and $0<\rho<1$).
A $\rho$-locally H\"older continuous function is then a locally Lipschitz function for this new metric.

\medskip
If $\varphi$ is locally H\"older continuous, we denote by $\varphi_n=\sum_{k=0}^{n-1}\varphi\circ T^{k}$ its $n$th Birkhoff sum.
We also define its transfer operator $\mathcal{L}_{\varphi}$ as
$$\mathcal{L}_{\varphi}f(x)=\sum_{Ty=x}\mathrm{e}^{\varphi(y)}f(y)$$
It acts on several spaces of functions.
We will be interested in some particular one, described below, on which the transfer operator has a spectral gap.
By definition, we have
$$\left (\mathcal{L}_{\varphi}^nf\right )(x)=\sum_{T^ny=x}\mathrm{e}^{\varphi_n(y)}f(y).$$
We also define, for $s\in\Sigma$,
$$Z_n(\varphi,s)=\sum_{\underset{x_1=s}{T^nx=x}}\mathrm{e}^{\varphi_n(x)}.$$
For every $s$, $\frac{1}{n}\log Z_n(\varphi,s)$ has a limit $P(\varphi,s)$.
If the Markov shift is irreducible, then it is independent of $s$ and we denote it by $P(\varphi)$.
Moreover, $P(\varphi,s)>-\infty$ and if $\|\mathcal{L}_{\varphi}1\|_{\infty}<+\infty$, then $P(\varphi,s)<+\infty$.
We refer to \cite[Theorem~1]{Sarig1} for a proof.
Independence of $s$ is proved under the assumption that the Markov shift is topologically mixing, although the proof only requires that it is irreducible.
We call $P(\varphi,s)$ the Gurevic pressure of $\varphi$ at $s$, or simply its pressure.

We say that $\varphi$ is positive recurrent if for every $s\in \Sigma$, there exist $M_{s}$ and $\lambda_{s}$ such that for every large enough $n$,
$\frac{Z_n(\varphi,s)}{\lambda_{s}^n}\in [M_{s}^{-1},M_{s}]$.
If it is the case, then one necessarily has $\lambda_{s}=P(\varphi,s)$.
The main result of \cite{Sarig1} is that positive recurrence is a necessary and sufficient condition for convergence of the iterates of the transfer operator $\mathcal{L}_{\varphi}^n$ (see \cite[Theorem~4]{Sarig1} for a precise statement).

If the set of symbols $\Sigma$ is finite, then every H\"older continuous function is positive recurrent.
Actually, we can say a little more in this case.
The convergence of $\mathcal{L}_{\varphi}^n$ is exponentially fast.
Precisely, if the Markov shift is topologically mixing, there exist $\lambda$, a positive function $h$ and a measure $\nu$ and constants $C\geq 0$ and $0<\theta<1$ satisfying, for all $\rho$-H\"older continuous function $f$ and all $n\in \N$,
$$\left \|\lambda^{-n}\mathcal{L}_{\varphi}^nf-h\int f d\nu\right \|\leq C\theta^n\|f\|.$$
This is the so-called Ruelle-Perron-Frobenius theorem.
Equivalently, $\lambda$ is a positive eigenvalue of the operator $\mathcal{L}_{\varphi}$ acting on the space of H\"older continuous functions and the remainder of the spectrum is contained in a disk of radius strictly smaller than $\lambda$.
In other words, $\mathcal{L}_{\varphi}$ acts on this space with a spectral gap

\medskip
When the set of symbols is countable, it can happen that the convergence is not exponentially fast (see \cite[Example~1]{Sarig1}).
However, there are sufficient conditions for this to hold, studied by J.~Aaronson, M.~Denker and M.~Urba\'{n}ski among others (see \cite{AaronsonDenkerUrbanski} and \cite{AaronsonDenker}, see also \cite{Gouezelthesis}).

\begin{definition}
Say that the Markov shift $(\overline{\Sigma}_A,T)$ has finitely many images if the set
$\{T[s],s \in \Sigma\}$ is finite.
Equivalently, there is only a finite number of different rows (and thus a finite number of different columns) in the matrix $A$.
\end{definition}


Fix $0<\rho<1$.
Let $\beta$ be the partition generated by the image sets, that is, $\beta$ is the $\sigma$-algebra generated by $\{T[s],s \in \Sigma\}$.
Then, define for a $\rho$-locally H\"older continuous function $f$,
$$D_{\rho,\beta} f=\sup_{b\in \beta}\sup_{x,y\in b}\frac{|f(x)-f(y)|}{d_{\rho}(x,y)},$$
where we recall that $d_{\rho}(x,y)=\rho^n$, where $n$ is the first time that the two sequences $x$ and $y$ differ.
Denote then $\|f\|_{\rho,\beta}=\|f\|_{\infty}+D_{\rho,\beta}f$ and define
$$\mathcal{B}_{\rho,\beta}=\{f, \|f\|_{\rho,\beta}<+\infty\}.$$
Then, $(\mathcal{B}_{\rho,\beta},\|\cdot \|_{\rho,\beta})$ is a Banach space.

Having finitely many images is a sufficient condition to have a spectral gap on $(\mathcal{B}_{\rho,\beta},\|\cdot \|_{\rho,\beta})$.
Indeed, R.~Mauldin and M.~Urba\'{n}ski introduced in \cite{MauldinUrbanski} the BIP property, which is automatically satisfied if the shift has finitely many images. Moreover, they proved that the BIP property is a sufficient condition for locally H\"older functions to have a Gibbs measure, whereas O.~Sarig proved in \cite{Sarig1} that having a Gibbs measure is a sufficient condition to be positive recurrent and to have a spectral gap.
In particular, we have the following two results.

\begin{proposition}\label{finitelymanyimagesimpliespositiverecurrent}
Let $(\overline{\Sigma}_A,T)$ be a topologically mixing countable Markov shift having finitely many images.
Let $\varphi$ be a locally H\"older continuous function with finite pressure $P(\varphi)$.
Then, $\varphi$ is positive recurrent.
\end{proposition}

\begin{proof}
Since $\varphi$ is locally H\"older, the sum
$$\sum_{n\geq 1} V_n(\varphi)=\sum_{n\geq 1}\sup \{|\varphi(x)-\varphi(y)|,x_1=y_1,...,x_n=y_n\}\leq C\sum_{n\geq 1}\rho^n$$
is finite.
Thus, \cite[Theorem~1]{Sarig2} shows that $\varphi$ has a Gibbs measure.
Consequently, \cite[Theorem~8]{Sarig1} shows that $\varphi$ is positive recurrent.
\end{proof}

\begin{theorem}\label{Theorem5Sarig}\cite[Corollary~3]{Sarig2},\cite[Theorem~4]{Sarig1}
Let $(\overline{\Sigma}_A,T)$ be a topologically mixing countable Markov shift having finitely many images.
Let $\varphi$ be a locally H\"older continuous function with finite pressure $P(\varphi)$.
Then there exist a $\sigma$-finite measure $\nu$ and a function $h$ bounded away from 0 and infinity such that $\mathcal{L}_{\varphi}^*\nu=\mathrm{e}^{P(\varphi)}\nu$ and $\mathcal{L}_{\varphi}h=\mathrm{e}^{P(\varphi)}h$.
There also exist $C\geq 0$ and $0<\theta<1$ such that for every $f\in \mathcal{B}_{\rho,\beta}$,
$$\left \|\mathrm{e}^{-nP(\varphi)}\mathcal{L}_{\varphi}^nf-h\int f d\nu\right \|_{\rho,\beta}\leq C\theta^n\|f\|_{\rho,\beta}.$$
Moreover, $\nu$ is supported on $\partial \Sigma_A^*$ and both measures $\nu$ and $m$ defined by $dm=hd\nu$ are ergodic.
\end{theorem}

The fact that $\nu$ is ergodic is not stated in Corollary~3 but in Corollary~2 of \cite{Sarig2}.
Ergodicity of $m$ follows (see the remarks after \cite[Theorem~4]{Sarig1}).
Finally, the fact that $h$ is bounded away from 0 and infinity is deduced from the fact that the shift has finitely many images, see \cite[Proposition~2]{Sarig1}.

Actually, we will never really use the $\|\cdot \|_{\rho,\beta}$ norm and all our bounds in the following will be on the $\rho$-H\"older norm, that is, we will both bound $\|\cdot \|_{\infty}$ and $D_{\rho}$, which is defined by
$$D_{\rho}f=\sup_{x,y\in \overline{\Sigma}_A}\frac{|f(x)-f(y)|}{d_{\rho}(x,y)}.$$
Obviously, a bound on $D_{\rho}$ is stronger than a bound $D_{\rho,\beta}$.
Moreover, when bounding $D_{\rho}f$, we will have to bound $\frac{|f(x)-f(y)|}{d_{\rho}(x,y)}$.
We will always assume that $x$ and $y$ start with the same element, otherwise $d(x,y)=\rho$ and one can thus bound $\frac{|f(x)-f(y)|}{d_{\rho}(x,y)}$ by $2\rho^{-1}\|f\|_{\infty}$.

\medskip
One issue we will have to deal with, when applying these results to random walks on relatively hyperbolic groups in the next subsection,
is that our Markov shift will not be topologically mixing.
It will not even be irreducible (but will have only finitely many recurrent classes).
This issue was already addressed in \cite{Gouezel1} for Markov shift with finitely many symbols, as we now explain.

In the following, we consider a countable Markov shift $(\overline{\Sigma}_A,T)$ with set of symbols $\Sigma$ and transition matrix $A$ and we assume that it has finitely many images.
If the Markov shift is irreducible, but not topologically mixing, then there is a minimal period $p>1$ such that for any symbol $s \in \Sigma$,
if $T^{-n}[s]\cap [s]\neq \emptyset$, then $n=pk$ for some $k\geq 0$.
Then, one can decompose the set of symbols as a finite union
$\Sigma=\Sigma_A^{(1)}\sqcup \Sigma_A^{(2)}\sqcup \cdots \sqcup \Sigma_A^{(p)}$,
such that for $i\in \Z/p\Z$, if $a_{s,s'}=1$ and $s\in \Sigma^{(i)}$, then
$s'\in \Sigma^{(i+1)}$.
We call such a decomposition a cyclic decomposition.
We denote by $\overline{\Sigma}_A^{(i)}$ the subset of $\overline{\Sigma}_A$ of sequences that begin with an element of $\Sigma_A^{(i)}$, so that the shift map $T$ maps $\overline{\Sigma}_A^{(i)}$ to $\overline{\Sigma}_A^{(i+1)}$.
Moreover, in this case, $T^p$ acts on $\overline{\Sigma}_A^{(i)}$ and the induced Markov shift is topologically mixing.
Using this decomposition together with Theorem~\ref{Theorem5Sarig}, we get that if $\varphi$ is locally H\"older continuous function with finite pressure $P(\varphi)$ and if the Markov shift has finitely many images, then there are positive functions $h^{(i)}$ on $\overline{\Sigma}_A^{(i)}$ and probability measures $\nu^{(i)}$ on $\overline{\Sigma}_A^{(i)}$ with $\int h^{(i)}d\nu^{(i)}=1$ such that for $f\in \mathcal{B}_{\rho,\beta}$,
$$\left \|\mathrm{e}^{-nP(\varphi)}\mathcal{L}_{\varphi}^nf-\sum_{i=1}^ph^{(i)}\int f d\nu^{((i-n) \text{ mod } p)}\right \|_{\rho,\beta}\leq C\theta^n\|f\|_{\rho,\beta}.$$

Assume that the Markov shift is not irreducible.
Then, since it has finitely many images, one can first decompose $\Sigma$ as
$\Sigma=\Sigma_{A,0}\sqcup \Sigma_{A,1}\sqcup  \cdots \Sigma_{A,q}$,
such that if a path starts at $s\in \Sigma_{A,0}$, then it never reaches $s$ again and for $s,s'\in \Sigma$, one can reach $s'$ starting at $s$ and conversely if and only if $s$ and $s'$ are in the same subset $\Sigma_{A,j}$, $j\geq 1$.
More formally, the decomposition of $\Sigma$ satisfies
the following properties.
\begin{itemize}
    \item If $s\in \Sigma_{A,0}$, then for all $n\geq 1$, $T^{-n}[s]\cap [s]=\emptyset$.
    \item If there exist $n,n'$ such that $T^{-n}[s]\cap [s']\neq \emptyset$ and $T^{-n'}[s']\cap [s]\neq \emptyset$, then there exists $j\geq 1$ such that $s,s'\in \Sigma_{A,j}$.
    \item Conversely, if $s,s'$ lie in the same $\Sigma_{A,j}$, $j\geq 1$, then there exist $n,n'$ such that $T^{-n}[s]\cap [s']\neq \emptyset$ and $T^{-n'}[s']\cap [s]\neq \emptyset$.
\end{itemize}
We call $\Sigma_{A,0}$ the transient component of $\Sigma$ and the sets $\Sigma_{A,j}$, $j\geq 1$, the biconnected components of $\Sigma$.
All the non-trivial dynamical behaviour of the Markov shift happens in the biconnected components.
We denote by $\overline{\Sigma}_{A,j}$ the subset of $\overline{\Sigma}_A$ of sequences $x$ that stay in $\Sigma_{A,j}$, that is, for every $n$, $x_n\in \Sigma_{A,j}$.
We similarly call the sets $\overline{\Sigma}_{A,j},j\geq 1$ the biconnected components of $\overline{\Sigma}_A$.

Then, $\overline{\Sigma}_{A,j}$ is stable under the shift map $T$ and we can apply the above discussion to $\overline{\Sigma}_{A,j}$.
If $\varphi$ is a locally H\"older continuous function on $\overline{\Sigma}_A$, denote by $\varphi_j$ its restriction to the component $\overline{\Sigma}_{A,j}$, with associated transfer operator $\mathcal{L}_{\varphi_j}$.
Denote the pressure of $\varphi_j$ by $P_j(\varphi)$.
Then, $\mathcal{L}_{\varphi_j}$ has a spectral gap and $P_j(\varphi)$ is its dominant eigenvalue.

Let $P(\varphi)$ be the maximum of all the $P_j(\varphi)$ and call a component $\overline{\Sigma}_{A,j}$ maximal if $P_j(\varphi)=P(\varphi)$.

\begin{definition}
We say that $\varphi$ is semisimple if one cannot reach a maximal component from another.
That is, for every two maximal components $\overline{\Sigma}_{A,j}$, $\overline{\Sigma}_{A,j'}$, $j\neq j'$, for any two symbols $s\in \Sigma_{A,j}$, $s'\in \Sigma_{A,j'}$, for any $n\geq 1$, $T^{-n}[s]\cap [s']=\emptyset$.
\end{definition}

Elaborating on ideas of D.~Calegari and K.~Fujiwara from \cite{CalegariFujiwara}, S.~Gou\"ezel proved in \cite{Gouezel1} a spectral gap theorem for the transfer operator of a semisimple H\"older continuous function, when the set of symbols is finite.
His proof works for a countable set of symbols, if the Markov shift has finitely many images and the H\"older continuous function is positive recurrent, since it is based on a spectral decomposition over the sets $\overline{\Sigma}_{A,j}$, on which one applies the Ruelle-Perron-Frobenius theorem (that we replace here with Theorem~\ref{Theorem5Sarig}).
Thus, combining Theorem~\ref{Theorem5Sarig} and the proof of \cite[Theorem~3.8]{Gouezel1}, we get the following.

\begin{theorem}\label{GouezelSarig}
Let $(\overline{\Sigma}_A,T)$ be a countable Markov shift with finitely many images.
Let $\varphi$ be a locally H\"older continuous function with finite maximal pressure $P(\varphi)$.
Assume that $\varphi$ is semisimple.
Denote by $\overline{\Sigma}_{A,1}$,...$\overline{\Sigma}_{A,k}$ the maximal components, with corresponding period $p_1,...,p_k$ and consider a cyclic decomposition
$$\Sigma_{A,j}=\Sigma_{A,j}^{(1)}\sqcup \cdots \Sigma_{A,j}^{(p_j)}.$$
Then, there exist functions $h_j^{(i)}$ and probability measures $\nu_j^{(i)}$ with $\int h_{j}^{(i)}d\nu_j^{(i)}=1$ and such that
$\mathcal{L}_{\varphi}^*\nu_{j}^{(i)}=\nu_j^{(i-1) \text{ mod } p_j}$ and $\mathcal{L}_{\varphi}h_{j}^{(i)}=h_j^{(i-1) \text{ mod } p_j}$.
Moreover, for $f\in \mathcal{B}_{\rho,\beta}$,
$$\left \|\mathrm{e}^{-nP(\varphi)}\mathcal{L}_{\varphi}^nf-\sum_{j=1}^k\sum_{i=1}^{p_j}h_j^{(i)}\int f d\nu_j^{((i-n) \text{ mod } p_j)}\right \|_{\rho,\beta}\leq C\theta^n\|f\|_{\rho,\beta},$$
for some $C\geq 0$ and $0<\theta<1$.
Finally, the functions $h_j^{(i)}$ are bounded away from 0 and infinity on the support of $\nu_j^{(i)}$.
\end{theorem}

We also get the following result, again proved in \cite{Gouezel1} for finite sets of symbols, which applies in our situation (see \cite[Lemma~3.7]{Gouezel1}).
\begin{lemma}\label{lemmaGouezelSarig}
Let $(\overline{\Sigma}_A,T)$ be a countable Markov shift with finitely many images.
Let $\varphi$ be a locally H\"older continuous function with finite maximal pressure $P(\varphi)$.
Let $s\in \Sigma$ and assume that there is a path starting with $s$ that visits $k$ maximal components.
Then, for any non-negative function $f$ with $f\geq 1$ on the set of paths starting with $s$, one has
$$\mathcal{L}^n_{\varphi}f(\emptyset)\geq Cn^{k-1}\mathrm{e}^{nP(\varphi)},$$
where we recall that $\emptyset$ is the empty sequence in $\Sigma_A^*$.
In particular, for $k=2$, if $\varphi$ is not semisimple, then
$$\mathcal{L}^n_{\varphi}1(\emptyset)\geq Cn \mathrm{e}^{nP(\varphi)}.$$
\end{lemma}


\subsection{Perturbation of the pressure}
The following perturbation result is also proved in \cite{Gouezel1} for finite sets of symbols (see precisely \cite[Proposition~3.10]{Gouezel1}). Its proof remains valid for countable shifts with finitely many images.
Denote by $\vertiii{\cdot}_{\rho,\beta}$ the operator norm for operators acting on $(\mathcal{B}_{\rho,\beta},\|\cdot \|_{\rho,\beta})$.

\begin{proposition}\label{propperturbation}
Let $(\overline{\Sigma}_A,T)$ be a countable Markov shift with finitely many images.
Let $\varphi$ and $\psi$ be locally H\"older continuous functions with finite maximal pressure.
Assume that $\varphi$ is semisimple and denote by $h_j^{(i)}$ and $\nu_j^{(i)}$ the functions and measures given by Theorem~\ref{GouezelSarig}.
If $\vertiii{\mathcal{L}_{\varphi}-\mathcal{L}_{\psi}}_{\rho,\beta}$ is small enough, then there exists numbers $\tilde{P}_j(\psi)$ and there exist eigenfunctions $\tilde{h}_{j}^{(i)}$ and eigenmeasures $\tilde{\nu}_j^{(i)}$ of $\mathcal{L}_{\psi}$ associated with the eigenvalue $\mathrm{e}^{\tilde{P}_j(\psi)}$ such that
$$\left \|\mathrm{e}^{-nP(\varphi)}\mathcal{L}_{\psi}^nf-\sum_{j=1}^k\mathrm{e}^{n(\tilde{P}_j(\psi)-P(\varphi))}\sum_{i=1}^{p_j}\tilde{h}_j^{(i)}\int f d\tilde{\nu}_j^{((i-n) \text{ mod } p_j)}\right \|_{\rho,\beta}\leq C\theta^n\|f\|_{\rho,\beta},$$
for some $C\geq 0$ and $0<\theta<1$.
The functions $\tilde{h}_{j}^{(i)}$ and the measures $\tilde{\nu}_j^{(i)}$ have the same support as $h_j^{(i)}$ and $\nu_j^{(i)}$ respectively.
Moreover, the functions
$\mathcal{L}_{\psi}\mapsto \tilde{P}_j(\psi)$, $\mathcal{L}_{\psi}\mapsto \tilde{h}_j^{(i)}$ and $\mathcal{L}_{\psi}\mapsto \tilde{\nu}_{j}^{(i)}$ are analytic and
$$\tilde{P}_j(\psi)=P(\varphi)+\int \psi dm_j + O\left (\vertiii{\mathcal{L}_{\varphi}-\mathcal{L}_{\psi}}_{\rho,\beta}^2\right ),$$
where $dm_j=\frac{1}{p_j}\sum_{i=1}^{p_j}h_{j}^{(i)}d\nu_j^{(i)}$.
\end{proposition}

Estimating $\vertiii{\mathcal{L}_{\varphi}-\mathcal{L}_{\psi}}_{\rho,\beta}^2$ will be very difficult in this paper, so we need finer results.
First, we state a theorem adapted from G.~Keller and C.~Liverani that will allow us to guaranty the existence of $\tilde{P}_j(\psi)$, $\tilde{h}_{j}^{(i)}$ and $\tilde{\nu}_j^{(i)}$ as above under weaker estimates on $\mathcal{L}_{\varphi}-\mathcal{L}_{\psi}$.
Then we prove a second result that will yield an asymptotic of $\tilde{P}_j(\psi)$ not involving $\vertiii{\mathcal{L}_{\varphi}-\mathcal{L}_{\psi}}_{\rho,\beta}^2$.

\medskip
Consider a Banach $(V,\|\cdot\|)$ endowed with a norm $|\cdot |_w$, satisfying
$|\cdot |_w\leq C \|\cdot \|$ for some uniform $C$.
Letting $\mathcal{L}:V\to V$ be a linear operator, let
$$\vertiii{\mathcal{L}}=\sup \left \{\|\mathcal{L}v\|,\|v\|\leq 1\right \}$$
denote the operator norm of $\mathcal{L}$ associated with $\|\cdot \|$ and let
$$\vertiii{\mathcal{L}}_{s\to w}=\sup \left \{|\mathcal{L}v|_w,\|v\|\leq 1\right \}$$
be the operator norm of $\mathcal{L}:(V,\|\cdot \|)\to (V,|\cdot |_w)$.

\begin{theorem}\label{KellerLiverani}
Consider a family of bounded operators $\mathcal{L}_r:(V,\|\cdot \|)\to (V,\|\cdot \|)$, with $r$ varying in $(0,R]$.
Assume there exist $0<\sigma<M$ and $C\geq 0$ and there exists a function $\tau(r)$ converging to 0 as $r$ tends to $R$ such that the following holds.
\begin{enumerate}[(i)]
    \item For every $n$, for every $v\in V$, $|\mathcal{L}_R^nv|_w\leq CM^n|v|_w$.
    \item For every $r\leq R$, for every $n$, for every $v\in V$,
    $\|\mathcal{L}_r^n v\|\leq C\sigma^n\|v\|+CM^n|v|_w$.
    \item For every $r\leq R$, $\vertiii{\mathcal{L}_r-\mathcal{L}_R}_{s\to w}\leq \tau(r)$.
\end{enumerate}
For fixed $\rho>0$ and $\rho'>0$ let
$$A_{\rho,\rho'}=\{z\in \C,|z|\geq \sigma +\rho,d(z,\spec(\mathcal{L}_R))\geq \rho'\}.$$
Then, for any $\rho,\rho'>0$ there exist $\beta_0<1$ and $K_0\geq 0$ and there exists $r_0$ such that for every $\beta\leq \beta_0$, for every $r\in [r_0,R]$ for every z $z\in A_{\rho,\rho'}$,
\begin{enumerate}[(a)]
    \item the operator $zI-\mathcal{L}_r:(V,\|\cdot\|)\to (V,\|\cdot\|)$ is invertible,
    \item the operator norm $\vertiii{(zI-\mathcal{L}_r)^{-1}}$ is bounded independently of $r$,
    \item the norm $\vertiii{\cdot}_{s\to w}$ satisfies $\vertiii{(zI-\mathcal{L}_r)^{-1}-(zI-\mathcal{L}_R)^{-1}}_{s\to w}\leq K_0 \tau(r)^\beta$.
\end{enumerate}
Moreover, $\beta_0$ only depends on $\rho$ and can be explicitly computed
whenever $\sigma+\rho\leq M$.
Indeed, one can then choose
$$\beta_0=\frac{\log\left (\frac{\sigma+\rho}{\sigma}\right )}{\log \left (\frac{M}{\sigma}\right )}.$$
In particular, $\beta_0$ converges to 0 as $\rho$ tends to 0 and converges to 1 as $\rho$ tends to $M-\sigma$.
\end{theorem}

For a proof, we refer to \cite[A.3]{Baladi}.
Note that it is asked there that for every $r$, $|\mathcal{L}_r^nv|_w\leq CM^n|v|_w$, whereas our condition~(i) only requires that this holds for $r=R$.
However, the proof in \cite[A.3]{Baladi} only uses this inequality for $r=R$.

Let us apply this to transfer operators.
Consider a countable shift with finitely many images $(X_A,T)$ and a family of locally H\"older functions $f_r$, for $r\in(0,R]$.
Let $\mathcal{L}_r=\mathcal{L}_{f_r}$ be the associated transfer operator
Assume that for every $r$, the maximal pressure $P_r$ of $f_r$ is finite and that $f_r$ is semisimple.
Let $h_{j}^{(i)}$ and $\nu_j^{(i)}$ be the functions and measures given by Theorem~\ref{GouezelSarig}, associated with $\mathcal{L}_{R}$.
Define the measure $m_j$ as $dm_j=\frac{1}{p_j}\sum_{i=1}^{p_j}h_{j}^{(i)}d\nu_j^{(i)}$.

Let $m=\sum m_j$.
Consider the Banach space $(V=H_{\rho,\beta},\|\cdot \|=\|\cdot \|_{\rho,\beta})$ on $V=H_{\rho,\beta}$, endowed with the norm $|\cdot|_w=\|\cdot\|_{L^1(m)}$.
Since $m$ is finite, we have $|\cdot|_w\leq C \|\cdot\|$.
We deduce from Theorem~\ref{KellerLiverani} the following.

\begin{theorem}\label{KellerLiveranitransfert}
With the same notations as above,
assume that there exists $\sigma$ such that $0<\sigma<\mathrm{e}^{P_R}$ and that there exist $C\geq 0$ and a function $\tau(r)$ converging to 0 as $r$ tends to $R$ such that the following holds.
\begin{enumerate}
    \item[($\alpha$)] For every $r\leq R$, for every $n$, for every $v\in V$,
    $$\|\mathcal{L}_r^n v\|\leq C\sigma^n\|v\|+C\mathrm{e}^{nP_R}|v|_w.$$
    \item[($\beta$)] For every $r\leq R$, $\vertiii{\mathcal{L}_r-\mathcal{L}_R}_{s\to w}\leq \tau(r)$.
\end{enumerate}
Then, for every $r$ which is close enough to $R$, there exist numbers $\tilde{P}_j(r)$ and eigenfunctions $\tilde{h}_{j,r}^{(i)}$ and eigenmeasures $\tilde{\nu}_{j,r}^{(i)}$ of $\mathcal{L}_r$ associated with the eigenvalue $\mathrm{e}^{\tilde{P}_j(r)}$ such that
$$\left \|\mathrm{e}^{-nP_R}\mathcal{L}_{r}^ng-\sum_{j=1}^k\mathrm{e}^{n(\tilde{P}_j(r)-P_R)}\sum_{i=1}^{p_j}\tilde{h}_{j,r}^{(i)}\int g d\tilde{\nu}_{j,r}^{((i-n) \text{ mod } p_j)}\right \|_{\rho,\beta}\leq C\theta^n\|g\|_{\rho,\beta}.$$
The functions $\tilde{h}_{j}^{(i)}$ and the measures $\tilde{\nu}_j^{(i)}$ have the same support as $h_j^{(i)}$ and $\nu_j^{(i)}$ respectively.
Moreover, $\left \|\tilde{h}_{j,r}^{(i)}\right \|$ is uniformly bounded.
Finally, $\left |\tilde{h}_{j,r}^{(i)}-h_j^{(i)}\right |_w$ converges to 0 as $r$ tends to $R$
and $\tilde{\nu}_{j,r}^{(i)}$ weakly converges to $\nu_j^{(i)}$ as $r$ tends to $R$.
\end{theorem}

\begin{proof}
Since $\mathcal{L}_R$ has a spectral gap according to Theorem~\ref{GouezelSarig}, there exists $\sigma_0<\mathrm{e}^{P_R}$ such that the spectrum of $\mathcal{L}_R$ outside of the disc of radius $\sigma_0$ exactly consists of the eigenvalue $\mathrm{e}^{P_R}$, with eigenfunctions $h_j^{(i)}$ and eigenmeasures $\nu_j^{(i)}$.
The result is then a consequence of Theorem~\ref{KellerLiverani}, choosing $\rho$ such that $\sigma_0<\sigma+\rho<\mathrm{e}^{P_R}$.
Indeed, condition~(i) there is satisfied with $M=\mathrm{e}^{P_R}$ since $\nu_j^{(i)}$ is an eigenmeasure of $\mathcal{L}_R$ associated with $\mathrm{e}^{P_R}$.
Also, conditions~(ii) and~(iii) are direct consequences of assumptions~($\alpha$) and~($\beta$).
\end{proof}

Note that $\int \tilde{h}_{j,r}^{(i)}d\nu_j^{(i)}\neq 0$ for $r$ close enough to $R$, since $\int h_j^{(i)}d\nu_j^{(i)}=1$ and $\left |\tilde{h}_{j,r}^{(i)}-h_j^{(i)}\right |_w$ converges to 0.
One can thus normalize $\tilde{h}_{j,r}$ declaring $\int \tilde{h}_{j,r}^{(i)}d\nu_j^{(i)}=1$.
We will make this assumption in the following.
We still have that $\left \|\tilde{h}_{j,r}^{(i)}\right \|$ is uniformly bounded and that $\left |\tilde{h}_{j,r}^{(i)}-h_j^{(i)}\right |_w$ converges to 0.
The following result will allow us to get a precise asymptotic of $\tilde{P}_{j,r}-P_R$ in the next section.


\begin{proposition}\label{theoremeperturbation}
Under the assumptions of Theorem~\ref{KellerLiveranitransfert},
$$\mathrm{e}^{\tilde{P}_j(r)-P_R}-1= \int \left (\mathrm{e}^{f_r-f_R}-1\right )dm_j+ \int \left (\mathrm{e}^{f_r-f_R}-1\right )\frac{1}{p_j}\left (\sum_{i=0}^{p_j-1}h_j^{(i)}-\tilde{h}_{j,r}^{(i)}\right)d\nu_j^{(i)}.$$
\end{proposition}

\begin{proof}
Since $\tilde{h}_{j,r}^{(i)}$ are eigenfunctions of $\mathcal{L}_r$ and $\tilde{h}_{j,r}$ is normalized, we have
$$\mathrm{e}^{\tilde{P}_j(r)}=\int \mathcal{L}_r \tilde{h}_{j,r}^{(i)}d\nu_j^{(i)}.$$
Consequently,
$$\mathrm{e}^{\tilde{P}_j(r)}-\mathrm{e}^{P_R}=\int \left ( \mathcal{L}_r \tilde{h}_{j,r}^{(i)}-\mathcal{L}_Rh_j^{(i)}\right )d\nu_j^{(i)}.$$
Note that for any function $g$, $\mathcal{L}_rg=\mathcal{L}_R\left (\mathrm{e}^{f_r-f_R}g\right )$.
In particular,
$$\mathrm{e}^{\tilde{P}_j(r)}-\mathrm{e}^{P_R}=\int  \mathcal{L}_R\left ( \mathrm{e}^{f_r-f_R}\tilde{h}_{j,r}^{(i)}-h_j^{(i)}\right )d\nu_j^{(i)}.$$
Using that $d\nu_j^{(i)}$ is an eigenmeasure of $\mathcal{L}_R$ associated with the eigenvalue $\mathrm{e}^{P_R}$, we get
\begin{align*}
    \mathrm{e}^{\tilde{P}_j(r)-P_R}-1&=\int  \left ( \mathrm{e}^{f_r-f_R}\tilde{h}_{j,r}^{(i)}-h_j^{(i)}\right )d\nu_j^{(i)}\\
    &=\int  \left ( \mathrm{e}^{f_r-f_R}-1\right )\left (\tilde{h}_{j,r}^{(i)}-h_j^{(i)}\right )d\nu_j^{(i)}\\
&\hspace{2cm}+\int \tilde{h}_{j,r}^{(i)}d\nu_j^{(i)}-\int \mathrm{e}^{f_r-f_R}h_j^{(i)}d\nu_j^{(i)} .
\end{align*}
Since $\int h_j^{(i)}d\nu_j^{(i)}=\int \tilde{h}_{j,r}^{(i)}d\nu_j^{(i)}=1$, we thus get
$$\mathrm{e}^{\tilde{P}_j(r)-P_R}-1= \int  \left ( \mathrm{e}^{f_r-f_R}-1\right )\left (\tilde{h}_{j,r}^{(i)}-h_j^{(i)}\right )d\nu_j^{(i)}+\int \left (\mathrm{e}^{f_r-f_R}-1\right )h_j^{(i)}d\nu_j^{(i)}.$$
This holds for every $i$, which concludes the proof summing over $i$ and then dividing by $p_j$.
\end{proof}

\section{Asymptotic of the first derivative of the Green function}\label{Sectionapplicationthermodynamicformalism}
In this section, we assume that $\Gamma$ is hyperbolic relative to $\Omega$
and choose a system of representatives of conjugacy classes $\Omega_0=\{\mathcal{H}_1,...,\mathcal{H}_N\}$ of elements of $\Omega$.
We consider a probability measure $\mu$ on $\Gamma$ that satisfies weak and strong relative Ancona inequalities up to the spectral radius
and we denote by $R_{\mu}$ this spectral radius.
Note that we do not need to assume that the measure $\mu$ is symmetric but only that relative Ancona inequalities are satisfied.

We assume that $\mu$ is not spectrally degenerate.
According to Proposition~\ref{relationR_kG'(R)}, we have $\frac{d}{dr}_{|r=R_\mu}G(e,e|r)=+\infty$, or equivalently $I^{(1)}(R_\mu)=+\infty$ by Lemma~\ref{lemmafirstderivative}.
Also,
Proposition~\ref{roughequadiff} shows that
$$I^{(2)}(r)\asymp \left (I^{(1)}(r)\right )^3.$$
Our goal in the two following sections is to get a more precise statement, transforming $\asymp$ into $\sim$, when $r\to R_{\mu}$.
Precisely, we prove the following.

\begin{theorem}\label{preciseequadiff}
Under these assumptions, there exists $\xi>0$ such that
$$I^{(2)}(r)\underset{r\to R_{\mu}}{\sim} \xi\left (I^{(1)}(r)\right )^3.$$
\end{theorem}

To do so, we use thermodynamic formalism, adapting \cite{Gouezel1} and \cite{GouezelLalley}.

\subsection{Transfer operator for the Green function}\label{SectiontransferoperatorGreen}
We choose a generating set $S$ of $\Gamma$ as in Theorem~\ref{thmcodingrelhypgroups}, so that $\Gamma$ is automatic relative to $\Omega_0$ and $S$, where $\Omega_0$ is a finite set of representatives of conjugacy classes of the parabolic subgroups 
Let $\mathcal{G}=(V,E,v_*)$ be a graph and $\phi:E\rightarrow S\cup \bigcup_{\mathcal{H}\in \Omega_0}\mathcal{H}$ be a labelling map as in the definition of a relative automatic structure.

The set of vertices $V$ is finite.
Moreover, if $\sigma \in \Sigma_0= S\cup \bigcup_{\mathcal{H}\in \Omega_0}\mathcal{H}$ and if $v\in V$, there is at most one edge that leaves $v$ and that is labelled with $\sigma$.
Thus, the set of edges $E$ is countable.
Set $\Sigma=E$
and consider the transition matrix $A=(a_{s,s'})_{s,s'\in \Sigma}$, defined by $a_{s,s'}=1$ if the edges $s$ and $s'$ are adjacent in $\mathcal{G}$ and $a_{s,s'}=0$ otherwise.
We then define $\Sigma_A^*$, $\partial \Sigma_A$ and $\overline{\Sigma}_A$ as above.
According to the definition of a relative automatic structure, elements of $\overline{\Sigma}_A$ represent relative geodesics and relative geodesic rays.

\medskip
We decompose $\Sigma_0= S\cup \bigcup_{\mathcal{H}\in \Omega_0}\mathcal{H}$ as follows.
The sets $\mathcal{H}_j\cap \mathcal{H}_k$ are finite if $j\neq k$ (see e.g.\ \cite[Lemma~4.7]{DrutuSapir}).
We can thus consider $\mathcal{H}'_k=\mathcal{H}_k\setminus \cup_{j\neq k}\mathcal{H}_j$ and $\mathcal{H}'_0=\Sigma_0\setminus \cup_k \mathcal{H}'_k$.
Then, $\mathcal{H}'_0$ still is finite and the sets $\mathcal{H}'_k$ are disjoint.
By analogy with free factors in a free product, we introduce the following terminology.
\begin{definition}\label{deffactors}
We call the sets $\mathcal{H}'_k$ the factors of the relatively automatic structure
\end{definition}

Paths of length $n$ in $\mathcal{G}$ beginning at $v_*$ are in bijection with the relative sphere $\hat{S}_n$.
Moreover, infinite paths in $\mathcal{G}$ starting at $v_*$ give relative geodesic rays starting at $e$.
Denote by $E_*\subset E$ the set of edges that starts at $v_*$
The labelling map $\phi$ can be extended to infinite paths.
When restricted to infinite words starting in $E_*$, it gives a surjective map from paths beginning at $v_*$ to the Gromov boundary of the graph $\hat{\Gamma}$, which is by definition the set of conical limit points of $\Gamma$, included in the Bowditch boundary.
Restricting the distance $d_\rho(x,y)=\rho^{-n}$ to $E_*$, this induced map is continuous, endowing the Bowditch boundary with the usual topology.
A formal way of restricting our attention to elements of the group and to conical limit points is to consider the function $1_{E_*}$ on $\overline{\Sigma}_A$ which takes value 1 on sequences in $\overline{\Sigma}_A$ beginning with an edge in $E_*$ and that takes value 0 elsewhere.
This function $1_{E_*}$ is locally H\"older continuous.

We have the following, which proves that every locally H\"older continuous function with finite pressure is positive recurrent, according to Proposition~\ref{finitelymanyimagesimpliespositiverecurrent}.

\begin{lemma}
The Markov shift $(\overline{\Sigma}_A,T)$ has finitely many images.
\end{lemma}

\begin{proof}
If an edge $s$ in $\Sigma$ ends at some vertex $v$ in $\mathcal{G}$, then the only edges $s'$ such that $a_{s,s'}=1$
are those that start at $v$.
Thus, if two edges end at the same vertex $v$, they have the same row in the matrix $A$.
The lemma follows, since there is only a finite number of vertices.
\end{proof}

Recall that $R_{\mu}$ is the spectral radius of the $\mu$-random walk on $\Gamma$.
Also recall that for $r\in [0,R_{\mu}]$, we write
$H(e,\gamma|r)=G(e,\gamma|r)G(\gamma,e|r)$.
For $r\in [1,R_{\mu}]$, we define the function $\varphi_r$ on $\Sigma_A^*$ by $\varphi_r(\emptyset)=1$ and
$$\varphi_r(x=x_1,...,x_n)=\log \left (\frac{H(e,\phi(x)|r)}{H(e,\phi(T x)|r)}\right ) = \log \left (\frac{H(e,\phi(x_1...x_n)|r)}{H(e,\phi(x_2...x_n)|r)}\right ).$$
Using equivariance of the Green function, we also have
$$\varphi_r(x=x_1,...,x_n)=\log \left (\frac{H(e,\phi(x_1...x_n)|r)}{H(\phi(x_1),\phi(x_1...x_n)|r)}\right ).$$

\begin{lemma}\label{GreenlocallyHolder}
For every $r\in [1,R_{\mu}]$, the function $\varphi_r$ can be extended to $\overline{\Sigma}_A$.
It is then locally H\"older continuous on $\overline{\Sigma}_A$.
\end{lemma}

\begin{proof}
Let $n\geq 1$ and let $x,y\in \Sigma_A^*$ be such that $x_1=y_1$,...,$x_n=y_n$.
Then, $\phi(x_1)=\phi(y_1)$,...,$\phi(x_n)=\phi(y_n)$ in $\Gamma$.
This means that relative geodesics from $e$ to $\phi(x)$ and from $\phi(x_1)=\phi(y_1)$ to $\phi(y)$ fellow-travel for a time at least $n-1$.
According to strong relative Ancona inequalities, we thus have, for some $C\geq 0$ and $0<\rho<1$,
$$\left | \frac{G(e,\phi(x)|r)G(\phi(y_1),\phi(y)|r)}{G(\phi(x_1),\phi(x)|r)G(e,\phi(x)|r)} -1\right |\leq C\rho^n.$$

Weak relative Ancona inequalities also show that
$$G(e,\phi(x)|r)G(\phi(y_1),\phi(y)|r)\geq \frac{1}{C}G(e,\phi(x_1)|r)G(\phi(x_1),\phi(x)|r)G(\phi(y_1),\phi(y)|r)$$
and since $x_1=y_1$, we get
$$G(e,\phi(x)|r)G(\phi(y_1),\phi(y)|r)\geq \frac{1}{C^2}G(e,\phi(y)|r)G(\phi(x_1),\phi(x)|r).$$
Thus, $\frac{G(e,\phi(x)|r)G(\phi(y_1),\phi(y)|r)}{G(\phi(x_1),\phi(x)|r)G(e,\phi(x)|r)}$
is bounded away from 0, so that
\begin{equation}\label{equationAnconaHolder}
    \begin{split}
    &\left | \log \left (\frac{H(e,\phi(x)|r)}{H(\phi(x_1),\phi(x)|r)}\right ) - \log \left (\frac{H(e,\phi(y)|r)}{H(\phi(y_1),\phi(y)|r)}\right )\right |\\
    &\leq C_1\left | \frac{G(e,\phi(x)|r)G(\phi(y_1),\phi(y)|r)}{G(\phi(x_1),\phi(x)|r)G(e,\phi(y)|r)} -1\right | \leq C\rho^n.
    \end{split}
\end{equation}
This proves that if $x=(x_1,...,x_n,...)\in \partial \Sigma_A^*$, then the sequence $\varphi_r(x_1,...,x_k)$ is Cauchy, so that it converges to some well-defined limit $\varphi_r(x)$.
This extended function $\varphi_r$ on $\overline{\Sigma}_A$ still satisfies~(\ref{equationAnconaHolder}), so that it is locally H\"older continuous.
\end{proof}

We denote by $\mathcal{L}_r$ the transfer operator associated with the function $\varphi_r$.

\begin{lemma}\label{finitepressureGreen}
For every $r\in [1,R_{\mu}]$, the function $\varphi_r$ has finite pressure and is positive recurrent.
\end{lemma}

\begin{proof}
As noted above, since the Markov shift has finitely many images, Proposition~\ref{finitelymanyimagesimpliespositiverecurrent} shows that any locally H\"older function with finite pressure is positive recurrent.
Thus, we only need to prove that $\varphi$ has finite pressure, which is equivalent to proving that $\|\mathcal{L}_r1\|_\infty <+\infty$ by \cite[Theorem~1]{Sarig1}.
For $x\in \overline{\Sigma}_A$, let $X_x^1$ be the set of symbols that can precede $x$ in the automaton $\mathcal{G}$.
Then, by weak relative Ancona inequalities,
$$\mathcal{L}_r1(x)=\sum_{\sigma\in X_x^1}\frac{H(e,\sigma x|r)}{H(\sigma,\sigma x|r)}\lesssim \sum_{k=0}^N\sum_{\sigma\in \mathcal{H}_k'}H(e,\sigma|r).$$
This last sum is bounded by Corollary~\ref{derivativeparabolicGreenfinite}, which concludes the proof.
\end{proof}

Let $P_j(r)$ be the pressure of the restriction of $\varphi_r$ to a component $\overline{\Sigma}_{A,j}$ of the Markov shift
and let $P(r)$ be the maximal pressure, that is the maximum of the $P_j(r)$.
Recall that we declared that the empty sequence is not a preimage of the empty sequence.
This will simplify the following

The main reason for introducing this function $\varphi_r$ is that
$$\mathcal{L}_r^n1_{E_*}(\emptyset)=\frac{1}{H(e,e|r)}\sum_{\gamma \in \hat{S}_n}H(e,\gamma|r),$$
where we recall that $1_{E_*}$ is the function on $\overline{\Sigma}_A$ that takes value 1 on paths that start at $v_*$ in the automaton $\mathcal{G}$ and 0 elsewhere.
Indeed, to prove Theorem~\ref{preciseequadiff}, we want to understand $I^{(1)}(r)=\sum_{\gamma\in \Gamma}H(e,\gamma|r)$.
Thus, we want to understand the behavior of $\sum_{\gamma \in \hat{S}_n}H(e,\gamma|r)$, which is thus the same as understanding the behavior of $\mathcal{L}_r^n1_{E_*}(\emptyset)$.

\subsection{Continuity properties of the transfer operator}\label{Sectioncontinuitytransferoperator}
Our goal in this subsection is to prove that the map $r\mapsto \mathcal{L}_r$ is continuous in a weak sense.
We begin by the following result.

\begin{lemma}\label{corollary3.3Gouezel}
There exists $C>0$ such that for all $r\in [1,R_{\mu})$,
$$\frac{1}{C}\frac{1}{\sqrt{R_{\mu}-r}}\leq \sum_{\gamma\in \Gamma}H(e,\gamma|r)\leq C \frac{1}{\sqrt{R_{\mu}-r}}.$$
\end{lemma}

\begin{proof}
Let $I_1(r)=\sum_\gamma H(e,\gamma|r)$ and $F(r)=r^2I_1(r)$.
According to \cite[Lemma~3.2]{DussauleLLT1},
$$F'(r)=2r\sum_{\gamma,\gamma'}G(e,\gamma'|r)G(\gamma',\gamma|r)G(\gamma,e|r).$$
Proposition~\ref{roughequadiff} gives
$$\frac{1}{C}\leq \frac{F'(r)}{F(r)^3}\leq C$$
and Proposition~\ref{relationR_kG'(R)} gives $F(R_{\mu})=+\infty$.
Thus, integrating the inequality above between $r$ and $R_{\mu}$, we get
$$\frac{1}{C}(R_{\mu}-r)\leq \frac{1}{F(r)^2}\leq C(R_{\mu}-r),$$ which is the desired inequality.
\end{proof}

We also prove the following.
\begin{lemma}\label{pressureatthespectralradius}
For any $r\in [1,R_{\mu}]$, $P(r)\leq 0$.
Moreover, $P(R_{\mu})=0$ and $\varphi_{R_{\mu}}$ is semisimple.
\end{lemma}

\begin{proof}
If $P(r)$ were positive, then Theorem~\ref{GouezelSarig} would show that
$$\mathcal{L}_r1_{E_*}(\emptyset)=\frac{1}{H(e,e|r)}\sum_{\gamma\in \hat{S}_n}H(e,\gamma|r)$$
tends to infinity.
However, Lemma~\ref{finitesumalongspheres} shows that this quantity is bounded, so we get a contradiction.

If $P(R_{\mu})$ were negative, then
$\sum_{\gamma \in \hat{S}_n}H(e,\gamma|r)$ would converge to 0 exponentially fast, according to Theorem~\ref{GouezelSarig}.
In particular, $\sum_{\gamma \in \Gamma}H(e,\gamma|r)$ would be finite, that is, using Lemma~\ref{lemmafirstderivative}, $\frac{d}{dr}_{|r=R_\mu}G(e,e|r)$ would be finite.
This would be a contradiction with Proposition~\ref{relationR_kG'(R)}.

Finally, if $\varphi_{R_{\mu}}$ were not semisimple, then Lemma~\ref{lemmaGouezelSarig} would again show that
$\sum_{\gamma\in \hat{S}_n}H(e,\gamma|r)$ would tend to infinity, since we already know that $P(R_{\mu})=0$.
Again, Lemma~\ref{finitesumalongspheres} shows that this quantity is bounded.
\end{proof}

Let $h_{j}^{(i)}$ and $\nu_j^{(i)}$ be the functions and measures given by Theorem~\ref{GouezelSarig}, associated with $\mathcal{L}_{R_\mu}$.
Let $m_j$ be the measure defined as $dm_j=\frac{1}{p_j}\sum_{i=1}^{p_j}h_{j}^{(i)}d\nu_j^{(i)}$.
According to \cite[Proposition~4]{Sarig1}, $m_j$ is a Gibbs measure.
However, we have to apply this proposition to each component $\overline{\Sigma}_{A,j}$ of the shift, so that we do not have that $m_j([x_1,...,x_n])\asymp H(e,x_1...x_n|R_\mu)$ for any cylinder $[x_1,...,x_n]$.
We still deduce that there exists $C\geq 0$ such that for any $n$, for any cylinder $[x_1,...,x_n]$,
\begin{equation}\label{dominationmesuredeGibbs}
    m_j([x_1...x_n]) \leq C H(e,x_1...x_n|R_\mu).
\end{equation}
Furthermore, letting $m=\sum_j m_j$, there exists $C\geq 0$
such that for any $n$, for any cylinder $[x_1...x_n]$ in the support of one of the measures $m_j$,
\begin{equation}\label{equivalencemesuredeGibbs}
    \frac{1}{C} H(e,x_1...x_n|R_\mu)\leq m([x_1,...,x_n]) \leq C  H(e,x_1...x_n|R_\mu).
\end{equation}

\begin{proposition}\label{preparationcontinuiteoperateurtransfert}
There exists a non-negative function $\varphi$  on $\overline{X}_A$, possibly taking the value $+\infty$ on $\partial \Sigma_A$, which is integrable with respect to the measure $m$ and such that for every $x\in \Sigma_A$ and for every $1\leq r,r'\leq R_\mu$,
$$|\varphi_r(x)-\varphi_{r'}(x)|\leq 2\varphi(x)\sqrt{|r-r'|}.$$
\end{proposition}

Integrating over $r$, this proposition is a direct consequence of the following lemma.

\begin{lemma}\label{lemmecontinuiteoperateurtransfert}
There exists a non-negative function $\varphi$  on $\overline{X}_A$, possibly taking the value $+\infty$ on $\partial \Sigma_A$, which is integrable with respect to the measure $m$ and such that for every $x\in \Sigma_A$ and for every $1\leq r< R_\mu$,
$$\left |\frac{d}{dr}\varphi_r(x)\right |\leq \varphi(x) \frac{1}{\sqrt{R_\mu-r}}.$$
\end{lemma}

\begin{proof}
Fix $x\in \Sigma_A$.
We compute the derivative of $r\mapsto \varphi_r(x)$.
To simplify the notations, we identify $x$ with $\phi(x)\in \Gamma$.
We get
$$\frac{d}{dr}\varphi_r(x)=\frac{d}{dr}\log \frac{H(e,x|r)}{H(x_1,x|r)}=\frac{d}{dr}\log r^2H(e,x|r)-\frac{d}{dr}\log r^2H(x_1,x|r).$$
\cite[Lemma~3.2]{DussauleLLT1} shows that
\begin{equation}\label{lemmecontinuite1}
\begin{split}
    &\frac{d}{dr}\varphi_r(x)=\sum_{y\in \Gamma}\frac{G(e,y|r)G(y,x|r)G(x,e|r)+G(e,x|r)G(x,y|r)G(y,e|r)}{rH(e,x|r)}\\
    &-\sum_{y\in \Gamma}\frac{G(x_1,y|r)G(y,x|r)G(x,x_1|r)+G(x_1,x|r)G(x,y|r)G(y,x_1|r)}{rH(x_1,x|r)}.
\end{split}
\end{equation}
We first give an upper-bound for
$$\left |\frac{\sum_{y\in \Gamma}G(e,y|r)G(y,x|r)G(x,e|r)}{rH(e,x|r)}-\frac{\sum_{y\in \Gamma}G(x_1,y|r)G(y,x|r)G(x,x_1|r)}{rH(x_1,x|r)}\right |.$$
The remaining term in~(\ref{lemmecontinuite1}) will be bounded in the same way.
Putting together these two sums, we get
\begin{align*}
    &\frac{1}{rH(e,x|r)}\sum_{y\in \Gamma}G(e,y|r)G(y,x|r)G(x,e|r)\\
    &\hspace{3cm}-G(x_1,y|r)G(y,x|r)G(x,x_1|r)\frac{H(e,x|r)}{H(x_1,x|r)}.
\end{align*}

We rewrite this as
$$\frac{1}{rH(e,x|r)}\sum_{y\in \Gamma}G(e,y|r)G(y,x|r)G(x,e|r) \left ( 1-\frac{G(x_1,y|r)G(e,x|r)}{G(x_1,x|r)G(e,y|r)}\right ).$$
We decompose the sum over $\Gamma$ in the following way.
Let $n=\hat{d}(e,x)$, so that the relative geodesic $[e,x]$ has length $n$.
Denote by $e,x_1,...,x_{n}$ successive points on this relative geodesic.
Also, for $0\leq k \leq n$, let $\Gamma_k$ be the set of $y\in \Gamma$ whose projection on $[e,x]$ which is closest to $x$ is exactly at $x_k$.

We will use several times Lemma~\ref{lemmarelativetripod}.
Let us first focus on the sum over $\Gamma_0$.
If $y\in \Gamma_0$, then any relative geodesic from $y$ to $x$ passes within a bounded distance of $e$.
Weak relative Ancona inequalities show that
$$G(y,x|r)\lesssim G(y,e|r)G(e,x|r).$$
Similarly, any relative geodesic from $x_1$ to $y$ passes within a bounded distance of $e$, hence
$$G(x_1,y|r)\lesssim G(x_1,e|r)G(e,y|r).$$
We also have
$$G(e,x|r)\lesssim G(e,x_1|r)G(x_1,x|r),$$
so that
\begin{align*}
    &\left |\frac{1}{rH(e,x|r)}\sum_{y\in \Gamma_0}G(e,y|r)G(y,x|r)G(x,e|r) \left ( 1-\frac{G(x_1,y|r)G(e,x|r)}{G(x_1,x|r)G(e,y|r)}\right )\right |\\
    &\hspace{2cm} \lesssim \sum_{y\in \Gamma}H(e,y|r)\left (1+H(e,x_1|r)\right ).
\end{align*}
Since $H(e,x_1|r)$ is uniformly bounded, we deduce from Lemma~\ref{corollary3.3Gouezel} that
\begin{align*}
    &\left |\frac{1}{rH(e,x|r)}\sum_{y\in \Gamma_0}G(e,y|r)G(y,x|r)G(x,e|r) \left ( 1-\frac{G(x_1,y|r)G(e,x|r)}{G(x_1,x|r)G(e,y|r)}\right )\right |\\
    &\hspace{2cm}\lesssim \frac{1}{\sqrt{R_\mu-r}}.
\end{align*}

Let us focus on the sum over $\Gamma_1$ now.
Let $\mathcal{H}_1$ be the union of parabolic subgroups containing $x_1$.
Let $y\in \Gamma_1$ and denote by $\sigma$ its projection on $\mathcal{H}_1$.
Then, any relative geodesic from $e$ to $y$ passes within a bounded distance of $\sigma$ and any relative geodesic from $y$ to $x$ passes first to a point within a bounded distance of $\sigma$, then to a point within bounded distance of $x_1$.
We thus get
$$G(e,y|r)\lesssim G(e,\sigma|r)G(\sigma,y|r)$$
and
$$G(y,x|r)\lesssim G(y,\sigma|r)G(\sigma,x_1|r)G(x_1,x|r).$$
Similarly,
$$\frac{G(x_1,y|r)G(e,x|r)}{G(x_1,x|r)G(e,y|r)}\lesssim \frac{G(x_1,\sigma|r)G(e,x_1|r)}{G(e,\sigma|r)}\lesssim 1.$$
Letting $\Gamma_1^{\sigma}$ be the set of $y$ whose projection on $\mathcal{H}_1$ is at $\sigma$, we get
\begin{align*}
    &\left |\frac{1}{rH(e,x|r)}\sum_{y\in \Gamma_1}G(e,y|r)G(y,x|r)G(x,e|r) \left ( 1-\frac{G(x_1,y|r)G(e,x|r)}{G(x_1,x|r)G(e,y|r)}\right )\right |\\
    &\hspace{2cm} \lesssim \sum_{\sigma\in \mathcal{H}_1}\sum_{y\in \Gamma_1^\sigma}\frac{G(e,\sigma|r)G(\sigma,x_1|r)}{G(e,x_1|r)}H(\sigma,y|r).
\end{align*}
We bound the sum over $y\in \Gamma_1^\sigma$ by a sum over $y\in \Gamma$, so that
\begin{align*}
&\left |\frac{1}{rH(e,x|r)}\sum_{y\in \Gamma_1}G(e,y|r)G(y,x|r)G(x,e|r) \left ( 1-\frac{G(x_1,y|r)G(e,x|r)}{G(x_1,x|r)G(e,y|r)}\right )\right |\\
&\hspace{2cm} \lesssim \frac{1}{\sqrt{R_\mu-r}}\sum_{\sigma\in \mathcal{H}_1}\frac{G(e,\sigma|r)G(\sigma,x_1|r)}{G(e,x_1|r)}.
\end{align*}

Suppose now that $k\geq 2$ and consider the sum over $\Gamma_k$.
For any $y\in \Gamma_k$, relative geodesic from $x_1$ to $y$ and from $e$ to $x$ travel together for a time at least $k-1$.
We deduce from strong relative Ancona inequalities that
$$\left |1-\frac{G(x_1,y|r)G(e,x|r)}{G(x_1,x|r)G(e,y|r)}\right |\lesssim \rho^k$$
for some $0<\rho<1$.
Letting $\mathcal{H}_k$ be the union of parabolic subgroups containing $x_{k-1}^{-1}x_{k}$, we also get
\begin{align*}
&\left |\frac{1}{rH(e,x|r)}\sum_{y\in \Gamma_k}G(e,y|r)G(y,x|r)G(x,e|r) \left ( 1-\frac{G(x_1,y|r)G(e,x|r)}{G(x_1,x|r)G(e,y|r)}\right )\right |\\
&\hspace{2cm} \lesssim \rho^k\frac{1}{\sqrt{R_\mu-r}}\sum_{\sigma\in \mathcal{H}_k}\frac{G(x_{k-1},x_{k-1}\sigma|r)G(x_{k-1}\sigma,x_{k}|r)}{G(x_{k-1},x_{k}|r)}.
\end{align*}
Putting everything together and letting $x_0=e$, we get
\begin{equation}\label{lemmecontinuite2}
\begin{split}
    &\left |\frac{1}{rH(e,x|r)}\sum_{y\in \Gamma}G(e,y|r)G(y,x|r)G(x,e|r) \left ( 1-\frac{G(x_1,y|r)G(e,x|r)}{G(x_1,x|r)G(e,y|r)}\right )\right |\\
    &\hspace{2cm} \lesssim \frac{1}{\sqrt{R_\mu-r}} \left (1+\sum_{k=0}^{n-1}\rho^k\sum_{\sigma\in \mathcal{H}_k}\frac{G(x_k,x_k\sigma|r)G(x_k\sigma,x_{k+1}|r)}{G(x_k,x_{k+1}|r)} \right ).
\end{split}
\end{equation}
    
We now bound $\sum_{k=0}^n\rho^k\sum_{\sigma\in \mathcal{H}_k}\frac{G(x_k,x_k\sigma|r)G(x_k\sigma,x_{k+1}|r)}{G(x_k,x_{k+1}|r)}$ by an integrable function independently of $r$ and $n$.
We first prove the following.

\begin{lemma}\label{souslemmecontinuite1}
There exists $\Lambda$ such that the following holds.
Let $\mathcal{H}$ be a parabolic subgroup.
For every $x$ in $\mathcal{H}$ and for any $1\leq r\leq R_\mu$, we have
$$\sum_{y\in \mathcal{H}}G(e,y|r)G(y,x|r)\leq( \Lambda d(e,x)+\Lambda )G(e,x|r).$$
\end{lemma}

\begin{proof}
We write $G^{(1)}_{r}(e,x)=\frac{d}{dt}_{|t=1}\left (G_{r}(e,x|t)\right )$, where $G_r$ is the Green function associated with the first return kernel $p_{r}$ to $\mathcal{H}$.
According to Lemma~\ref{lemmafirstderivative}, it is enough to prove that
$$G^{(1)}_{r}(e,x)\leq ( \Lambda d(e,x)+\Lambda )G(e,x|r).$$
Since we are assuming that $\mu$ is non-spectrally degenerate along $\mathcal{H}$, there exists $\rho<1$ such that for any $x$, for any $n$,
$$p_{r}^{(n)}(e,x)\leq p_{R_\mu}^{(n)}(e,x)\leq \rho^{n},$$
where $p_{r}^{(n)}$ denotes the $n$th power of convolution of $p_{r}$.
By definition,
$$G^{(1)}_{r}(e,x)=\sum_{n\geq 0}np_{r}^{(n)}(e,x).$$
Notice then that
$$\sum_{n\geq \Lambda d(e,x)+\Lambda}np_{r}^{(n)}(e,x)\lesssim (\rho')^{\Lambda d(e,x)+\Lambda}.$$
Since $r\geq 1$,
for any $x$, $G(e,x|r)\geq p^{d(e,x)}$ for some $p<1$ and so
$$G^{(1)}_{r}(e,x)\geq G(e,x|r)\geq p^{d(e,x)}.$$
If $\Lambda$ is large enough, we thus have
$$\sum_{n\geq \Lambda d(e,x)+\Lambda}np_{r}^{(n)}(e,x)\leq \frac{1}{2}G^{(1)}_{r}(e,x),$$
so that
$$G^{(1)}_{r}(e,x)\leq 2\sum_{n\leq \Lambda d(e,x)+\Lambda}np_{r}^{(n)}(e,x)\leq (2\Lambda d(e,x)+2\Lambda)G(e,x|r).$$
This concludes the proof.
\end{proof}

Going back to the proof of Lemma~\ref{lemmecontinuiteoperateurtransfert}, we get the upper-bound
$$\sum_{k=0}^n\rho^k\sum_{\sigma\in \mathcal{H}_k}\frac{G(x_k,x_k\sigma|r)G(x_k\sigma,x_{k+1}|r)}{G(x_k,x_{k+1}|r)}\leq \sum_{k=0}^n\rho^k(\Lambda d(x_{k},x_{k+1})+\Lambda).$$

We fix $N$ and define $\varphi_1^{(N)}$ by
\begin{equation}\label{defvarphi}
    \varphi_1^{(N)}(x)=1+\sum_{k=0}^{n-1}\rho^k(\Lambda d(x_{k},x_{k+1})+\Lambda)
\end{equation}
for any word $x$ of length $n\leq N$.
We extend $\varphi_1^{(N)}$ to a function on $\overline{\Sigma}_A$ declaring $\varphi_1^{(N)}$ to be constant on cylinders of length $N$.
For fixed $x$, the sequence $\varphi_1^{(N)}(x)$ is non-decreasing.
Let $\varphi_{1}(x)$ be its limit, possibly infinite if $x\in \partial\Sigma_A$.
We now prove that $\varphi_1$ is integrable with respect to $m$.
This will be based on the following.

\medskip
Letting $E$ be a set, a transition kernel $p$ is a function
$p:E\times E\rightarrow [0,+\infty)$.
We fix a base point $x_0\in E$.
We say that $p$ is finite if its total mass is finite, that is
$$\sum_{x\in E}p(x_0,x)<+\infty.$$
If the total mass is 1, then $p$ defines a Markov chain $Z_n$ on $E$.
Otherwise, $p$ still defines a chain $Z_n$ with transition given by $p$.
We let $z_n$ be the increments of this chain.
Whenever $E$ is endowed with a distance $d$, we say that $p$ is $C$-quasi-invariant if there exists $C$ such that for any $k$,
$$d(Z_k,Z_{k+1})\leq C d(e,z_{k+1}).$$

\begin{lemma}\label{souslemmecontinuite2}
Let $p$ be a finite $C$-quasi-invariant transition kernel on a countable metric space $(E,d)$.
Let $x_0$ be a fixed point in $E$.
Assume that $p$ has exponential moments in the sense that
$$\sum_{x\in E}p(x_0,x)\mathrm{e}^{\alpha d(x_0,x)}$$
for some positive $\alpha$.
Then, for any $\beta>0$, there exists $\lambda>0$ and $C_{\lambda}$ such that for any $x\in E$,
$$\sum_{n\leq d(e,x)/\lambda}p^{(n)}(x_0,x)\leq C_{\lambda}\mathrm{e}^{-\beta d(e,x)},$$
where $p^{(n)}$ denotes the $n$th power of convolution of $p$.
\end{lemma}

\begin{proof}
The proof is contained in the proof of \cite[Lemma~3.6]{BlachereHassinskyMathieu}, although the statement and the assumptions there are different, so we rewrite it for convenience.
To simplify the notations, we assume that the total mass of $p$ is 1, so that $p$ defines a Markov chain $Z_n$.
The general proof is the same.
By assumption, we have
$$\mathbb{E}\left (\mathrm{e}^{\alpha d(x_0,Z_1)}\right )=E<+\infty.$$
For any $\lambda$, Markov inequality shows that
$$\mathbb{P}\left (\sup_{1\leq k\leq n}d(e,Z_k)\geq \lambda n\right )\leq \mathrm{e}^{-\frac{\alpha}{C} \lambda n}\mathbb{E}\left (\mathrm{e}^{\frac{\alpha}{C} \sup_{1\leq k\leq n}d(e,Z_k)}\right ).$$
Since $p$ is $C$-quasi-invariant, we have for any $k\leq n$, letting $Z_0=x_0$,
$$d(x_0,Z_k)\leq \sum_{0\leq j\leq n-1}d(Z_j,Z_{j+1})\leq C \sum_{1\leq j\leq n}d(x_0,z_j).$$
Since the $z_j$ are independent and follow the same law as $Z_1$, we get
$$\mathbb{P}\left (\sup_{1\leq k\leq n}d(e,Z_k)\geq \lambda n\right )\leq \mathrm{e}^{-\frac{\alpha}{C} \lambda n}E^n\leq \mathrm{e}^{n(-\frac{\alpha}{C} \lambda +\log E)}.$$
We choose $\lambda$ large enough so that $-\frac{\alpha}{C} \lambda +\log E\leq -2\beta$.
Then,
$$\sum_{n\leq d(e,x)/\lambda}p^{(n)}(x_0,x)\leq \frac{d(e,x)}{\lambda}\mathrm{e}^{-2\beta d(e,x)}\lesssim \mathrm{e}^{-\beta d(e,x)}.$$
This concludes the proof.
\end{proof}

We apply this in our situation.
Let $\mathcal{H}$ be a parabolic subgroup.
For any $\eta>0$, we let $p_{\eta,R_\mu}$ be the first return kernel associated with $R_\mu\mu$ to the $\eta$-neighborhood $\mathcal{N}_{\eta}(\mathcal{H})$ of $\mathcal{H}$.
Then, \cite[Lemma~4.5]{DussauleGekhtman} shows that if $\eta$ is large enough, then $p_{\eta,R_\mu}$ has exponential moments.
Since it is defined as the first return associated with $R_\mu\mu$, it is $C$-quasi-invariant for the induced metric on $\mathcal{N}_{\eta}(\mathcal{H})$ (it is actually invariant for this distance).
Thus, for any $\beta>0$, there exists $\lambda$ and $C_\lambda$ such that
$$\sum_{n\leq d(e,x)/\lambda}p_{\eta,R_\mu}^{(n)}(e,x)\leq C_{\lambda}\mathrm{e}^{-\beta d(e,x)}.$$
The Green function associated with $p_{\eta,R_\mu}$ coincides with the restriction of the Green function associated with $R_\mu\mu$ on $\mathcal{N}_{\eta}(\mathcal{H})$, see \cite[Lemma~3.4]{DussauleLLT1} for a proof.
Since there exists $q<1$ such that $G(e,x|R_\mu)\geq q^{d(e,x)}$, we can choose $\lambda$ so that
$$\sum_{n\leq d(e,x)/\lambda}p_{\eta,R_\mu}^{(n)}(e,x)\leq \frac{1}{2}G(e,x|R_\mu)$$
and so
$$ G(e,x|R_\mu)\leq 2\sum_{n\geq d(e,x)/\lambda}p_{\eta,R_\mu}^{(n)}(e,x)\leq \frac{2\lambda}{d(e,x)}\sum_{n\geq d(e,x)}np_{\eta,R_\mu}^{(n)}(e,x).$$
According to Lemma~\ref{lemmafirstderivative},
$$G(e,x|R_\mu)\leq \frac{2\lambda}{d(e,x)}\sum_{y\in \mathcal{N}_{\eta}(\mathcal{H})}G(e,y|R_\mu)G(y,x|R_\mu).$$
Finally, any point in $\mathcal{N}_\eta(\mathcal{H})$ is within $\eta$ of a point in $\mathcal{H}$, hence for any $x\in \mathcal{H}$,
\begin{equation}
   d(e,x) G(e,x|R_\mu)\lesssim \sum_{y\in \mathcal{H}}G(e,y|R_\mu)G(y,x|R_\mu),
\end{equation}
since $\eta$ is fixed.

Recall that we want to prove that $\varphi_1$ is integrable with respect to $m$.
Since $\varphi_1^{(n)}$ is non decreasing, it is enough to show that there exists a uniform $C\geq 0$ such that for any $n$,
$$\int \varphi_1^{(n)}dm\leq C.$$
By definition, $\varphi_1^{(n)}$ is constant on cylinders of the form $[x_1,...,x_n]$.
According to~(\ref{dominationmesuredeGibbs}), we just need to show that for every $n$,
\begin{equation}\label{sommefiniedefvarphi}
    \sum_{x\in \hat{S}^N}H(e,x|R_\mu)\sum_{k=0}^{n-1}\rho^kd(x_k,x_{k+1})
\end{equation}
is uniformly bounded.

We decompose $x\in \hat{S}^n$ as $x=x_1...x_n$.
For any $y\in \hat{S}^k$, denote by $X^y_1$ the set of symbols $\sigma$ which can follow $y$ in the automaton $\mathcal{G}$.
More generally, denote by $X^y_j$ the set of words of length $j$ which can follow $y$.
For fixed $k$ writing $\sigma_{k+1}=x_{k}^{-1}x_{k+1}$ and $y=x_{k+1}^{-1}x$, we have, using weak relative Ancona inequalities,
\begin{align*}
    &\sum_{x\in \hat{S}^n}H(e,x|R_\mu) d(x_k,x_{k+1})\\
    &\lesssim \sum_{x_k\in \hat{S}^k}\sum_{\sigma_{k+1} \in X^{x_k}_1}\sum_{y\in X^{x_{k+1}}_{n-k-1}}H(e,x_k|R_\mu)H(e,y|R_\mu)\\
    &\hspace{5cm}d(e,\sigma_{k+1})G(e,\sigma_{k+1}|R_\mu)G(\sigma_{k+1},e|R_\mu)\\
    &\lesssim \sum_{x_k\in \hat{S}^k}\sum_{\sigma_{k+1} \in X^{x_k}_1}\sum_{y\in X^{x_{k+1}}_{n-k-1}}
    H(e,x_k|R_\mu)H(e,y|R_\mu)\\
    &\hspace{5cm}\sum_{\sigma\in \mathcal{H}_k}G(e,\sigma|R_\mu)G(\sigma,\sigma_{k+1}|R_\mu)G(\sigma_{k+1},e|R_\mu).
\end{align*}

Lemma~\ref{finitesumalongspheres} shows that
$$\sum_{y\in X^{x_{k+1}}_{n-k-1}}H(e,y|R_\mu)\lesssim 1.$$
Since $\mu$ is not spectrally degenerate,
$$\textbf{}\sum_{\sigma\in \mathcal{H}_k}G(e,\sigma|R_\mu)G(\sigma,\sigma_{k+1}|R_\mu)G(\sigma_{k+1},e|R_\mu)\lesssim 1.$$
Using again Lemma~\ref{finitesumalongspheres},
$$\sum_{x_k\in \hat{S}^k}H(e,x_k|R_\mu)\lesssim 1.$$
Finally, we find that~(\ref{sommefiniedefvarphi}) is bounded by $C\sum_{k=0}^{n-1}\rho^k$ for some $C$.
Since $\rho<1$, this last sum is uniformly bounded.

To conclude, we give a similar bound for the remaining term in~(\ref{lemmecontinuite1}), with an integrable function $\varphi_2$.
We set $\varphi=\varphi_1+\varphi_2$.
This concludes the proof.
\end{proof}

The function $\varphi$ is constructed as the non-decreasing limit of functions $\varphi^{(n)}$ which are uniformly integrable and satisfy that for any word $x$ of length $n$,
\begin{equation}\label{majorationderiveeparphin}
    \left |\frac{d}{dr}\varphi_r(x)\right |\leq \varphi^{(n)}(x) \frac{1}{\sqrt{R_\mu-r}}.
\end{equation}
Let us notice that we proved something a bit stronger than
$\int \varphi^{(n)}dm\lesssim 1$.
Indeed, we proved there exists $C\geq 0$ such that for every $n$,
\begin{equation}\label{meilleureestimeeintegralephin}
    \sum_{x\in \hat{S}^n}H(e,x|R_\mu)\varphi^{(n)}(x)\leq C.
\end{equation}
We will both use~(\ref{majorationderiveeparphin}) and~(\ref{meilleureestimeeintegralephin}) in the following.
However, to simplify the notations, we only stated Lemma~\ref{lemmecontinuiteoperateurtransfert} using $\varphi$ and $m$.

\medskip
We also prove the following result.
We will not use in full generality, but only for $x=\emptyset$.

\begin{proposition}\label{continuiteoperateurtransfert}
For every $x\in \Sigma_A$, there exists $C_x$ such that for every $r,r'\leq R_{\mu}$ and for every bounded function $f$,
$$\left | (\mathcal{L}_rf)(x)-(\mathcal{L}_{r'}f)(x)\right |\leq C_x\|f\|_{\infty}\sqrt{|r-r'|}.$$
\end{proposition}

\begin{proof}
Fix $x\in \Sigma_A$ and let $n=\hat{d}(e,x)$.
Let $X_x^1$ be the set of symbols which can precede $x$ in the automaton $\mathcal{G}$.
Then,
$$(\mathcal{L}_rf)(x)-(\mathcal{L}_{r'}f)(x)=\sum_{\sigma\in X_x^1}\left (\mathrm{e}^{\varphi_r(\sigma x)}-\mathrm{e}^{\varphi_{r'}(\sigma x)}\right )f(\sigma x).$$
Differentiating in $r$ the quantity
$$\sum_{\sigma\in X_x^1}\mathrm{e}^{\varphi_r(\sigma x)}f(\sigma x),$$
we get
$$\sum_{\sigma\in X_x^1}\left (\frac{d}{dr}\varphi_r(\sigma x)\right )\mathrm{e}^{\varphi_r(\sigma x)}f(\sigma x).$$
Using Lemma~\ref{lemmecontinuiteoperateurtransfert}, this is bounded by
$$\|f\|_{\infty}\frac{1}{\sqrt{R_\mu-r}}\sum_{\sigma\in X_x^1}\varphi(\sigma x) \mathrm{e}^{\varphi_r(\sigma x)}.$$
We deduce from weak relative Ancona inequalities that this is bounded by
$$\|f\|_{\infty}\frac{1}{\sqrt{R_\mu-r}}\sum_{\sigma\in X_x^1}\varphi(\sigma x) H(e,\sigma|r).$$
Since $\sum_{\sigma\in X_x^1}\varphi(\sigma x) H(e,\sigma|r)$ only depends on $x$, it is enough to show that this last sum is bounded independently of $r$.
This is done exactly like showing that the sum~(\ref{sommefiniedefvarphi}) is bounded.
\end{proof}

We want to apply Theorem~\ref{KellerLiveranitransfert}, so we now prove that the assumptions of this theorem are satisfied, for $\tau(r)=\sqrt{R_\mu-r}$.
Recall that $m=\sum_j m_j$.

\begin{proposition}\label{prophypothesesKellerLiveranisatisfaites}
There exist constants $0<\sigma<1$ and $C\geq 0$ such that
for every $1 \leq r\leq R_\mu$, for every $n$, for every function $f\in H_{\rho,\beta}$,
$$\|\mathcal{L}_r^n f\|_{\rho,\beta}\leq C\sigma^n\|f\|_{\rho,\beta}+C\int |f| dm$$
and
$$\int \left |\left (\mathcal{L}_r-\mathcal{L}_{R_\mu}\right )f\right |dm\leq C\|f\|_{\rho,\beta}\sqrt{R_\mu-r}.$$
\end{proposition}

\begin{proof}
Let $f\in H_{\rho,\beta}$ and let $x\in \overline{\Sigma}_A$.
Denote by $S_n\varphi_r$ the $n$th Birkhoff sum of $\varphi_r$ and let $X_x^n$ the set of words of length $n$ which can precede $x$ in the automaton $\mathcal{G}$.
Then,
$$\mathcal{L}_r^n f(x)=\sum_{\gamma\in X_x^n}\mathrm{e}^{S_n\varphi_r(\gamma x)}f(\gamma x).$$
Let $f^{(n)}$ be the function which is constant on cylinders of length $n$ and which is equal to $f$ elsewhere.
In particular,
$$f^{(n)}(\gamma x)=f(\gamma)$$
for $x\in \overline{\Sigma}_A$ and $\gamma\in X_x^n$.
Since $f$ is $\rho$-locally H\"older,
$$|f^{(n)}(\gamma x)-f(\gamma x)|\leq \rho^n D_{\rho,\beta}(f).$$
Hence,
$$\left |\mathcal{L}_r^n f(x)\right |\leq \rho^nD_{\rho,\beta}(f)\sum_{\gamma\in X_x^n}\mathrm{e}^{S_n\varphi_r(\gamma x)}+\sum_{\gamma\in X_x^n}\mathrm{e}^{S_n\varphi_r(\gamma x)}|f^{(n)}(\gamma x)|.$$
To simplify, we identify an element $\gamma \in X_x^1$ with the corresponding element in $\hat{S}^n\subset \Gamma$.
Note that $\mathrm{e}^{S_n\varphi_r(\gamma x)}=\frac{H(e,\gamma x|r)}{H(\gamma,\gamma x|r)}$.
Using weak relative Ancona inequalities, we obtain
$$\left |\mathcal{L}_r^n f(x)\right |\lesssim \rho^nD_{\rho,\beta}(f)\sum_{\gamma \in \hat{S}^n}H(e,\gamma|R_\mu) +\sum_{\gamma \in \hat{S}^n}H(e,\gamma|R_\mu) |f^{(n)}(\gamma x)|.$$

For every $\gamma\in \Gamma$, we can use the automaton $\mathcal{G}$ and choose a relative geodesic from $e$ to $\gamma$ whose increments we denote by $x_1,...,x_n$.
Let $[\gamma]$ be the corresponding cylinder $[x_1,...,x_n]$.
Let $\hat{S}^n_{\mathrm{max}}$ be the set of $\gamma\in \hat{S}^n$ such that the cylinder $[\gamma]$ is in a maximal component.
Since $f^{(n)}$ is constant on cylinders of length $n$,~(\ref{equivalencemesuredeGibbs}) shows that
$$\sum_{\gamma \in \hat{S}^n_{\mathrm{max}}}H(e,\gamma|R_\mu) |f^{(n)}(\gamma x)|\lesssim \int |f^{(n)}|dm.$$
Also, by definition of maximal components, there exists $\rho'<e^{P(R_\mu)}=1$ such that
$$\sum_{\gamma \in \hat{S}^n\setminus \hat{S}^n_{\mathrm{max}}}H(e,\gamma|R_\mu) |f^{(n)}(\gamma x)|\lesssim (\rho')^n\|f\|_\infty.$$
Using that $f$ is $\rho$-locally H\"older, we get
$$\int |f^{(n)}|dm\lesssim \rho^n D_{\rho,\beta}(f)+\int |f|dm.$$
Lemma~\ref{finitesumalongspheres} shows that the sum $\sum_{\gamma \in \hat{S}^n}H(e,\gamma|R_\mu)$ is bounded independently of $n$.
We thus get
\begin{equation}\label{premierpartieKellerLiveranisatisfait1}
    \left |\mathcal{L}_r^n f(x)\right |\lesssim 2\rho^n D_{\rho,\beta}(f)+(\rho')^n\|f\|_\infty + \int |f|dm\lesssim \sigma^n \|f\|_{\rho,\beta}+\int |f|dm,
\end{equation}
where $\sigma=\max(\rho,\rho')$.
We can thus control $\|\mathcal{L}_r^n f\|_\infty$.

We now focus on $D_{\rho,\beta}(\mathcal{L}_rf)$.
Let $x,x'\in \overline{\Sigma}_A$ and let $\rho^m=d_{\rho}(x,x')$, $m\geq 1$.
We have
\begin{align*}
    \mathcal{L}_r^nf(x)-\mathcal{L}_r^nf(x')=&\sum_{\gamma\in X_x^n}\left (\mathrm{e}^{S_n\varphi_r(\gamma x)}-\mathrm{e}^{S_n\varphi_r(\gamma x')}\right )f(\gamma x)\\
    &\hspace{1cm}+\sum_{\gamma\in X_x^n}\mathrm{e}^{S_n\varphi_r(\gamma x')}(f(\gamma x)-f(\gamma x')).
\end{align*}
On the one hand, using that $f$ is $\rho$-locally H\"older and using weak relative Ancona inequalities to bound $\mathrm{e}^{S_n\varphi_r(\gamma x)}$ by $H(e,\gamma|r)$ as above, we get
$$\sum_{\gamma\in X_x^n}\mathrm{e}^{S_n\varphi_r(\gamma x')}|(f(\gamma x)-f(\gamma x'))|\lesssim \sum_{\gamma\in \hat{S}^n}H(e,\gamma|R_\mu) \rho^{n+m}D_{\rho,\beta}(f).$$
It follows from Lemma~\ref{finitesumalongspheres} that the sum $\sum_{\gamma\in \hat{S}^n}H(e,\gamma|R_\mu)$ is bounded and since $\rho^m=d_\rho(x,x')$, we have
$$\sum_{\gamma\in X_x^n}\mathrm{e}^{S_n\varphi_r(\gamma x')}|(f(\gamma x)-f(\gamma x'))|\lesssim \rho^n\|f\|_{\rho,\beta}d_\rho(x,x').$$
On the other hand,
\begin{align*}
    &\sum_{\gamma\in X_x^n}\left (\mathrm{e}^{S_n\varphi_r(\gamma x)}-\mathrm{e}^{S_n\varphi_r(\gamma x')}\right )f(\gamma x)\\
    &\hspace{3cm}=\sum_{\gamma\in X_x^n}\mathrm{e}^{S_n\varphi_r(\gamma x)}\left (1-\mathrm{e}^{S_n\varphi_r(\gamma x')-S_n\varphi_r(\gamma x)}\right )f(\gamma x).
\end{align*}
By definition,
$$\left (1-\mathrm{e}^{S_n\varphi_r(\gamma x')-S_n\varphi_r(\gamma x)}\right )=\left (1-\frac{H(e,\gamma x|r)H(\gamma,\gamma x'|r)}{H(\gamma,\gamma x|r)H(e,\gamma x'|r)}\right ).$$
Since relative geodesics $[e,x]$ and $[e,x']$ fellow travel for a time at least $m$ and since $\gamma$ both precedes $x$ and $x'$ in the automaton $\mathcal{G}$, relative geodesics $[e,\gamma x]$ and $[\gamma,\gamma x']$ also fellow travel for a time at least $m$.
Strong relative Ancona inequalities thus yield
$$\left |1-\mathrm{e}^{S_n\varphi_r(\gamma x')-S_n\varphi_r(\gamma x)}\right |\lesssim \rho^m$$
and so
$$\left |\sum_{\gamma\in X_x^n}\left (\mathrm{e}^{S_n\varphi_r(\gamma x)}-\mathrm{e}^{S_n\varphi_r(\gamma x')}\right )f(\gamma x)\right |\lesssim \rho^m\sum_{\gamma\in X_x^n}\mathrm{e}^{S_n\varphi_r(\gamma x)}|f(\gamma x)|.$$
We bound $\sum_{\gamma\in X_x^n}\mathrm{e}^{S_n\varphi_r(\gamma x)}|f(\gamma x)|$ by $\sigma^n\|f\|_{\rho,\beta}+\int |f|dm$ like above
to obtain
\begin{equation}\label{premierpartieKellerLiveranisatisfait2}
\left |\mathcal{L}_rf(x)-\mathcal{L}_rf(x')\right |\lesssim \left (\sigma^n\|f\|_{\rho,\beta}+\int |f|dm\right )d_\rho(x,x').
\end{equation}
We deduce from~(\ref{premierpartieKellerLiveranisatisfait1}) and~(\ref{premierpartieKellerLiveranisatisfait2}) that
$$\left \|\mathcal{L}_r^nf\right \|_{\rho,\beta}\lesssim \sigma^n \|f\|_{\rho,\beta}+\int |f|dm,$$
which concludes the first part of the proposition.

We now prove that
\begin{equation}\label{deuxiemepartiepropKellerLiveranisatisfait}
    \int \left |\left (\mathcal{L}_r-\mathcal{L}_{R_\mu}\right )f\right |dm\leq C\|f\|_{\rho,\beta}\sqrt{R_\mu-r}.
\end{equation}
The function $f$ is $\rho$-locally H\"older.
Since the operator $\mathcal{L}_r-\mathcal{L}_{R_\mu}$ is bounded on $(H_{\rho,\beta},\|\cdot \|_{\rho,\beta})$, for every word $x=x_1...x_n$ of length $n$ and for every $y\in [x_1,...,x_n]$,
$$\left |\left (\mathcal{L}_r-\mathcal{L}_{R_\mu}\right )f(y)-\left (\mathcal{L}_r-\mathcal{L}_{R_\mu}\right )f(x)\right |\lesssim \rho^n\|f\|_{\rho,\beta}.$$
Hence,
$$\left |\left (\mathcal{L}_r-\mathcal{L}_{R_\mu}\right )f(y)\right |\lesssim \rho^n\|f\|_{\rho,\beta}+\left |\left (\mathcal{L}_r-\mathcal{L}_{R_\mu}\right )f(x)\right |.$$
Fixing $r$, we choose $n$ large enough so that $\rho^n\leq \sqrt{R_\mu-r}$.
Let $\tilde{f}_r$ be the function which is constant on cylinders of length $n$ and which is equal to $\left (\mathcal{L}_r-\mathcal{L}_{R_\mu}\right )f(x)$ elsewhere.
To prove~(\ref{deuxiemepartiepropKellerLiveranisatisfait}), we just need to show that
$$\int \left |\tilde{f}_r\right | dm\lesssim \|f\|_{\rho,\beta}\sqrt{R_\mu-r}.$$
For every $x$ of length $n$ and for every $y$ in $[x_1,...,x_n]$,
we have
$$\tilde{f}_r(y)=\sum_{\sigma\in X_x^1}\left (\mathrm{e}^{\varphi_r(\sigma x)}-\mathrm{e}^{\varphi_{R_\mu}(\sigma x)}\right )f(\sigma x).$$
Differentiating this, we get
$$\sum_{\sigma\in X_x^1}\left (\frac{d}{dr}\varphi_r(\sigma x)\right )\mathrm{e}^{\varphi_r(\sigma x)}f(\sigma x).$$
Using~(\ref{majorationderiveeparphin}), we bound the absolute value of this term by
$$\frac{1}{\sqrt{R_\mu-r}}\|f\|_{\infty}\sum_{\sigma\in X_x^1}\varphi^{(n+1)}(\sigma x)\mathrm{e}^{\varphi_r(\sigma x)}.$$
Inverting the sum and the derivative is legitimate since weak relative Ancona inequalities show that $\mathrm{e}^{\varphi_r(\sigma x)}\lesssim H(e,\sigma)$.
As in the proof of Lemma~\ref{lemmecontinuiteoperateurtransfert}, we show that the sum
$$\sum_{\sigma\in X_x^1}\varphi^{(n+1)}(\sigma x)\mathrm{e}^{\varphi_r(\sigma x)}$$
is finite.
Therefore,
$$\left |\frac{d}{dr}\tilde{f}_r(y)\right |\lesssim  \frac{1}{\sqrt{R_\mu-r}}\|f\|_{\infty} \sum_{\sigma\in X_x^1}H(e,\sigma|R_\mu)\varphi^{(n+1)}(\sigma x).$$
Integrating this, we get
$$\int \left |\tilde{f}_r\right |dm\lesssim \|f\|_{\rho,\beta}\sqrt{R_\mu-r}\int \sum_{\sigma\in X_x^1}H(e,\sigma|R_\mu)\varphi^{(n+1)}(\sigma x)dm(x).$$
Since $x\mapsto \varphi^{(n+1)}(\sigma x)$ is constant on cylinders of length $n$,~(\ref{equivalencemesuredeGibbs}) shows that
\begin{align*}
    \int \left |\tilde{f}_r\right |dm&\lesssim \|f\|_{\rho,\beta}\sqrt{R_\mu-r}\sum_{x\in \hat{S}^n}\sum_{\sigma\in X_x^1}H(e,\sigma|R_\mu)H(e,x|R_\mu)\varphi^{(n+1)}(\sigma x)\\
    &\lesssim \|f\|_{\rho,\beta}\sqrt{R_\mu-r}\sum_{y\in \hat{S}^{n+1}}H(e,y|R_\mu)\varphi^{(n+1)}(y).
\end{align*}
According to~(\ref{meilleureestimeeintegralephin}), we thus have
$$\int \left |\tilde{f}_r\right |dm\lesssim \|f\|_{\rho,\beta}\sqrt{R_\mu-r},$$
which concludes the proof.
\end{proof}

We obtain the following corollary from Theorem~\ref{KellerLiveranitransfert}.

\begin{corollary}\label{coroKellerLiveranisatisfait}
For every $r$ close enough of $R_\mu$, there exist numbers $\tilde{P}_j(r)$, eigenfunctions $\tilde{h}_{j,r}^{(i)}$ and eigenmeasures $\tilde{\nu}_{j,r}^{(i)}$ of $\mathcal{L}_r$ associated with the eigenvalue $\mathrm{e}^{\tilde{P}_j(r)}$such that for every $g\in H_{\rho,\beta}$,
$$\left \|\mathcal{L}_{r}^ng-\sum_{j=1}^k\mathrm{e}^{n(\tilde{P}_j(r))}\sum_{i=1}^{p_j}\tilde{h}_{j,r}^{(i)}\int g d\tilde{\nu}_{j,r}^{((i-n) \text{ mod } p_j)}\right \|_{\rho,\beta}\leq C\theta^n\|g\|_{\rho,\beta},$$
where $C\geq 0$ and $0<\theta<1$.
The functions $\tilde{h}_{j}^{(i)}$ and the measures $\tilde{\nu}_j^{(i)}$ have the same support as the functions $h_j^{(i)}$ and the measures $\nu_j^{(i)}$ respectively.
Moreover, $\left \|\tilde{h}_{j,r}^{(i)}\right \|_{\rho,\beta}$ is uniformly bounded.
Finally,
$$\int \left |\tilde{h}_{j,r}^{(i)}-h_j^{(i)}\right |dm\underset{r\to R_\mu}{\longrightarrow}0$$
and $\tilde{\nu}_{j,r}^{(i)}$ weakly converges to $\nu_j^{(i)}$ as $r$ tends to $R_\mu$.
\end{corollary}

To conclude this section, note that Proposition~\ref{continuiteoperateurtransfert} yields
\begin{equation}\label{howtogettildehtendsimplementversh}
    |\mathcal{L}_rf(\emptyset)-\mathcal{L}_{r'}f(\emptyset)|\lesssim \|f\|_{\rho,\beta}\sqrt{|r'-r|}.
\end{equation}
Consequently,
\begin{equation}\label{tildehtendsimplementversh}
    |\tilde{h}_{j,r}^{(i)}(\emptyset)-h_j^{(i)}(\emptyset)|\underset{r\to R_\mu}{\longrightarrow}0.
\end{equation}
Indeed, this last estimate~(\ref{howtogettildehtendsimplementversh}) shows that we can replace the norm $|\cdot|_w$ when applying Theorem~\ref{KellerLiveranitransfert} with the norm $|\cdot|'_w$ defined by
$$|f|'_w=|f(\emptyset)|+\int |f|dm.$$
Although we will use~(\ref{tildehtendsimplementversh}) in the following, we preferred using the norm $|\cdot|_w$ in the statements and in the proofs for convenience.

\subsection{Evaluating the pressure}\label{Sectionpressure}
Our goal in the next two subsections is to estimate the numbers $\tilde{P}_j(r)$ given by Corollary~\ref{coroKellerLiveranisatisfait} and to compare them with the maximal pressure $P(r)$.
This allows us to obtain a precise estimate of $I^{(1)}(r)$.

According to Corollary~\ref{coroKellerLiveranisatisfait}, if $R_{\mu}-r$ is small enough, then
$$\mathcal{L}_{r}^n1_{E_*}(\emptyset)=\sum_{j=1}^k\mathrm{e}^{n\tilde{P}_j(\varphi_r)}\sum_{i=1}^{p_j}\tilde{h}_j^{(i)}(\emptyset)\int 1_{E_*} d\tilde{\nu}_j^{((i-n) \text{ mod } p_j)}+O\left (\theta^n\right ).$$
Here, $0<\theta<1$ and $k$ is the number of maximal components for the function $\varphi_{R_{\mu}}$.
Denote by $p$ the least common multiple of the periods of these components, so that if $n\geq 0$ and $0\leq q < p$, then $d\tilde{\nu}_j^{((i-np+q) \text{ mod } p_j)}$ only depends on $q$.
In particular, we can write
$$\mathcal{L}_{r}^{np+q}1_{E_*}(\emptyset)=\sum_{j=1}^k\mathrm{e}^{(np+q)\tilde{P}_j(\varphi_r)}\xi_{q,j}(r)+O\left (\theta^{np+q}\right ),$$
where $\xi_{q,j}$ is a non-negative function of $r$ defined on some fixed neighborhood of $R_{\mu}$.
Note that $\xi_{q,j}(r)$ only depends on $\tilde{h}_{j,r}(\emptyset)$ and on $\int 1_{E_*}d\tilde{\nu}_{j,r}^{(i)}$
and that it is continuous in $r$ according to Corollary~\ref{coroKellerLiveranisatisfait} and~(\ref{tildehtendsimplementversh}).

If $r<R_{\mu}$, then $\sum_{\gamma}H(e,\gamma|r)$ is finite, so the numbers $\tilde{P}_j(\varphi_r)$ are negative.
Summing over $n$ and $q\in \{0,...,p-1\}$, we get
\begin{equation}\label{3.6Gouezel}
    \sum_{\gamma \in \Gamma}H(e,\gamma|r)=H(e,e|r)\sum_{n,q}\mathcal{L}_{r}^{np+q}1_{E_*}(\emptyset)=\sum_{j=1}^k\frac{\xi_j(r)}{|\tilde{P}_j(\varphi_r)|}+O(1), r\to R_{\mu},
\end{equation}
for some non-negative functions $\xi_j$, which are continuous in $r$ on some neighborhood of $R_{\mu}$.

We have the following result, which shows that the pressure is asymptotically independent of the maximal components.
Its proof is postponed to the next subsection.
\begin{proposition}\label{pressureindependentcomponents}
For every $j\in \{1,...,k\}$,
$\frac{\tilde{P}_j(\varphi_r)}{P(r)}$ tends to 1 when $r$ tends to $R_{\mu}$,
where $P(r)$ is the maximal pressure of the function $\varphi_r$.
\end{proposition}

Combining Proposition~\ref{pressureindependentcomponents} and~(\ref{3.6Gouezel}), we get that
\begin{equation}\label{3.7Gouezel}
    I^{(1)}(r)=\sum_{\gamma \in \Gamma}H(e,\gamma|r)=\frac{\xi(r)}{|P(r)|}+O(1), r\to R_{\mu},
\end{equation}
for some non-negative function $\xi$, which is continuous in $r$ on some neighborhood of $R_{\mu}$.
Recall that $I^{(1)}(r)\asymp \sqrt{R_\mu-r}$ and $|P(r)|\asymp \sqrt{R_\mu-r}$.
Therefore, $\xi(R_{\mu})>0$, so that $\xi(r)$ is bounded away from zero on a neighborhood of $R_{\mu}$.

\subsection{Independence of the pressure: proof of Proposition~\ref{pressureindependentcomponents}}
This section is devoted to proving Proposition~\ref{pressureindependentcomponents}.
The analog of this result is proved in~\cite{Gouezel1}.
This is done showing that for $r\in[1,R_{\mu}]$, $\int \varphi_rdm_j$ does not depend on $j$, where
$m_j$ is the measure in Proposition~\ref{propperturbation}.
This in turn is proved in several steps.

\medskip
\textit{Step 1:} Fix $c$ and define $U(c)\subset \partial \hat{\Gamma}$ as the set of points $\xi\in \partial \hat{\Gamma}$ such that if $x$ is an infinite sequence in $\partial \Sigma_A^*$ defining $\xi$ (that is, denoting $\gamma_n=x_1...x_n$, $e,\gamma_1,...,\gamma_n,...$ is a relative geodesic ray that converges to $\xi$ in $\partial\hat{\Gamma}$), then $\log H(e,\gamma_n|r)/\hat{d}(e,\gamma_n)$ converges to $c$.
Then, the definition of $U(c)$ does not depend on the choice of the sequence $x_n$. Moreover, $U(c)$ is $\Gamma$-invariant, that is,
for any $\gamma \in \Gamma$, $\gamma \cdot U(c)=U(c)$.

\textit{Step 2:} Define the sequence of measures $\lambda_n=\sum_{\gamma \in \hat{S}_n}H(e,\gamma|R_{\mu})\delta_{\gamma}$ on $\Gamma$.
Then, up to a subsequence, $\tilde{\lambda}_N:=\sum_{n=1}^N\lambda_n/\left (\sum_{n=1}^N\lambda_n(\Gamma)\right )$ converges weakly to a probability measure on $\partial \hat{\Gamma}$, which we denote by $\lambda_{R_{\mu}}$.

\textit{Step 3:} The limit measure $\lambda_{R_{\mu}}$ is ergodic for the action of $\Gamma$ on $\partial \hat{\Gamma}$.

\textit{Step 4:} Let $c_j=\int \varphi_rdm_j$.
Then, $\lambda_{R_{\mu}}(U(c_j))>0$.
Since $\lambda_{R_{\mu}}$ is ergodic and $U(c_i)$ is $\Gamma$-invariant, we thus have $\lambda_{R_{\mu}}(U(c_j))=1$ for all $j$.
In particular, all the sets $U(c_j)$ intersect, which proves that $c_j$ is independent of $j$.

\medskip
\textit{Step 1} in \cite{Gouezel1} is stated with the Gromov boundary $\partial \Gamma$ of $\Gamma$ instead of $\partial \hat{\Gamma}$, since groups are hyperbolic in there and not relatively hyperbolic.
It is a consequence of the fact that geodesics converging to $\xi\in \partial \Gamma$ in a hyperbolic group stay within a bounded distance of each other.
This property still holds in our situation as we now explain.

First, let us show that the definition of $U(c)$ does not depend on the choice of the sequence $x$ defining $\xi$.
Assume that $x$ and $x'$ are two sequences such that, setting $\gamma_n=x_1...x_n$ and $\gamma'_n=x_1'...x_n'$, both sequences $e,\gamma_1,...,\gamma_n,...$ and $e,\gamma_1',...,\gamma_n',...$ are relative geodesics converging to $\xi$.
Then, according to Lemma~\ref{Proposition315Osininfinite}, for every $n$, there exists $k_n$ such that $d(\gamma_n,\gamma_{k_n}')\leq C$, so that
$H(e,\gamma_n|r)\asymp H(e,\gamma_{k_n}'|r)$.
We thus have $|\log H(e,\gamma_n|r)-\log H(e,\gamma_{k_n}'|r)|\leq C'$.
Moreover, since $d(\gamma_n,\gamma_{k_n}')\leq C$, $\hat{d}(\gamma_n,\gamma_{k_n}')\leq C$, so that $|n-k_n|\leq C''$ and thus $|\hat{d}(e,\gamma_n')-\hat{d}(e,\gamma_{k_n}')|\leq C''$.
This proves that $\log H(e,\gamma_n|r)/\hat{d}(e,\gamma_n)$ and $\log H(e,\gamma_n'|r)/\hat{d}(e,\gamma_n')$ have the same limit.

Let $\gamma\in \Gamma$ and let $\xi \in U(c)$.
We want to prove that $\gamma \cdot \xi \in U(c)$.
Consider a sequence $x$ defining $\xi$ and a sequence $x'$ defining $\gamma \cdot \xi$, that is, setting $\gamma_n=x_1...x_n$ and $\gamma'_n=x_1'...x_n'$, the sequence $e,\gamma_1,...,\gamma_n,...$ is a relative geodesic converging to $\xi$ and the sequence $e,\gamma_1',...,\gamma_n',...$ is a relative geodesic converging to $\gamma \cdot \xi$.
Then, $\gamma,\gamma\gamma_1,...,\gamma\gamma_n,...$ is a relative geodesic starting at $\gamma$ and converging to $\gamma \cdot \xi$.
According to Lemma~\ref{Proposition315Osininfinite}, for every $n$, there exists $k_n$ such that $d(\gamma \gamma_n,\gamma_{k_n}')\leq C$.
This time, the bound $C$ depends on $\gamma$, but not on the sequences $\gamma_n$ and $\gamma'_n$.
This shows that $H(\gamma^{-1},\gamma_n|r)=H(e,\gamma \gamma_n|r)\asymp H(e,\gamma_{k_n}'|r)$, and since $\gamma$ is fixed,
$H(\gamma^{-1},\gamma_n|r)\asymp H(e,\gamma_n|r)$, so that
$H(e,\gamma_n|r) \asymp H(e,\gamma_{k_n}'|r)$.
The same proof then shows that $\log H(e,\gamma_n|r)/\hat{d}(e,\gamma_n)$ and $\log H(e,\gamma_n'|r)/\hat{d}(e,\gamma_n')$ have the same limit.

\medskip
\textit{Step 2} follows directly from the convergence properties of the transfer operator $\mathcal{L}_r$ in \cite{Gouezel1} and one does not need to extract a subsequence.
However, in our situation, we can only prove convergence of $\int f d\lambda_n$ for functions $f\in \mathcal{B}_{\rho,\beta}$.
Since our space is not compact and not even locally compact, this set of function is not dense in the set of all continuous and bounded functions for the $\|\cdot\|_{\infty}$ norm.
To fix this problem, we need to consider a compact space that contains $\partial \hat{\Gamma}$ so that $m_N$ converges to a measure on this compact space (up to a subsequence).
We then prove that this limit measure gives full measure to $\partial \hat{\Gamma}$.
The compact space in question is a version of the Martin boundary that we will define.
Actually, we will both deal with the Martin boundary and the Bowditch boundary at the same time.

We first define the Green distance at the spectral radius as
$$d_G(\gamma,\gamma')=-\log F(\gamma,\gamma'|R_{\mu})F(\gamma',\gamma|R_{\mu}),$$ where
$F(\gamma,\gamma'|R_{\mu})$ is the first visit Green function at $R_{\mu}$.
More precisely, we have
\begin{equation}\label{defF}
    F(\gamma,\gamma'|r)=\sum_{n\geq 0}r^n \mathbb{P}(X_0=\gamma,X_n=\gamma',X_k\neq \gamma', 1\leq k \leq n-1),
\end{equation}
where $X_k$ is the position of the $\mu$-random walk at time $k$.
Note that for $r=1$, $F(\gamma,\gamma'|1)$ is the probability of ever reaching $\gamma'$ starting at $\gamma$.

Using the relation
$$G(\gamma,\gamma'|r)=F(\gamma,\gamma'|r)G(\gamma',\gamma'|r)=F(\gamma,\gamma'|r)G(e,e|r)$$
(see \cite[Lemma~1.13.(b)]{Woess}), we also have that
$$d_G(\gamma,\gamma')=-\log G(\gamma,\gamma'|R_{\mu})-\log G(\gamma',\gamma|R_{\mu})+2G(e,e|R_{\mu}).$$
Actually, the Green distance was introduced by S.~Blach\`ere and S.~Brofferio in \cite{BlachereBrofferio} as
$d_G(\gamma,\gamma')=-\log F(\gamma,\gamma'|1)$.
What we call the Green distance here is thus a symmetrized version at the spectral radius of what they call the Green distance.

In general, in any metric space $(X,d)$, one can consider a compactification given by the distance called the horofunction compactification.
It was introduced by C.~Kuratowski in \cite{Kuratowski} and used a lot by M.~Gromov, see for example \cite{BallmannGromovSchroeder}.
It is the smallest compact set $H$ such that the function $\phi:(x,y)\mapsto d(x,y)-d(x_0,y)$ extends continuously to $X\times H$, where $x_0$ is a base point.
Its homeomorphism type does not depend on $x_0$.
The horofunction boundary is the complement of $\Gamma$ in the horofunction compactification.
We refer to \cite[Section~3]{MaherTiozzo} for a construction and many more details.

Define the Martin kernel as
$$\tilde{K}(\gamma,\gamma')=\frac{G(\gamma,\gamma'|R_{\mu})G(\gamma',\gamma|R_{\mu})}{G(e,\gamma'|R_{\mu})G(\gamma',e|R_{\mu})}.$$
The Martin compactification is defined as the horofunction compactification for the Green distance.
In other words, a sequence $\gamma_n$ in $\Gamma$ converges to a point $\xi$ in the Martin boundary if and only if the Martin kernel $\tilde{K}(\cdot,\gamma_n)$ converge pointwise to a limit function $\tilde{K}(\cdot,\xi)$.
Usually, the Martin compactification is defined using the Martin kernel
$$K(\gamma,\gamma')=\frac{G(\gamma,\gamma')}{G(e,\gamma')}.$$
Again, our Martin compactification is a symmetrized version of the usual Martin compactification.

It is proved in \cite{GGPY} that as soon as weak relative Ancona inequalities are satisfied, there is a one-to-one continuous map from $\Gamma \cup \partial \hat{\Gamma}$ to the Martin compactification, which is a homeomorphism on its image.
Actually, this is proved for the usual definition of the Martin boundary.
Although the proof still works for our symmetrized version, the terminology is a bit different and we give a proof for completeness.

\begin{lemma}\label{homeoconicalintoMartin}
There is a one-to-one continuous map from $\Gamma \cup \partial \hat{\Gamma}$ to the Martin compactification, which is a homeomorphism on its image.
\end{lemma}

\begin{proof}
Let $\xi$ be a conical limit point and let $[e,\xi)$ be a relative geodesic ray from $e$ to $\xi$.
Let $\gamma_n$ be a sequence along $[e,\xi)$ converging to $\xi$.
Let $\gamma \in \Gamma$ and let $\tilde{\gamma}$ be its projection on $[e,\xi)$ in $\hat{\Gamma}$.
Lemma~\ref{lemmarelativetripod} shows that for large enough $n$, a relative geodesic from $\gamma$ to $\gamma_n$ passes within a bounded distance of $\tilde{\gamma}$.
Also, \cite[Lemma~4.17]{DussauleLLT1} shows that for large enough $n$, relative geodesics from $e$ to $\gamma_n$ and from $\gamma$ to $\gamma_n$ fellow travel for an arbitrarily long time, when $n$ goes to infinity.
Then, strong relative Ancona inequalities show that for every $\gamma$, $\tilde{K}(\gamma,\gamma_n)$ converges to some limit $\tilde{K}_{\xi}(\gamma)$, exactly as in the proof of Lemma~\ref{GreenlocallyHolder}.
We thus proved that $\gamma_n$ converges to a limit that we still denote by $\xi$ in the Martin boundary.

More generally, let $\xi$ be a conical limit point and let $\xi_n$ be a sequence in $\Gamma \cup \partial\hat{\Gamma}$ converging to $\xi$.
Let $\alpha$ be a relative geodesic ray from $e$ to $\xi$ and let $\alpha_n$ be a (finite or infinite) relative geodesic from $e$ to $\xi_n$.
Let $d_\mu$ be an arbitrary distance on the Martin compactification.
Then, there exists $\gamma_n\in \Gamma$ on $\alpha_n$ such that $d_{\mu}(\gamma_n,\xi_n)\leq \frac{1}{n}$.
If $\xi_n\in \Gamma$, we can choose $\xi_n=\gamma_n$.
Otherwise, we use what we just proved above.
Up to choosing $\hat{d}(e,\gamma_n)$ large enough, we can also assume that $\gamma_n$ converges to $\xi$ in $\Gamma \cup \partial \hat{\Gamma}$.
Thus, there exists a sequence $k_n$ going to infinity such that the projection $\tilde{\gamma}_n$ of $\gamma_n$ on $\alpha$ in $\hat{\Gamma}$ satisfies $\hat{d}(e,\tilde{\gamma}_n)\geq k_n$.
In particular, $\tilde{\gamma}_n$ converges to $\xi$ in the Martin boundary,
that is, for any $\gamma$, $\tilde{K}(\gamma,\tilde{\gamma}_n)$ converges to $\tilde{K}_{\xi}(\gamma)$.
Let $\gamma \in \Gamma$.
Then, according to \cite[Lemma~4.17]{DussauleLLT1} applied twice, relative geodesics from $e$ to $\gamma_n$ and from $\gamma$ to $\tilde{\gamma}_n$ fellow travel for an arbitrarily long time, when $n$ goes to infinity.
Strong relative Ancona inequalities show that $\tilde{K}(\gamma,\gamma_n)$ also converges to $\tilde{K}_{\xi}(\gamma)$.
Thus, $d_{\mu}(\xi,\gamma_n)$ goes to 0.
Since $d_{\mu}(\gamma_n,\xi_n)\leq \frac{1}{n}$, we also have that $\xi_n$ converges to $\xi$ in the Martin boundary.

\medskip
We thus constructed a map from $\Gamma \cup \partial \hat{\Gamma}$ to the Martin compactification.
We also proved that this map is continuous.
Let us prove that it is one-to-one.
Let $\xi\neq \xi'$ be two conical limit points.
We just need to prove that $\xi\neq \xi'$ in the Martin boundary.
Consider two relative geodesics $[e,\xi)$ and $[e,\xi')$ from $e$ to $\xi$ and from $e$ to $\xi'$.
Let $\gamma_n$, respectively $\gamma_n'$ be a sequence on $[e,\xi)$, respectively $[e,\xi')$, converging to $\xi$, respectively $\xi'$.
Since $\xi\neq \xi'$, the projection of $\gamma_n$ on $[e,\xi')$ in $\hat{\Gamma}$ stays within a bounded distance of $e$.
Thus, for large enough $n$ and $m$, a relative geodesic from $\gamma_n$ to $\gamma_m'$ passes within a bounded distance of $e$.
Weak relative Ancona inequalities show that
$\tilde{K}(\gamma_n,\gamma_m')\asymp H(\gamma_n,e|R_\mu)$.
Letting $m$ tend to infinity, we thus have that
$\tilde{K}_{\xi'}(\gamma_n)\asymp H(\gamma_n,e|R_\mu)$, so that $\tilde{K}_{\xi'}(\gamma_n)$ converges to 0.
Weak relative Ancona inequalities also show that if $n<m$, we have
$\tilde{K}(\gamma_n,\gamma_m)\asymp \frac{1}{H(e,\gamma_n|R_\mu)}$.
Letting $m$ tend to infinity, we get
$\tilde{K}_{\xi}(\gamma_n)\asymp \frac{1}{H(\gamma_n,e|R_\mu)}$, so that $\tilde{K}_{\xi}(\gamma_n)$ goes to infinity.
We can thus find $n$ such that $\tilde{K}_{\xi}(\gamma_n)\neq \tilde{K}_{\xi'}(\gamma_n)$ and so $\xi\neq \xi'$ in the Martin boundary.

Finally, we prove that this map is a homeomorphism on its image.
Let $\xi_n$ be a sequence in $\Gamma \cup \partial \hat{\Gamma}$ converging to $\xi$ in the Martin compactification.
Assume by contradiction that it does not converge to $\xi$ in $\Gamma \cup \partial \hat{\Gamma}$.
Fix a relative geodesic $\alpha$ from $e$ to $\xi$ and for every $n$, a relative geodesic $\alpha_n$ from $e$ to $\xi_n$.
Then, up to choosing a subsequence, we can assume that the projection of $\alpha_n$ on $\alpha$ in $\hat{\Gamma}$ stays within a uniform bounded distance of $e$.
In particular, if $\gamma_m$ is a sequence on $\alpha$ converging to $\xi$ and if $\gamma_k'$ is a sequence on $\alpha_n$ converging to $\xi_n$, then a relative geodesic from $\gamma_m$ to $\gamma_k'$ passes within a bounded distance of $e$, independently of $k,m,n$.
Weak relative Ancona inequalities show that $\tilde{K}(\gamma_m,\gamma_k')\asymp H(\gamma_m,e|R_{\mu})$ so that letting $k$ tend to infinity,
$\tilde{K}_{\xi_n}(\gamma_m)\asymp H(\gamma_m,e|R_{\mu})$.
In particular, $\tilde{K}_{\xi_n}(\gamma_m)\leq C$ for some uniform $C$.
However, as we saw above, $\tilde{K}_{\xi}(\gamma_m)$ tends to infinity, so there exists $m$ such that $\tilde{K}_{\xi}(\gamma_m)\geq C+1$.
Fixing such an $m$, we obtain a contradiction, since $\tilde{K}_{\xi_n}(\gamma_m)$ converges to $\tilde{K}_{\xi}(\gamma_m)$ when $n$ tends to infinity.
\end{proof}

We first prove that $\tilde{\lambda}_N$ converges to a probability measure on the Bowditch compactification.
We will then prove that it also converges to a probability measure on the Martin compactification.

\begin{proposition}\label{convergencePattersonSullivanlike}
Up to a subsequence, the measure $\tilde{\lambda}_N$ weakly converges to a measure $\lambda_{R_\mu}$ on the Bowditch compactification.
This limit measure gives full measure to the set of conical limit points.
\end{proposition}

\begin{proof}
Convergence up to a subsequence follows directly from compactness of the Bowditch compactification.
We just need to prove that any limit measure of $\tilde{\lambda}_N$ gives full measure to the set of conical limit points.
Recall that by definition, $\lambda_n=\sum_{\gamma \in \hat{S}_n}H(e,\gamma|R_{\mu})\delta_{\gamma}$ and
$\tilde{\lambda}_N=\sum_{n=1}^N\lambda_n/\left (\sum_{n=1}^N\lambda_n(\Gamma)\right ).$

First, we prove that any limit measure $\lambda_{R_{\mu}}$ of $\tilde{\lambda}_N$ gives full mass to the Bowditch boundary.
Let $K\subset \Gamma$ be a compact subset.
Then, $K$ is finite, so that for any $N$,
$\sum_{n=1}^N\lambda_n(K)$ is bounded, independently of $N$.
Moreover, according to Proposition~\ref{relationR_kG'(R)}, $\sum_{n=1}^N\lambda_n(\Gamma)$ tends to infinity.
This proves that for any subsequence $\tilde{\lambda}_{N_j}$ of $\tilde{\lambda}_N$, $\tilde{\lambda}_{N_j}(K)$ converges to 0 when $j$ tends to infinity.
Since $K$ is both open and closed in the Bowditch compactification, the Portmanteau Theorem shows that $\lambda_{R_{\mu}}(K)=0$.

Consider a parabolic limit point $\xi$ in the Bowditch boundary.
Since the set of parabolic limit points is countable, we just need to prove that $\lambda_{R_{\mu}}(\{\xi\})=0$ to conclude.
Let $\mathcal{H}$ be the corresponding parabolic subgroup, that is $\mathcal{H}$ is the stabilizer of $\xi$.
Choose $\mathcal{H}_0\in \Omega_0$ so that $\mathcal{H}$ is conjugated to $\mathcal{H}_0$, say $\mathcal{H}=\gamma_0\mathcal{H}_0\gamma_0^{-1}$
Denote by $U_{\xi,n}$ the set of $\gamma \in \Gamma$ such that the projection of $\gamma$ on $\gamma_0 \mathcal{H}_0$ in the Cayley graph $\Cay(\Gamma,S)$ is at $d$-distance at least $n$ from $e$.
Let $V(\xi,n)$ be the closure of $U(\xi,n)$ in the Bowditch compactification.
Then, $V(\xi,n)$ contains $\{\xi\}$ so we only need to prove that $\lambda_{R_{\mu}}(V(\xi,n))$ converges to 0 when $n$ tends to infinity.

According to the BCP property, if $\zeta \in V(\xi,n)$, then there exists $\gamma\in \mathcal{H}_0$ such that $\gamma_0\gamma$ is within a bounded distance of a relative geodesic from $e$ to $\zeta$.
In particular, if $N$ is large enough,
for every $\gamma'\in V(\xi,n)\cap \Gamma$ such that $\hat{d}(e,\gamma')=N$, weak relative Ancona inequalities show that
$$H(e,\gamma'|R_\mu)\lesssim H(e,\gamma_0|R_\mu)H(e,\gamma|R_\mu)H(\gamma_0\gamma,\gamma'|R_\mu).$$
Thus,
$$\lambda_N(V(\xi,n))\lesssim \sum_{\gamma \in \mathcal{H}_0,d(e,\gamma_0\gamma)\geq n}H(e,\gamma|R_{\mu})\sum_{k=1}^N\lambda_k(\Gamma)$$
and so
$$\tilde{\lambda}_N(V(\xi,n))\lesssim \sum_{\gamma \in \mathcal{H}_0,d(e,\gamma_0\gamma)\geq n}H(e,\gamma|R_{\mu}).$$
Since $\gamma_0$ is fixed, this proves that
$$\lambda_{R_{\mu}}(V(\xi,n))\lesssim \sum_{\gamma \in \mathcal{H}_0,d(e,\gamma)\geq n}H(e,\gamma|R_{\mu}).$$
According to Corollary~\ref{derivativeparabolicGreenfinite}, this last term converges to 0 when $n$ tends to infinty, which concludes the proof.
\end{proof}

We can thus see the measure $\lambda_{R_\mu}$ as a measure on $\partial \hat{\Gamma}$.
We fix a subsequence $\tilde{\lambda}_{N_k}$ such that $\tilde{\lambda}_{N_k}$ weakly converges to $\lambda_{R_\mu}$.
We also prove the following.

\begin{lemma}\label{convergencePattersonSullivanlikeinMartin}
The measure $\tilde{\lambda}_{N_k}$ also weakly converges to $\lambda_{R_\mu}$ on $\Gamma\cup \partial \hat{\Gamma}$ and on the Martin compactification.
\end{lemma}

\begin{proof}
There is a one-to-one and continuous map from $\Gamma \cup \partial \hat{\Gamma}$ to the Bowditch compactification, which is a homeomorphism on its image.
Since the limit measure $\lambda_{R_\mu}$ does not give any mass to the complement of $\Gamma \cup \partial\hat{\Gamma}$, the Portmanteau Theorem shows that $\tilde{\lambda}_{N_k}$ also weakly converges to $\lambda_{R_\mu}$ on $\Gamma\cup \partial\hat{\Gamma}$.

There is also a one-to-one and continuous map from $\Gamma \cup \partial \hat{\Gamma}$ to the Martin compactification, which is a homeomorphism on its image.
Let $f$ be a bounded continuous function on the Martin compactification.
Its restriction $\tilde{f}$ to $\Gamma \cup \partial\hat{\Gamma}$ also is bounded and continuous, so that $\tilde{\lambda}_{N_k}(f)=\tilde{\lambda}_{N_k}(\tilde{f})$ converges to $\lambda_{R_\mu}(\tilde{f})$.
This proves that $\tilde{\lambda}_{N_k}$ also weakly converges to $\lambda_{R_\mu}$ on the Martin compactification.
\end{proof}

\medskip
\textit{Step 3} is a bit more complicated.
Recall that we need to prove that $\lambda_{R_\mu}$ is ergodic for the action of $\Gamma$ on $\partial \hat{\Gamma}$.
We follow the strategy of \cite{MatsuzakiYabukiJaerisch}.
We first prove that $\lambda_{R_\mu}$ is conformal for the Green distance defined above.
%
Let $\gamma\in \Gamma$ and let $L_{\gamma}$ the operator of multiplication by $\gamma$ on the left.
\begin{lemma}\label{conformality}
For every $\gamma$, we have
$$\frac{d(L_{\gamma})_*\lambda_{R_{\mu}}}{d\lambda_{R_{\mu}}}(\xi)=\tilde{K}_{\xi}(\gamma).$$
\end{lemma}

\begin{proof}
Direct computation shows that for fixed $\gamma_0$, one has, for any $\gamma$ and any $N$ such that
$N\geq \hat{d}(e,\gamma_0)+\hat{d}(e,\gamma)$,
\begin{equation}\label{conformality1}
    (L_{\gamma_0})_*\tilde{\lambda}_N(\gamma)=H(\gamma_0,\gamma|R_{\mu})/H(e,\gamma|R_{\mu})\tilde{\lambda}_N(\gamma)=\tilde{K}(\gamma_0,\gamma)\tilde{\lambda}_N(\gamma).
\end{equation}
Lemma~\ref{finitesumalongspheres} shows that $\sum_{n=N-\hat{d}(e,\gamma_0)}^{N}\sum_{\gamma \in \hat{S}_n}H(e,\gamma|R_{\mu})$ is bounded.
Thus, according to Proposition~\ref{relationR_kG'(R)},
$$\frac{\sum_{n=N-\hat{d}(e,\gamma_0)}^{N}H(e,\gamma|R_{\mu})}{\sum_{n\leq N}H(e,\gamma|R_{\mu})}\underset{N\to \infty}{\longrightarrow} 0.$$
Combined with~(\ref{conformality1}), this shows that $\left ((L_{\gamma_0})_*\tilde{\lambda}_N-\tilde{K}(\gamma_0,\cdot)\tilde{\lambda}_N\right )$ converges to 0 in total variation norm.

By definition, for fixed $\gamma_0$, the function $\tilde{K}(\gamma_0,\cdot)$ is continuous and bounded on the Martin compactification.
Thus, $\tilde{K}(\gamma_0,\cdot)\tilde{\lambda}_{N_k}$ weakly converges to $\tilde{K}(\gamma_0,\cdot)\lambda_{R_\mu}$.
Moreover, left multiplication by $\gamma_0$ on $\Gamma$ extends to a homeomorphim on the Martin compactification,
so that $(L_{\gamma_0})_*\tilde{\lambda}_{N_k}$ weakly converges to $(L_{\gamma_0})_*\lambda_{R_\mu}$.
We thus proved that $(L_{\gamma_0})_*\lambda_{R_\mu}=\tilde{K}(\gamma_0,\cdot)\lambda_{R_\mu}$.
\end{proof}

We use this property to prove the following.
\begin{lemma}\label{conformalwithoutatoms}
The measure $\lambda_{R_\mu}$ on $\partial \hat{\Gamma}$ has no atom.
\end{lemma}

\begin{proof}
Assume on the contrary that there exists $\xi\in \partial \hat{\Gamma}$ such that $\lambda_{R_\mu}(\xi)>0$.
Consider a sequence $\gamma_n$ converging along a relative geodesic ray to $\xi$.
Then, weak relative Ancona inequalities show that if $n\leq m$, then
$\tilde{K}(\gamma_n,\gamma_m)\geq \frac{C}{H(e,\gamma_n|R_{\mu})}$, so that letting $m$ tend to infinity,
$\tilde{K}_{\xi}(\gamma_n)\geq \frac{C}{H(e,\gamma_n|R_{\mu})}$.
Since $d(\gamma_n,e)$ tends to infinity, $H(e,\gamma_n|R_{\mu})$ converges to 0 and so
$\tilde{K}_{\xi}(\gamma_n)$ tends to infinity.
Lemma~\ref{conformality} shows that
$\lambda_{R_\mu}(\gamma_n^{-1}\xi)=\tilde{K}_{\xi}(\gamma_n)\lambda_{R_{\mu}}(\xi)$, which goes to infinity.
This is a contradiction, since $\lambda_{R_{\mu}}$ is a probability measure.
\end{proof}



In the following, it will be simpler to see the measure $\lambda_{R_\mu}$ as a measure on the Bowditch boundary that gives full mass to the set of conical limit points.
In the hyperbolic setting, using results of M.~Coornaert (see \cite{Coornaert}), conformal measures for hyperbolic distances are ergodic.
Actually, \cite{Coornaert} only deals with geodesic distances and this was generalized by \cite{BlachereHassinskyMathieu} for distances that are hyperbolic and quasi-isometric to a word distance, such as the Green distance as long as weak Ancona inequalities hold (this is also proved in \cite{BlachereHassinskyMathieu}).
Comparing a geodesic distance with the Green distance is more difficult here and we need another approach.
With the same strategy as in \cite[Theorem~4.1]{MatsuzakiYabukiJaerisch}, we prove the following generalization of M.~Coornaert's result.

\begin{proposition}\label{ergodicityconformal}
The measure $\lambda_{R_{\mu}}$ is ergodic for the action of $\Gamma$ on $\partial \hat{\Gamma}$.
\end{proposition}

Before proving this proposition, let us introduce some notions of geometric measure theory from \cite{MatsuzakiYabukiJaerisch} and some constructions of \cite{GekhtmanDussaule} and \cite{Yang2}.
Let $\Lambda$ be a metric space. A covering relation $\mathcal{C}$ is a subset of the set of all pairs $(\xi,S)$ such that $\xi \in S\subset \Lambda$. A covering relation $\mathcal{C}$ is said to be fine at $\xi \in \Lambda$ if there exists a sequence $S_n$ of subsets of $\Lambda$ with $(\xi,S_n)\in \mathcal{C}$ and such that the diameter of $S_n$ converges to $0$.
Let $\mathcal{C}$ be a covering relation.
For any measurable subset $E\subset \Lambda$, define $\mathcal{C}(E)$ to be the collection of subsets $S\subset \Lambda$ such that $(\xi,S)\in \mathcal{C}$ for some $\xi \in E$. A covering relation $\mathcal{C}$ is said to be a Vitali relation for a finite measure $\kappa$ on $\Lambda$ if it is fine at every point of $\Lambda$ and if the following holds: 
if $\mathcal{C}'\subset \mathcal{C}$ is fine at every point of $\Lambda$ then for every measurable subset $E$, $\mathcal{C}'(E)$ has a countable disjoint subfamily $\{S_n\}$ such that $\kappa(E\setminus \cup^\infty_{n=1}S_{n})=0$.
We will use the letter $\mathcal{V}$ to denote a Vitali relation in the following.

Recall that an $(\eta_1,\eta_2)$-transition point on a geodesic $\alpha$ in the Cayley graph $\Cay(\Gamma,S)$ is a point $\gamma$ such that for any coset $\gamma_0\mathcal{H}$ of a parabolic subgroup, the part of $\alpha$ consisting of points at distance at most $\eta_2$ from $\gamma$ is not contained in the $\eta_1$-neighborhood of $\gamma_0\mathcal{H}$.
Let $\xi$ be a conical limit point.
Following W.~Yang \cite{Yang2}, the partial shadow $\Omega_{\eta_1,\eta_2}(\gamma)$ at $\gamma \in \Gamma$ is the set of points $\xi$ in the Bowditch boundary such that there is a geodesic ray $[e,\xi)$ in $\Cay(\Gamma,S)$ containing an $(\eta_1,\eta_2)$-transition point in the ball $B(\gamma,2\eta_2)$.

We define the following relation $\mathcal{V}_{\eta_1,\eta_2}$ on the Bowditch boundary.
For $\xi$ parabolic, we declare $(\xi,\{\xi\})\in \mathcal{V}_{\eta_1,\eta_2}$.
For $\xi$ conical, we declare $(\xi, \Omega_{\eta_1,\eta_2}(\gamma))\in \mathcal{V}_{\eta_1,\eta_2}$ whenever $\xi \in  \Omega_{\eta_1, \eta_2}(\gamma)$.
According to \cite[Proposition~3.3]{GekhtmanDussaule}, the relation $\mathcal{V}_{\eta_1,\eta_2}$ is fine at every limit point in the Bowditch boundary.

Let $\gamma\in \Gamma$ and let $\eta_1,\eta_2>0$.
Consider a neighborhood $U$ of $\Omega_{\eta_1,\eta_2}(\gamma)$ in the Bowditch compactification.
One can choose $U$ such that for any point $\xi$ in $U$, $\gamma$ is within a bounded distance of a transition point on a geodesic from $e$ to $\xi$.
According to Lemma~\ref{projectiontransitionpoints}, $\gamma$ is within a bounded distance of a point on a relative geodesic from $e$ to $\xi$.
In particular, weak relative Ancona inequalities imply that there exists a constant $C$ (depending on $\eta_2$) such that for $N$ large enough,
$$\tilde{\lambda}_N(\Omega_{\eta_1,2\eta_2}(\gamma))\leq C \tilde{\lambda}_N(\Omega_{\eta_1,\eta_2}(\gamma)),$$
so that
$$\lambda_{R_{\mu}}(\Omega_{\eta_1,2\eta_2}(\gamma))\leq C \lambda_{R_{\mu}}(\Omega_{\eta_1,\eta_2}(\gamma)).$$
The measure $\lambda_{R_{\mu}}$ on the Bowditch boundary gives full measure to the set of conical limit points.
Thus, \cite[Proposition~3.4]{GekhtmanDussaule} shows that the relation $\mathcal{V}_{\eta_1,\eta_2}$ is a Vitali relation for $\lambda_{R_{\mu}}$.

To prove Proposition~\ref{ergodicityconformal}, we will need the following two results.
\begin{lemma}\label{Theorem4.2MYJ}\cite[Theorem~4.2]{MatsuzakiYabukiJaerisch}
Let $E$ be a measurable subset of the Bowditch boundary.
Then, for $\lambda_{R_{\mu}}$-almost every point $\xi \in E$, one has
$$\frac{\lambda_{R_{\mu}}(E\cap S_n)}{\lambda_{R_{\mu}}(S_n)}\underset{n\to \infty}{\longrightarrow}1$$
for every sequence $\{S_n\}$ such that $(\xi, S_n )\in  \mathcal{V}_{\eta_1,\eta_2}$ for all $n$ and such that the diameter of $S_n$ converges to 0 when $n$ tends to infinity.
\end{lemma}

For the second result, we need to choose a distance on the Bowditch boundary.
To simplify the argument, we choose the shortcut distance, so that the following holds.
We refer to \cite[Section~2.4]{Yang2} for the definition of the shortcut distance.

\begin{lemma}\label{Proposition2.10MYJ}
For every $\epsilon>0$ and $\eta_1>0$, there exists $\eta_2>0$ such that for every $\gamma\in \Gamma$, the diameter of the complement of $\gamma^{-1}\Omega_{\eta_1,\eta_2}(\gamma)$ is smaller than $\epsilon$.
\end{lemma}

\begin{proof}
This is exactly the content of the remark after the claim inside the proof of \cite[Lemma~4.1]{Yang2}.
\end{proof}

Actually, the choice of the distance is not relevant and with a bit of work, one could have proved the same result for a visual distance on the Bowditch boundary, adapting the arguments of \cite[Proposition~2.10]{MatsuzakiYabukiJaerisch}.
We only chose the shortcut distance to avoid reproving this technical claim.

We can finally prove Proposition~\ref{ergodicityconformal}.
Everything is settled so that we only have to rewrite the proof of \cite[Theorem~4.1]{MatsuzakiYabukiJaerisch}.
We still rewrite it for convenience.
\begin{proof}
Denote by $\partial_B\Gamma$ the Bowditch boundary of $\Gamma$.
Consider a $\Gamma$-invariant measurable subset $E$ of the Bowditch boundary, such that $\lambda_{R_{\mu}}(E)>0$.
Assume by contradiction that $\lambda_{R_{\mu}}(E)<1$.
We fix $\epsilon>0$ arbitrarily small.
For technical reasons, we assume that $\lambda_{R_{\mu}}(\partial_B\Gamma)\geq 2\epsilon$, that is, $\epsilon\leq 1/2$.

According to Lemma~\ref{Theorem4.2MYJ}, if $\eta_2$ is large enough, for $\lambda_{R_{\mu}}$-almost every $\xi$ in $E^c$,
$$\frac{\lambda_{R_{\mu}}(E\cap \Omega_{\eta_1,\eta_2}(\gamma_n))}{\lambda_{R_{\mu}}(\Omega_{\eta_1,\eta_2}(\gamma_n))}\underset{n\to \infty}{\longrightarrow}0,$$
whenever $\gamma_n$ converges to $\xi$ along a relative geodesic ray.
Take such a $\xi$ and such a sequence $\gamma_n$.
Up to taking a subsequence, we can assume that $\gamma_n^{-1}$ converges to a point $\zeta$ in the Bowditch boundary.
According to Lemma~\ref{conformalwithoutatoms}, $\lambda_{R_{\mu}}(\zeta)=0$.

Then, since the Bowditch boundary is compact, there exists $\delta>0$ (not depending on $\zeta$) such that the ball centered at $\zeta$ of radius $\delta$ has measure at most $\epsilon$.
Moreover, according to Lemma~\ref{Proposition2.10MYJ}, there exists $\eta_2>0$ such that the diameter of the complement of $\gamma_n^{-1}\Omega_{\eta_1,\eta_2}(\gamma_n)$ is smaller than $\delta$.
Fixing such an $\eta_2>0$, for large enough $n$, we have that $\zeta\notin \gamma_n^{-1}\Omega_{\eta_1,\eta_2}(\gamma_n)$, so that the complement of $\gamma_n^{-1}\Omega_{\eta_1,\eta_2}(\gamma_n)$ is contained in the ball of center $\zeta$ and of radius $\delta$.
In particular,
$$\lambda_{R_{\mu}}(\partial_B\Gamma \setminus \gamma_n^{-1}\Omega_{\eta_1,\eta_2}(\gamma_n))\leq \epsilon.$$
Since $\lambda_{R_{\mu}}(\partial_B\Gamma)\geq 2\epsilon$, we thus also have
$$\lambda_{R_{\mu}}(\gamma_n^{-1}\Omega_{\eta_1,\eta_2}(\gamma_n))\geq \epsilon.$$
Since $E$ is $\Gamma$-invariant, we have
$$\lambda_{R_{\mu}}(E\cap \gamma_n^{-1}\Omega_{\eta_1,\eta_2}(\gamma_n))=(L_{\gamma_n})_*\lambda_{R_\mu}(E\cap \Omega_{\eta_1,\eta_2}(\gamma_n)).$$
Weak relative Ancona inequalities show that if $\xi'\in \Omega_{\eta_1,\eta_2}(\gamma_n)$, then
$$\frac{1}{C(\eta_2)}\frac{1}{H(e,\gamma_n|R_{\mu})}\leq \tilde{K}_{\xi'}(\gamma_n)\leq  C(\eta_2)\frac{1}{H(e,\gamma_n|R_{\mu})}.$$
In particular, Lemma~\ref{conformality} shows that
$$\lambda_{R_{\mu}}(E\cap \gamma_n^{-1}\Omega_{\eta_1,\eta_2}(\gamma_n))\leq C(\eta_2)\frac{1}{H(e,\gamma_n|R_{\mu})}\lambda_{R_\mu}(E\cap \Omega_{\eta_1,\eta_2}(\gamma_n)).$$
Similarly, we have
$$\lambda_{R_{\mu}}(\gamma_n^{-1}\Omega_{\eta_1,\eta_2}(\gamma_n))\geq \frac{1}{C(\eta_2)}\frac{1}{H(e,\gamma_n|R_{\mu})}\lambda_{R_\mu}(\Omega_{\eta_1,\eta_2}(\gamma_n)).$$
This proves that
$$\frac{\lambda_{R_{\mu}}(E\cap \gamma_n^{-1}\Omega_{\eta_1,\eta_2}(\gamma_n))}{\lambda_{R_{\mu}}(\gamma_n^{-1}\Omega_{\eta_1,\eta_2}(\gamma_n))}\leq C'(\eta_2)\frac{\lambda_{R_{\mu}}(E\cap \Omega_{\eta_1,\eta_2}(\gamma_n))}{\lambda_{R_{\mu}}(\Omega_{\eta_1,\eta_2}(\gamma_n))}.$$
The right-hand side of this last equation converges to 0 when $n$ tends to infinity.
Since $\lambda_{R_{\mu}}(\gamma_n^{-1}\Omega_{\eta_1,\eta_2}(\gamma_n))\geq \epsilon$, this proves that
$\lambda_{R_{\mu}}(E\cap \gamma_n^{-1}\Omega_{\eta_1,\eta_2}(\gamma_n)))$ converges to 0 when $n$ tends to infinity.
Finally, recall that $\lambda_{R_{\mu}}(\partial_B\Gamma \setminus \gamma_n^{-1}\Omega_{\eta_1,\eta_2}(\gamma_n))\leq \epsilon$, so that
$\lambda_{R_{\mu}}(E)\leq 2\epsilon$.
Since $\epsilon$ is arbitrarily small, we get that $\lambda_{R_{\mu}}(E)=0$, which is a contradiction.
\end{proof}

\textit{Finally, step 4} is a consequence of the two following lemmas.
We use the notation $\alpha_j=\sum_{i}\nu_j^{(i)}$, where the measures $\nu_j^{(i)}$ are given by Theorem~\ref{GouezelSarig}.

\begin{lemma}
The measure $T_*\alpha_j$ is absolutely continuous with respect to the measure $\alpha_j$.
\end{lemma}

\begin{proof}
Since by Lemma~\ref{pressureatthespectralradius} the maximal pressure at the spectral radius is zero, we have
$$\mathcal{L}_{R_{\mu}}^{*}\alpha_j=\alpha_j.$$
Denote by $[x_1,...,x_n]$ the cylinder consisting of elements of $\overline{\Sigma}_A$ starting with the symbols $x_1,...,x_n$.
Then, we have
$$\alpha_j([x_1,...,x_n])=\alpha_j(\mathcal{L}_{R_{\mu}}1_{[x_1,...,x_n]}).$$
Moreover, weak Ancona inequalities show that
$$\alpha_j(\mathcal{L}_{R_{\mu}}1_{[x_1,...,x_n]})\leq C H(e,x_1|R_{\mu})\alpha_j([x_2,...,x_n]).$$
Now, $T^{-1}[x_2,...,x_n]$ is contained in the union of cylinders of the form $[x_1,...,x_n]$, where $x_1\in S$ or $x_1\in \mathcal{H}$ for some parabolic subgroup.
According to Corollary~\ref{derivativeparabolicGreenfinite}, the sum
$\sum_{\sigma\in \mathcal{H}}H(e,\sigma|r)$ is uniformly bounded, so that
$$\alpha_j(T^{-1}[x_2,...,x_n])\leq C\alpha_j([x_2,...,x_n]).$$
This is true for any cylinder $[x_2,...,x_n]$.
It follows that for any measurable set $E\in \overline{\Sigma}_A$, $\alpha_j(T^{-1}E)\leq C\alpha_j(E)$.
\end{proof}

Recall that $\phi$ maps paths of $\overline{\Sigma}_A$ that start with $v_*$ to $\Gamma\cup \partial\hat{\Gamma}$.
Let $\alpha_j(\cdot \cap E_*)$ be the measure $\alpha_j$ restricted to paths that start at $v_*$.
Then, $\phi_*\alpha_j(\cdot \cap E_*)$ is a measure on $\partial \hat{\Gamma}$.

\begin{lemma}
The measure $\phi_*\alpha_j(\cdot \cap E_*)$ is absolutely continuous with respect to the measure $\lambda_{R_\mu}$.
\end{lemma}

\begin{proof}
The sequence of measures $\tilde{\lambda}_{N_k}$ weakly converges to $\lambda_{R_{\mu}}$ in the Bowditch compactification.
Recall that according to Lemma~\ref{convergencePattersonSullivanlikeinMartin}, it also weakly converges to $\lambda_{R_\mu}$ in $\Gamma \cup \partial \hat{\Gamma}$.

If $f$ is a function defined on $\partial \hat{\Gamma}$, then $f\circ \phi$ is defined on $\overline{\Sigma}_A$ and it vanishes on the complement of $E_*$.
We see $\partial\hat{\Gamma}$ as the Gromov boundary of the hyperbolic space $\hat{\Gamma}$ and we endow $\partial \hat{\Gamma}$ with a visual distance $d_v$, as in \cite{GhysHarpe}.
Then, there exists $\epsilon>0$ such that $d_v(\xi,\xi')\leq \mathrm{e}^{-\epsilon(\xi,\xi')_e}$, where $(\xi,\xi')_e$ is the Gromov product of $\xi$ and $\xi'$, based at $e$, see \cite[Proposition~7.3.10]{GhysHarpe}.
In particular, if $f$ is a bounded locally H\"older continuous function on $\partial\hat{\Gamma}$, then $f\circ \phi$ is in $\mathcal{B}_{\rho,\beta}$.
Theorem~\ref{GouezelSarig} shows that
$\mathcal{L}_{R_{\mu}}^{np+q}(f\circ \phi)(\emptyset)$ converges to $\sum_{j=1}^k\sum_{i=1}^{p_j}h_j^{(i)}\int (f\circ \phi) d\nu_j^{((i-n) \text{ mod } p_j)}$.
Also note that
$$\sum_{\gamma \in \hat{S}_n}H(e,\gamma|R_\mu)=H(e,e|R_\mu)\mathcal{L}_{R_{\mu}}^n(1_{E_*}f\circ \phi)(\emptyset).$$
Let $\beta_j=\phi_*\alpha_j(\cdot \cap E_*)$.
Since the functions $h_j^{(i)}$ are bounded away from zero and infinity on the support of $\nu_j^{(i)}$, this proves that for any bounded locally h\"older continuous function $f$ on $\partial \hat{\Gamma}$, we have
\begin{equation}\label{boundforholderfunctions}
    \beta_j(f)\leq C \lambda_{R_\mu}(f).
\end{equation}

Bounded H\"older continuous functions are not dense in bounded continuous functions for the supremum norm, but they are dense in the space of integrable functions for the $L^1$-norm, see \cite[Corollary~3.14]{AliprantisBorder}.
Let $A\subset \partial \hat{\Gamma}$ be any measurable set.
We want to prove that $\beta_j(A)\leq C \lambda_{R_{\mu}}(A)$.
Let $\epsilon>0$.
There exists a bounded locally H\"older continuous function $f$ such that
$$\|f-1_A\|_{L^1(\partial \hat{\Gamma},\beta_j+\lambda_{R_\mu})}\leq \epsilon.$$
Then,
$\beta_j(A)\leq \beta_j(|f|)+\epsilon$.
Since $|f|$ still is locally H\"older continuous,~(\ref{boundforholderfunctions}) shows that
$$\beta_j(A)\leq C\lambda_{R_\mu}(|f|)+\epsilon\leq C\lambda_{R_\mu}(A)+(1+C)\epsilon.$$
Since $\epsilon$ is arbitrary, this concludes the proof.
\end{proof}

Those two lemmas allow us to conclude step 4.
We only outline the proof and refer to the end of the proof of \cite[Proposition~3.16]{Gouezel1} for the details.
The probability measure $dm_j=\frac{1}{p_j}\sum_{i=1}^{p_j}h_{j}^{(i)}d\nu_j^{(i)}$ is invariant and ergodic for the shift $T$ (see for example \cite[Lemma~11]{Sarig1}).
Fix $r$ close enough to $R_{\mu}$ so that the conclusions of Theorem~\ref{KellerLiveranitransfert} hold.
Let $O_j$ be the set of points where the Birkhoff sums $\frac{1}{n}\sum_{1\leq k\leq n} \varphi_r\circ T^k$ converge to $c_j=\int \varphi_r dm_j$.
By the Birkhoff ergodic theorem, $m_j(O_j)=1$.
We first deduce that $\alpha_j(O_j\cap \overline{\Sigma}_{A,j})>0$.
Using that $T_*\alpha_j$ is absolutely continuous with respect to $\alpha_j$, we then deduce that $\alpha_j(O_j)>0$.
Then, using that $\phi_*\alpha_j(\cdot \cap E_*)$ is absolutely continuous with respect to $\lambda_{R_\mu}$, we deduce that $\lambda_{R_\mu}(\phi (O_j\cap E_*))>0$.
Finally, direct computation shows that $\phi (O_j\cap E_*)\subset U(c_j)$.
This proves that $\lambda_{R_\mu}(U(c_j))>0$.
Since $U(c_j)$ is invariant and $\lambda_{R_\mu}$ is ergodic, we thus have $\lambda_{R_\mu}(U(c_j))=1$.
This holds for all $j$, so we finally obtain that every $U(c_j)$ intersect, so that $c_j$ does not depend on $j$.

We can now conclude the proof of Proposition~\ref{pressureindependentcomponents}.
\begin{proof}
According to Proposition~\ref{prophypothesesKellerLiveranisatisfaites}, the assumptions of Proposition~\ref{theoremeperturbation} are satisfied, so that
\begin{equation}\label{asymptotiqueP(r)}
\begin{split}
    \mathrm{e}^{\tilde{P}_j(r)}-1= &\int \left (\mathrm{e}^{\varphi_r-\varphi_{R_\mu}}-1\right )dm_j\\
    &\hspace{1cm}+\int \left (\mathrm{e}^{\varphi_r-\varphi_{R_\mu}}-1\right )\frac{1}{p_j}\left (\sum_{i=0}^{p_j-1}h_j^{(i)}-\tilde{h}_{j,r}^{(i)}\right)d\nu_j^{(i)}.
\end{split}
\end{equation}

Lemma~\ref{corollary3.3Gouezel} and~(\ref{3.6Gouezel}) show that $|P(r)|$ has order of magnitude $\sqrt{R_\mu-r}$.
We will actually show that
$$\mathrm{e}^{\tilde{P}_j(r)}-1= \int \left (\varphi_r-\varphi_{R_\mu}\right )dm_j+o\left (\sqrt{R_\mu-r}\right ).$$

\begin{lemma}\label{premierterme=oracineR-r}
We have
$$\int \left (\mathrm{e}^{\varphi_r-\varphi_{R_\mu}}-1\right )\frac{1}{p_j}\left (\sum_{i=0}^{p_j-1}h_j^{(i)}-\tilde{h}_{j,r}^{(i)}\right)d\nu_j^{(i)}=o\left (\sqrt{R_\mu-r}\right ).$$
\end{lemma}

\begin{proof}
According to Theorem~\ref{GouezelSarig}, the functions $h_j^{(i)}$ are bounded away from zero and infinity on the support of $\nu_j^{(i)}$.
We can thus replace $\nu_j^{(i)}$ with $m_j$, which is itself dominated by the measure $m$.
Thus, we just need to show that for every $i$,
$$\frac{1}{\sqrt{R_\mu-r}}\int \left |\mathrm{e}^{\varphi_r-\varphi_{R_\mu}}-1\right |\left |h_j^{(i)}-\tilde{h}_{j,r}^{(i)}\right|dm\underset{r\to R_\mu}{\longrightarrow}0.$$
Let $r_n$ be a sequence converging to $R_\mu$ such that
$$\frac{1}{\sqrt{R_\mu-r_n}}\int \left |\mathrm{e}^{\varphi_{r_n}-\varphi_{R_\mu}}-1\right |\left |h_j^{(i)}-\tilde{h}_{j,r_n}^{(i)}\right|dm\underset{n\to +\infty}{\longrightarrow}\alpha \in [0,+\infty].$$
According to Corollary~\ref{coroKellerLiveranisatisfait},
$$\int \left |h_j^{(i)}-\tilde{h}_{j,r_n}^{(i)}\right|dm\underset{n\to +\infty}{\longrightarrow} 0,$$
so up to taking a subsequence, $\left |h_j^{(i)}-\tilde{h}_{j,r_n}^{(i)}\right|$ converges to 0 $m$-almost everywhere

We now focus on
$\frac{1}{\sqrt{R_\mu-r_n}}\left |\mathrm{e}^{\varphi_{r_n}-\varphi_{R_\mu}}-1\right |$.
We show that
\begin{equation}\label{deuxiemeintegraleperturbationpression}
    \frac{1}{\sqrt{R_\mu-r}}\int \left |\mathrm{e}^{\varphi_{r}-\varphi_{R_\mu}}-1\right |dm\lesssim 1.
\end{equation}
Differentiating the expression $\mathrm{e}^{\varphi_{r}-\varphi_{R_\mu}}(x)$, we get
$$\left (\frac{d}{dr}\varphi_r (x)\right )\mathrm{e}^{\varphi_{r}-\varphi_{R_\mu}}(x).$$
Weak relative Ancona inequalities yield
\begin{equation}\label{varphi_rbornee}
    \mathrm{e}^{\varphi_{r}-\varphi_{R_\mu}}(x)=\frac{H(e,x|r)/H(x_1,x|r)}{H(e,x|R_\mu)/H(x_1,x|R_\mu)}\lesssim \frac{H(e,x_1|r)}{H(e,x_1|R_\mu)}\lesssim1.
\end{equation}
Thus, Lemma~\ref{lemmecontinuiteoperateurtransfert} shows that
\begin{equation}\label{varphi_rsuperbornee}
    \left |\frac{d}{dr}\left (\mathrm{e}^{\varphi_{r}-\varphi_{R_\mu}}(x)\right )\right |\lesssim \frac{1}{\sqrt{R_\mu-r}}\varphi(x).
\end{equation}
Integrating this inequality, we get
$$\left |\mathrm{e}^{\varphi_{r}-\varphi_{R_\mu}}(x)-1\right |\lesssim \varphi(x)\sqrt{R_\mu-r}.$$
Integrating with respect to $m$, we finally obtain~(\ref{deuxiemeintegraleperturbationpression}).
Also,
$$\frac{1}{\sqrt{R_\mu-r_n}}\left |\mathrm{e}^{\varphi_{r_n}-\varphi_{R_\mu}}-1\right |\left |h_j^{(i)}-\tilde{h}_{j,r_n}^{(i)}\right|\lesssim \varphi\left |h_j^{(i)}-\tilde{h}_{j,r_n}^{(i)}\right|$$
so that
$$\frac{1}{\sqrt{R_\mu-r_n}}\left |\mathrm{e}^{\varphi_{r_n}-\varphi_{R_\mu}}-1\right |\left |h_j^{(i)}-\tilde{h}_{j,r_n}^{(i)}\right|$$
converges to 0 $m$-almost everywhere.

Finally, we deduce from Corollary~\ref{coroKellerLiveranisatisfait} that $\left |h_j^{(i)}-\tilde{h}_{j,r_n}^{(i)}\right|$ is uniformly bounded, hence
$$\frac{1}{\sqrt{R_\mu-r_n}}\left |\mathrm{e}^{\varphi_{r_n}-\varphi_{R_\mu}}-1\right |\left |h_j^{(i)}-\tilde{h}_{j,r_n}^{(i)}\right|\lesssim \varphi.$$
We apply the dominated convergence theorem, so that
$$\frac{1}{\sqrt{R_\mu-r_n}}\int \left |\mathrm{e}^{\varphi_{r_n}-\varphi_{R_\mu}}-1\right |\left |h_j^{(i)}-\tilde{h}_{j,r_n}^{(i)}\right|dm\underset{n\to +\infty}{\longrightarrow}0.$$
In other words, $\alpha=0$, which concludes the proof.
\end{proof}

\begin{lemma}\label{deuxiemeterme=oracineR-r}
We have
$$\int\left ( \left (\mathrm{e}^{\varphi_r-\varphi_{R_\mu}}-1\right )-(\varphi_r-\varphi_{R_\mu})\right )dm=o\left (\sqrt{R_\mu-r}\right ).$$
\end{lemma}

\begin{proof}
Differentiating the integrand, we get
$$\left (\frac{d}{dr}\varphi_r\right ) \left (\mathrm{e}^{\varphi_r-\varphi_{R_\mu}}-1\right ).$$
Fix $R\leq R_\mu$ and let
$$\overline{ \varphi}_R=\sup_{R\leq r\leq R_\mu}|\mathrm{e}^{\varphi_r-\varphi_{R_\mu}}-1|.$$
For every $R\leq r\leq R_\mu$, according to Lemma~\ref{lemmecontinuiteoperateurtransfert},
$$\left |\frac{d}{dr} \left (\mathrm{e}^{\varphi_r-\varphi_{R_\mu}}-1\right )-(\varphi_r-\varphi_{R_\mu}) \right |\leq \varphi \overline{\varphi}_R\frac{1}{\sqrt{R_\mu-r}}.$$
Integrating this inequality over $r$ varying between $R$ and $R_\mu$, we get
$$\left |\left (\mathrm{e}^{\varphi_R-\varphi_{R_\mu}}-1\right )-(\varphi_R-\varphi_{R_\mu}) \right |\leq \varphi \overline{\varphi}_R \sqrt{R_\mu-R}.$$
It is thus enough to prove that
$$\int \varphi\overline{\varphi}_Rdm\underset{R\to R_\mu}{\longrightarrow}0.$$
Consider a sequence $r_k$ converging to $R_\mu$ such that
$$\int \varphi\overline{\varphi}_{r_k}dm\underset{R\to R_\mu}{\longrightarrow}\alpha.$$
According to~(\ref{varphi_rsuperbornee}), $\int \overline{\varphi}_Rdm$ converges to 0, so up to taking a subsequence, $\overline{\varphi}_{r_k}$ converges to 0 $m$-almost everywhere.
Hence, $\varphi\overline{\varphi}_{r_k}$ converges to 0 $m$-almost everywhere.
Also, according to~(\ref{varphi_rbornee}), $\overline{\varphi}_r$ is uniformly bounded.
We apply the dominated convergence theorem, so that
$$\int \varphi\overline{\varphi}_{r_k}dm\underset{R\to R_\mu}{\longrightarrow}0.$$
In other words, $\alpha=0$, which concludes the proof.
\end{proof}

We combine Lemma~\ref{premierterme=oracineR-r}, Lemma~\ref{deuxiemeterme=oracineR-r} and~(\ref{asymptotiqueP(r)}) to show that
$$\mathrm{e}^{\tilde{P}_j(r)}-1=\int (\varphi_r-\varphi_{R_\mu})dm_j + o\left ( \sqrt{R_\mu-r}\right ).$$
We proved in \textit{step 4} above that the integral in the right member does not depend on $j$.
Choose $j'$ so that the pressure is maximal.
Then, for every $j$,
$$\mathrm{e}^{\tilde{P}_j(r)}-1=\mathrm{e}^{P(r)}-1+ o\left ( \sqrt{R_\mu-r}\right )$$
hence,
$$\mathrm{e}^{\tilde{P}_j(r)}-1=P(r)+o\left (P(r)\right )+o\left ( \sqrt{R_\mu-r}\right ).$$
We deduce from Lemma~\ref{corollary3.3Gouezel} and from~(\ref{3.6Gouezel}) that $|P(r)|/\sqrt{R_{\mu}-r}$ is bounded away from zero and infinity.
Thus,
$$\mathrm{e}^{\tilde{P}_j(r)}-1=P(r)+o\left (P(r)\right ).$$
Consequently, for every $j$, $\tilde{P}_j(r)$ converges to 0 as $r$ tends to $R_\mu$, so that
$$\tilde{P}_j(r)\sim \mathrm{e}^{\tilde{P}_j(r)}-1, r\to R_\mu.$$
This also proves that
$$\tilde{P}_j(r)\sim P(r), r\to R_\mu,$$
which concludes the proof.
\end{proof}

\section{Asymptotic of the second derivative of the Green function}\label{Sectionproofofpreciseequadiff}

Our goal here is to prove the following proposition.
We still assume that $\mu$ is not spectrally degenerate.

\begin{proposition}\label{propositioncomparaisonI_2etI_1}
When $r\to R_{\mu}$, we have
$$I^{(2)}(r)=\xi I^{(1)}(r)^3+O\left (I^{(1)}(r)^2\right ),$$
for some $\xi>0$.
\end{proposition}

We introduce some notations.
We define for $r<R_{\mu}$ the function
$$\Phi_r(\gamma)=\frac{\sum_{\gamma'}G(e,\gamma'|r)G(\gamma',\gamma|r)}{G(e,\gamma|r)}.$$
By definition,
$$I^{(2)}(r)=\sum_{n\geq 0}\sum_{\gamma\in \hat{S}^n}H(e,\gamma|r)\Phi_r(\gamma)=H(e,e|r)\sum_{n\geq 0}\mathcal{L}_r^n(1_{E_*}\Phi_r\circ \phi) (\emptyset).$$


According to Theorem~\ref{GouezelSarig}, for any $f:\Gamma\cup \partial\hat{\Gamma}\rightarrow \R$ such that $f\circ \phi:\overline{\Sigma}_A\rightarrow \R$ is in $H_{\rho,\beta}$,
$$\sum_{n\geq 0}\mathcal{L}_r^n(1_{E_*}f\circ \phi)(\emptyset)=\sum_{j=1}^k\frac{c_f(j,r)}{|\tilde{P}_j(\varphi_r)|}+O(1),r\to R_{\mu}.$$
Also, according to~(\ref{3.7Gouezel}),
$$I^{(1)}(r)\sim \frac{\xi(r)}{P(r)},r\to R_\mu.$$
Finally, Proposition~\ref{pressureindependentcomponents}, shows that $|\tilde{P}_j(\varphi_r)|/P(r)$ converges to 1.
Hence,
\begin{equation}\label{estimeesommeoperateurtransfert}
    \sum_{n\geq 0}\mathcal{L}_r^n(1_{E_*}f\circ \phi)(\emptyset)\sim I^{(1)}(r) c_f,r\to R_\mu,
\end{equation}
where $c_f$ only depends on $f$.
In other words $\frac{1}{I^{(1)}(r)}\sum_{n\geq 0}\mathcal{L}_r^n(1_{E_*}f\circ \phi)(\emptyset)$ converges.
However, $\Phi_r\circ \phi \notin H_{\rho,\beta}$.
The goal of the next subsection is to transform $\Phi_r$ in order to apply~(\ref{estimeesommeoperateurtransfert}).

\subsection{A partition of unity}
We start with a rough study of $\Phi_r$.
Let $\gamma\in \Gamma$ and let $[e,\gamma]=(e,\gamma_1,...,\gamma_{\hat{d}(e,\gamma)-1},\gamma)$ be a relative geodesic from $e$ to $\gamma$.
For every $k\leq \hat{d}(e,\gamma)$, denote by $\Gamma_k$ the set of $\gamma'\in \Gamma$ whose projection on $[e,\gamma]$ is at $\gamma_k$.
If there are more than one such projections, we choose the closest to $\gamma$.
Also denote by $\tilde{\gamma}_k$ the projection of $\gamma'$ on the union $\mathcal{H}_k$ of parabolic subgroups containing $\gamma_{k-1}^{-1}\gamma_k$.
Lemma~\ref{lemmarelativetripod} shows that any relative geodesic from $e$ to $\gamma'$ passes within a bounded distance of $\gamma_{k-1}$.
Also, \cite[Lemma~1.13~(1)]{Sisto2} shows that the exit point from $\mathcal{H}_k$ of any such relative geodesic is within a bounded distance of $\tilde{\gamma}_k$.
Thus, any relative geodesic from $e$ to $\gamma$' passes first within a bounded distance of $\gamma_{k-1}$ and then within a bounded distance of $\tilde{\gamma}_k$.
Also, any relative geodesic from $\gamma'$ to $\gamma$ passes within a bounded distance of $\tilde{\gamma}_k$, then of $\gamma_k$.
Weak relative Ancona inequalities imply that for every $\gamma'\in \Gamma_k$,
\begin{align*}
    &\frac{G(e,\gamma'|r)G(\gamma',\gamma|r)}{G(e,\gamma|r)}\asymp\\
    &\hspace{1cm}\frac{G(e,\gamma_{k-1}|r)G(\gamma_{k-1},\tilde{\gamma}_k|r)G(\tilde{\gamma}_k,\gamma'|r)G(\gamma',\tilde{\gamma}_k|r)G(\tilde{\gamma}_k,\gamma_k|r)G(\gamma_k,\gamma|r)}{G(e,\gamma_{k-1}|r)G(\gamma_{k-1},\gamma_k|r)G(\gamma_{k},\gamma|r)}.
\end{align*}
We thus obtain that
$$\frac{G(e,\gamma'|r)G(\gamma',\gamma|r)}{G(e,\gamma|r)}\asymp \frac{G(\gamma_{k-1},\tilde{\gamma}_k|r)G(\tilde{\gamma}_k,\gamma_k|r)}{G(\gamma_{k-1},\gamma_k|r)}G(\tilde{\gamma}_k,\gamma'|r)G(\gamma',\tilde{\gamma}_k|r).$$
We then sum over all $\gamma'\in \Gamma_k$.
Let $\mathcal{H}_{\gamma_k}$ be the union of all parabolic subgroups in $\Omega_0$ containing $\gamma_{k-1}^{-1}\gamma_k$.
Then $\gamma_{k-1}^{-1}\tilde{\gamma}_k\in \mathcal{H}_{\gamma_k}$.
We can decompose the sum over $\gamma'\in \Gamma_k$ according to the projection on $\tilde{\gamma}_k\in \mathcal{H}_{\gamma_k}$.
Bounding $\sum G(\tilde{\gamma}_k,\gamma'|r)G(\gamma',\tilde{\gamma}_k|r)$ by $I^{(1)}(r)$, where $\tilde{\gamma}_k$ is fixed and the sum is over all $\gamma'$ projecting on $\tilde{\gamma}_k$, we finally get
$$\sum_{\gamma'\in \Gamma_k}\frac{G(e,\gamma'|r)G(\gamma',\gamma|r)}{G(e,\gamma|r)}\lesssim I_1(r) \sum_{\tilde{\gamma}_k\in \mathcal{H}_{\gamma_k}}\frac{G(\gamma_{k-1},\tilde{\gamma}_k|r)G(\tilde{\gamma}_k,\gamma_k|r)}{G(\gamma_{k-1},\gamma_k|r)}.$$

We then construct a function $\Upsilon_r$ as follows.
For every $\gamma\in \Gamma$, we choose a relative geodesic from $e$ to $\gamma$, using the automaton $\mathcal{G}$.
Let $\gamma_1$ be the first point after $e$ on this relative geodesic
Notice that $\gamma_1$ coincides with the first increment $\sigma_1$ on this relative geodesic.
In general, we denote by $\sigma_k=\gamma_{k-1}^{-1}\gamma_k$ the $k$th increment.
Also, let $\mathcal{H}_{\sigma_1}$ be the union of all parabolic subgroups in $\Omega_0$ containing $\sigma_1$.
We set
$$\Upsilon_r(\gamma)=\sum_{\sigma\in \mathcal{H}_{\sigma_1}}\frac{G(e,\sigma|r)G(\sigma,\sigma_1|r)}{G(e,\sigma_1|r)}.$$
This function  $\Upsilon_r$ only depends on the first element of the relative geodesic $[e,\gamma]$.
In other words, the function $\Upsilon_r\circ \phi(x)$ only depends on the first symbol of $x\in \overline{\Sigma}_A$.
The estimate above yields
\begin{equation}\label{equationPhiUpsilon}
    \Phi_r(\gamma)\lesssim I_1(r)\sum_{k=0}^{\hat{d}(e,\gamma)-1}\Upsilon_r \circ T^k ([e,\gamma]),
\end{equation}
where $T$ is the left shift on relative geodesic, that is
$T([e,\gamma])$ is the relative geodesic $(e,\gamma_1^{-1}\gamma_2,...,\gamma_1^{-1}\gamma)$.
Note that $\Upsilon_r$ is not bounded.
Indeed, assuming that $\sigma_1$ is only in one parabolic subgroup $\mathcal{H}_1$ to simplify, then $\Upsilon_r(\gamma)$ is essentially given by
$$\sum_{\sigma\in \mathcal{H}_{\sigma_1}}\frac{G^{(1)}_{1,r}(e,\sigma_1)}{G(e,\sigma_1|r)}.$$
This quantity is not bounded.

However, to prove Proposition~\ref{propositioncomparaisonI_2etI_1}, we need to estimate $\mathcal{L}_r(\Upsilon_r\circ \phi)(x)$.
Recall that $X_x^1$ is the set of symbols $\sigma$ that can precede $x$ in $\Sigma_A$.
Seeing $x$ and $\sigma$ as elements of $\Gamma$,
$$\mathcal{L}_r(\Upsilon_r\circ \phi)(x)=\sum_{\sigma\in X_{\gamma}^1}\frac{H(e,\sigma x|r)}{H(e,x|r)}\sum_{\sigma '\in \mathcal{H}_{\sigma}}\frac{G(e,\sigma'|r)G(\sigma',\sigma|r)}{G(e,\sigma|r)}.$$
Therefore, weak relative Ancona inequalities show that
$$\mathcal{L}_r(\Upsilon_r\circ \phi)(x)\lesssim \sum_{j}\sum_{\sigma,\sigma'\in \mathcal{H}_j}G(e,\sigma'|r)G(\sigma',\sigma|r)G(\sigma,e|r).$$
We rewrite this as
$$\mathcal{L}_r(\Upsilon_r\circ \phi)(x)\lesssim \sum_{j} I_j^{(2)}(r).$$
Since $\mu$ is not spectrally degenerate, we get
\begin{equation}\label{Upsilonuniformementbornee}
    \mathcal{L}_r(\Upsilon_r\circ \phi)(x)\lesssim 1.
\end{equation}

We only gave a rough estimate.
To obtain a precise asymptotic, we will replace the decomposition of $\Gamma$ into subsets $\Gamma_k$ as above by a continuous decomposition, using a partition of unity.
We construct such a partition of unity adapting the arguments of \cite[Lemma~8.5]{GouezelLalley}.

We first introduce the following terminology.
Consider a relative geodesic $\alpha$ that we write $\alpha=(\alpha_{-m},...,\alpha_{-1},\alpha_0,\alpha_1,...,\alpha_n)$.
Let
$\hat{l}(\alpha)=\hat{d}(\alpha_{-m},\alpha_n)$ be the relative length of $\alpha$ (here $n+m$) and let $l(\alpha)$ be its total length, defined by
$$l(\alpha)=\sum_{k=-m+1}^{n}d(\alpha_{k-1},\alpha_k).$$
Note that if $\alpha,\alpha'$ are two relative geodesic with the same endpoints, we do not have $l(\alpha)=l(\alpha')$ in general.
However, the distance formula given by \cite[Theorem~3.1]{Sisto2} shows that
$$\frac{1}{\lambda_1}d(\alpha_{-m},\alpha_n)-c_1\leq l(\alpha)\leq \lambda_1 d(\alpha_{-m},\alpha_n)+c_1$$
and so
$$\frac{1}{\lambda_2}l(\alpha)-c_2\leq l(\alpha')\leq \lambda_2l(\alpha)+c_2.$$

Our goal is to construct a partition of unity $\kappa_{\alpha}$ associated to such a relative geodesic $\alpha$.
Write $\alpha=(\alpha_{-m},...,\alpha_{-1},\alpha_0,\alpha_1,...,\alpha_n)$ and suppose that $\alpha_0=e$. 
To simplify, let $\alpha_-$ and $\alpha_+$ be the endpoints of $\alpha$, that is $\alpha_-=\alpha_{-m}$ and $\alpha_+=\alpha_n$.
Denote by $\alpha_l$ the sub-relative geodesic of $\alpha$ from $\alpha_{-}$ to $e$ and by $\alpha_r$ the sub-relative geodesic from $\alpha_1$ to $\alpha_+$.

Consider two constants $K_1$ and $K_2$ that we will choose later.
Assume that $l(\alpha_l)\geq 2K_1$ and $l(\alpha_r)\geq 2K_1$.
Denote by $A(K_1)$ the set of $\gamma\in \Gamma$ such that
\begin{itemize}
    \item there either exists a relative geodesic from $\gamma$ to $\alpha_+$ whose distance from $\alpha_1$ is at least $K_1$,
    \item or there exists a relative geodesic from $\alpha_-$ to $\gamma$ whose distance from $e$ is at least $K_1$.
\end{itemize}
Also denote by $B(K_2)$ the set of $\gamma\in \Gamma$ such that
\begin{itemize}
    \item any relative geodesic from $\gamma$ to $\alpha_+$ passes within $K_2$ of $\alpha_1$
    \item and any relative geodesic from $\alpha_-$ to $\gamma$ passes within $K_2$ of $e$.
\end{itemize}
In other words, $A(K_1)=B(K_1)^c$.
Note that $B(K_2)$ is not empty and that $A(K_1)$ is not empty, for $l(\alpha_l)\geq 2K_1$ and $l(\alpha_r)\geq 2K_1$.

The following is a simple consequence of the fact that triangles are thin along transition points \cite[Lemme~2.4]{GekhtmanDussaule} and that transition points are within a bounded distance of points on a relative geodesic \cite[Proposition~8.13]{Hruska}.
\begin{lemma}
For fixed $K_2$, if $K_1$ is large enough, then the closures of $A(K_1)$ and $B(K_2)$ in the Bowditch compactification $\overline{\Gamma}_B$ are disjoint.
\end{lemma}

Since $\overline{\Gamma}_B$ is compact, there exists a continuous function $f_{\alpha}$ on $\overline{\Gamma}_B$ taking values in $[0,1]$, that vanishes on $A(K_1)$ and which is equal to 1 on $B(K_2)$.

\medskip

We now finish the construction of the partition of unity associated with $\alpha$.
Let $n_1=n_1(\alpha)$ be the largest integer such that translating $n_1$ times the relative geodesic $\alpha$, on the right, the length on the left is still at least $K_1$.
Formally,
\begin{equation}\label{defn_1}
    n_1(\alpha)=\sup \{k\geq 0, l\left ((T^{-k}\alpha )_l\right )\geq K_1\}.
\end{equation}
Similarly, let $n_2=n_2(\alpha)$ be the largest integer such that translating $n_2$ times $\alpha$ on the left, the length on the right is at least $K_1$.
That is,
\begin{equation}\label{defn_2}
    n_2(\alpha)=\sup \{k\geq 0, l\left ((T^{k}\alpha )_r\right )\geq K_1\}.
\end{equation}
Let $A'(K_1)$ be the set of $\gamma$ such that for every $k\in [-n_1,n_2-1]$,
\begin{itemize}
    \item there either exists a relative geodesic from $\gamma$ to $\alpha_+$ whose distance from $\alpha_{k+1}$ is at least $K_1$,
    \item or there exists a relative geodesic from $\alpha_-$ to $\gamma$ whose distance from $\alpha_k$ is at least $K_1$.
\end{itemize}
Let $B'(K_2)$ be the set of $\gamma$ such that there exists $k\in [-n_1,n_2-1]$ such that
\begin{itemize}
    \item any relative geodesic from $\gamma$ to $\alpha_+$ passes within $K_2$ of $\alpha_{k+1}$
    \item and any relative geodesic from $\alpha_-$ to $\gamma$ passes within $K_2$ of $\alpha_k$.
\end{itemize}
Again, if $K_1$ is large enough, the closures of $A'(K_1)$ and $B'(K_2)$ in the Bowditch compactification are disjoint.

For technical reasons, we will further need to truncate relative geodesics.
Letting $\beta$ be any relative geodesic with $\beta_0=e$, denote by $\beta_{(2K_1)}$ the shortest sub-relative geodesic of $\beta$ such that $l((\beta_{(2K_1)})_l)\geq 2K_1$ and $l((\beta_{(2K_1)})_r)\geq 2K_1$.
In other words, we truncate $\beta$ on the left, respectively on the right, as soon as the length on the left, respectively on the right, is at least $2K_1$.
Note that if $K_1$ is large enough, whenever $\gamma\in B'(K_2)$, we have
\begin{equation}\label{defsigmatranslateskappa}
    \Sigma=\sum_{k=-n_1}^{n_2}f_{(T^k\alpha)_{(2K_1)}}(\alpha_k^{-1}\gamma)\geq 1.
\end{equation}
Thus, there exists a function $g_{\alpha}$, which is continuous on the Bowditch compactification, which is equal to 1 on $A'(K_1)$, which is equal to the sum $\Sigma$ above on $B'(K_2)$ and whose values are between 1 and $\Sigma$.
We set
$$\kappa_{\alpha}=\frac{f_{\alpha_{(2K_1)}}}{g_{\alpha}}.$$

Denote by $\alpha'_+$ the right endpoint of $\alpha_{(2K_1)}$ and by $\alpha'_-$ the left endpoint of $\alpha_{(2K_1)}$.
According to Lemma~\ref{lemmarelativetripod}, the fact that a relative geodesic from $\gamma$ to $\alpha'_+$ passes within $K_2$ of $\alpha_1$ and a relative geodesic from $\alpha'_-$ to $\gamma$ passes within $K_2$ of $\alpha_0$ means that $\gamma$ projects on $\alpha$ approximately between $\alpha_0=e$ and $\alpha_1$.
More precisely, the projections on $\alpha_{(2K_1)}$ which are the closest to $\alpha'_+$ and to $\alpha'_-$ are within a bounded distance of $\alpha_1$ and $e$ respectively.
We deduce that the sum $\Sigma$ is bounded by some constant that only depends on $K_2$ and $K_1$.
Roughly speaking, $\kappa_{\alpha}$ is a continuous function whose successive images by the shift mimics the decomposition of $\Gamma$ into the subsets $\Gamma_k$.

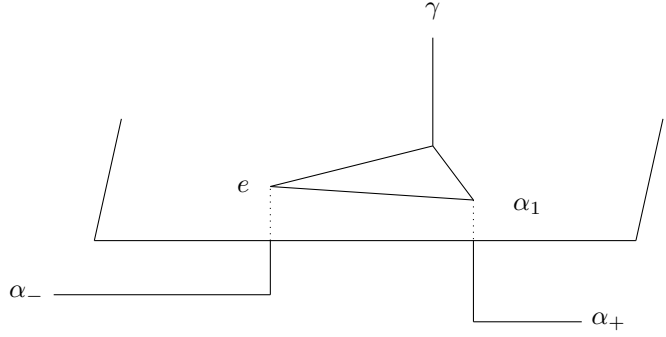
\begin{center}
\begin{tikzpicture}[scale=1.8]
\draw (-1.8,.5)--(-2,-.4);
\draw (2.2,.5)--(2,-.4);
\draw (-2,-.4)--(2,-.4);
\draw (-.7,-.8)--(-.7,-.4);
\draw[dotted] (-.7,-.4)--(-.7,0);
\draw (-.7,-.8)--(-2.3,-.8);
\draw (-.7,0)--(.8,-.1);
\draw (-.7,0)--(.5,.3);
\draw[dotted] (.8,-.1)--(.8,-.4);
\draw (.8,-.4)--(.8,-1);
\draw (.8,-1)--(1.6,-1);
\draw (.5,.3)--(.5,1.1);
\draw (.8,-.1)--(.5,.3);
\draw (-.9,0)node{$e$};
\draw (-2.5,-.8) node{$\alpha_-$};
\draw (1.2,-.15) node{$\alpha_{1}$};
\draw (1.8,-1) node{$\alpha_+$};
\draw (.5,1.3) node{$\gamma$};
\end{tikzpicture}
\end{center}


\medskip
Let $\alpha=(\alpha_{-k},...,\alpha_0,\alpha_1,...,\alpha_l)$ be a relative geodesic, with $\alpha_0=e$.
Whenever $l(\alpha_l)<2K_1$ or $l(\alpha_r)<2K_1$, we set $\Psi_r(\alpha)=0$.
Otherwise, we set
\begin{equation}\label{defPsi_r}
    \Psi_r(\alpha)=\frac{1}{I^{(1)}(r)}\sum_{\gamma\in \Gamma}\kappa_{\alpha}(\gamma)\frac{G(\alpha_-,\gamma|r)G(\gamma,\alpha_+|r)}{G(\alpha_-,\alpha_+|r)}.
\end{equation}
This defines a function $\Psi_r$ on the set of relative geodesics $\alpha$ with $\alpha_0=e$.

According to the discussion above, $\kappa_{\alpha}(\gamma)\neq 0$ can only happen if the projection of $\gamma$ on $\alpha_{(2K_1)}$ lies between $e$ and $\alpha_1$, up to a bounded distance.
Choosing $K_1$ large enough, the same is true replacing $\alpha_{(2K_1)}$ with $\alpha$.
Indeed, let $\alpha_-$ and $\alpha_+$ be the endpoints of $\alpha$ and let $\alpha'_-$ and $\alpha'_+$ the endpoints of $\alpha_{(2K_1)}$.
If $K_1$ is large enough, then a relative geodesic from $\gamma$ to $\alpha'_-$ passes within a bounded distance of $e$ if and only if a relative geodesic from $\gamma$ to $\alpha_-$ also passes within a bounded distance of $e$.
Similarly, a relative geodesic from $\gamma$ to $\alpha_+$ passes within a bounded distance of $\alpha_1$ if and only if the same is true for a relative geodesic from $\gamma$ to $\alpha'_+$.
Weak relative Ancona inequalities thus show that for any $k$ large enough so that $\Psi_r(T^k[e,\gamma])\neq 0$,
\begin{equation}\label{upsilonvspsi}
    \Psi_r(T^k[e,\gamma])\asymp \Upsilon_r(T^k\gamma).
\end{equation}
Hence, the function $\Psi_r$ will replace the function $\Upsilon_r$ in~(\ref{equationPhiUpsilon}) to obtain a more accurate estimate.

\begin{proposition}\label{propestimeePhietPsi}
Let $\gamma\in \Gamma$ and let $\alpha_\gamma$ be the relative geodesic from $e$ to $\gamma$ given by the automaton $\mathcal{G}$.
Then,
$$\Phi_r(\gamma)=I^{(1)}(r)\sum_{k=0}^{\hat{d}(e,\gamma)-1}\Psi_r(T^k\alpha_\gamma)+O\left (I^{(1)}(r)\right ).$$
\end{proposition}

\begin{proof}
Consider $\gamma\in \Gamma$.
To simplify things, denote by $\alpha$ the relative geodesic from $e$ to $\gamma$.
Let $m_1$ be the smallest integer such that $l\left ((T^{m_1}\alpha)_l\right )\geq K_1$
and $m_2$ the largest one such that $l\left ((T^{m_2}\alpha)_r\right )\geq K_1$.
By definition of $\Psi_r$,
\begin{align*}
    \sum_{k=0}^{\hat{d}(e,\gamma)-1}\Psi_r(T^k\alpha)&=\sum_{k=m_1}^{m_2}\Psi_r(T^k\alpha)\\
    &=\frac{1}{I^{(1)}(r)}\sum_{k=m_1}^{m_2}\sum_{\gamma'\in \Gamma}\kappa_{(T^k\alpha)}(\gamma')\frac{G(\alpha_k^{-1},\gamma'|r)G(\gamma',\alpha_k^{-1}\gamma|r)}{G(\alpha_k^{-1},\alpha_k^{-1}\gamma|r)}.
\end{align*}
Translating on the left by $\alpha_k^{-1}$ and replacing $\gamma'$ with $\alpha_k^{-1}\gamma'$ in the sum, we get
$$\sum_{k=0}^{\hat{d}(e,\gamma)-1}\Psi_r(T^k\alpha)=\frac{1}{I^{(1)}(r)}\sum_{\gamma'\in \Gamma}\left ( \sum_{k=m_1}^{m_2}\kappa_{(T^k\alpha)}(\alpha_k^{-1}\gamma')\right ) \frac{G(e,\gamma'|r)G(\gamma',\gamma|r)}{G(e,\gamma|r)}.$$
Recall that
$$\Phi_r(\gamma)=\sum_{\gamma'\in \Gamma} \frac{G(e,\gamma'|r)G(\gamma',\gamma|r)}{G(e,\gamma|r)}.$$
We are thus looking for the elements $\gamma'\in \Gamma$ such that
\begin{equation}\label{partitiondeluniteincomplete}
    \sum_{k=m_1}^{m_2}\kappa_{(T^k\alpha)}(\alpha_k^{-1}\gamma')\neq 1.
\end{equation}

Suppose there exists $i\in [m_1,m_2-1]$ such that any relative geodesic from $\gamma'$ to $\gamma$ passes within $K_2$ of $\alpha_{i+1}$ and that any relative geodesic from $e$ to $\gamma'$ passes within $K_2$ of $\alpha_i$.
Fix $k\in [m_1,m_2-1]$ and consider the translated relative geodesic $T^k\alpha$.
Then, any relative geodesic from $\alpha_k^{-1}\gamma'$ to the right endpoint of $T^k\alpha$ passes within $K_2$ of $\alpha_k^{-1}\alpha_{i+1}$ and any relative geodesic from the left endpoint of $T^k\alpha$ to $\alpha_k^{-1}\gamma'$ passes within $K_2$ of $\alpha_k^{-1}\alpha_i$.

In particular,
$\alpha_k^{-1}\gamma'\in B'(K_2)$, where $B'(K_2)$ is the set constructed as above, using the relative geodesic $T^k\alpha$.
Hence,
$$g_{T^k\alpha}(\alpha_k^{-1}\gamma')=\sum_{j=-n_1}^{n_2}f_{(T^j(T^k\alpha))_{(2K_1)}}((T^k\alpha)_j^{-1}\gamma').$$
Note that $(T^k\alpha)_j=\alpha_{k+j}$, so that
$$g_{T^k\alpha}(\alpha_k^{-1}\gamma')=\sum_{j=m_1-k}^{m_2-k}f_{(T^{k+j}\alpha)_{(2K_1)}}(\alpha_{k+j}^{-1}\gamma')=\sum_{j=m_1}^{m_2}f_{(T^{j}\alpha)_{(2K_1)}}(\alpha_j^{-1}\gamma').$$
In particular, we see that
$$\sum_{k=m_1}^{m_2}\kappa_{T^k\alpha}(\alpha_k^{-1}\gamma')= 1.$$

We proved that whenever~(\ref{partitiondeluniteincomplete}) holds, then for every $i\in [m_1,m_2-1]$,
either a relative geodesic from $\gamma'$ to $\gamma$ stays at distance at least $K_2$ from $\alpha_{i+1}$ or a relative geodesic from $e$ to $\gamma'$ stays at distance at least $K_2$ from $\alpha_i$.
We will again use Lemma~\ref{lemmarelativetripod}.
Let $\alpha_{k+1}$ be the projection of $\gamma'$ the closest to $\gamma$ on $\alpha$.
If $K_2$ was chosen large enough, then we necessarily have $k\geq m_2$ or $k\leq m_1-1$.
Also let $\mathcal{H}_{\alpha_k}$ be the union of parabolic subgroups in $\Omega_0$ containing $\alpha_{k}^{-1}\alpha_{k+1}$
and let $\tilde{\alpha}_k$ be the projection of $\gamma'$ on $\mathcal{H}_{\alpha_k}$.
According to \cite[Lemma~1.13~(1)]{Sisto2}, the exit point from $\mathcal{H}_{\alpha_k}$ of any relative geodesic from $e$ to $\gamma'$ is within a bounded distance of $\tilde{\alpha}_k$.
Thus, any such relative geodesic passes first within a bounded distance of $\alpha_k$ and then within a bounded distance of $\tilde{\alpha}_k$.
Similarly, any relative geodesic from $\gamma'$ to $\gamma$ passes first within a bounded distance of $\tilde{\alpha}_k$ and then within a bounded distance of $\alpha_{k+1}$.
Weak relative Ancona inequalities yield
$$\frac{G(e,\gamma'|r)G(\gamma',\gamma|r)}{G(e,\gamma|r)}\lesssim I^{(1)}(r) \frac{G(\alpha_k,\tilde{\alpha}_k|r)G(\tilde{\alpha}_k,\alpha_{k+1}|r)}{G(\alpha_k,\alpha_{k+1}|r)}.$$

Also recall that~(\ref{defsigmatranslateskappa}) is uniformly bounded.
Consequently, we have
\begin{align*}
    &\left |\Phi_r(\gamma)-I^{(1)}(r)\sum_{k=0}^{\hat{d}(e,\gamma)-1}\Psi_r(T^k\alpha)\right |\lesssim  \ I^{(1)}(r)\sum_{0\leq k< m_1}\sum_{\sigma\in \mathcal{H}_{\alpha_k}}\frac{G(\alpha_k,\sigma|r)G(\sigma,\alpha_{k+1}|r)}{G(\alpha_k,\alpha_{k+1}|r)}\\
    &\hspace{2cm}+I^{(1)}(r)\sum_{m_2\leq k\leq \hat{d}(e,\gamma)}\sum_{\sigma\in \mathcal{H}_{\alpha_k}}\frac{G(\alpha_k,\sigma|r)G(\sigma,\alpha_{k+1}|r)}{G(\alpha_k,\alpha_{k+1}|r)}.
\end{align*}
By definition, $m_1$ and $\hat{d}(e,\gamma)-m_2$ are bounded by $K_1$ and
$d(\alpha_{k},\alpha_{k+1})\leq K_1$.
In particular,
$$\sum_{\sigma\in \mathcal{H}_{\alpha_k}}\frac{G(\alpha_k,\sigma|r)G(\sigma,\alpha_{k+1}|r)}{G(\alpha_k,\alpha_{k+1}|r)}$$
is uniformly bounded.
We thus have
$$ \Phi_r(\gamma)-I^{(1)}(r)\sum_{k=0}^{\hat{d}(e,\gamma)-1}\Psi_r(T^k\alpha)\lesssim I^{(1)}(r),$$
which concludes the proof.
\end{proof}

\subsection{Truncating $\Psi_r$}
We say that a function $f$ defined on relative geodesic $\alpha$ satisfying $\alpha_0=e$ is locally H\"older if for every $n\geq 1$, as soon as $\alpha$ and $\alpha'$ coincide in the relative ball (for the distance $\hat{d}$) of center $e$ and radius $n$, 
$|f(\alpha)-f(\alpha')|\leq C\rho^n$, for some $C\geq 0$ and some $0<\rho<1$.

A similar function $\Psi_r$ is defined in \cite{Gouezel1} for hyperbolic groups.
It is proved there that for every $r$, $\Psi_r$ is continuous, locally H\"older and that $\Psi_r$ uniformly converges to a locally H\"older function $\Psi_{R_\mu}$, as $r$ tends to $R_\mu$.
However, in our situation, such properties do not hold, so we cannot directly apply the strategy in \cite{Gouezel1}.
We are going instead to truncate $\Psi_r$ such that our new function only depends on a finite number of symbols.

Precisely, fix a constant $N\in \N$ and denote by
$\alpha^{(N)}$ the relative geodesic $\alpha$ restricted to $[-N,N]$.
In particular, if $\hat{d}(e,\alpha_-)\leq N$ and $\hat{d}(e,\alpha_+)\leq N$, then $\alpha^{(N)}=\alpha$.
We set
$\Psi_r^{(N)}(\alpha)=\Psi_r(\alpha^{(N)})$.
Let us fix another constant $D$ and define $\Psi_r^{(D,N)}(\alpha)$ as follows.
If one of the increments $\alpha_{k-1}^{-1}\alpha_k$ of $\alpha$ satisfies $d(\alpha_{k-1}^{-1}\alpha_k)\geq D$, for some $k$ between $-N+1$ and $N$, then
we set $\Psi_r^{(D,N)}(\alpha)=0$.
Otherwise, we set $\Psi_r^{(D,N)}(\alpha)=\Psi_r^{(N)}(\alpha)$.
To simplify the notations below, we use the following convention.
Whenever $\hat{d}(e,\alpha_-)\leq N$, we set $\alpha_{-N}=\alpha_-$ and similarly, whenever $\hat{d}(e,\alpha_+)\leq N$, we set $\alpha_N=\alpha_+$.

Our goal is to prove estimates that will allow us to replace $\Psi_r$ with $\Psi_r^{(D,N)}$.
We start with the following lemma.

\begin{lemma}\label{kappaunchanged}
If $N$ is large enough, depending on $K_1$ and $K_2$, then for every relative geodesic $\alpha$, we have $\kappa_{\alpha}=\kappa_{\alpha^{(N)}}$.
\end{lemma}

\begin{proof}
First note that if $N\geq 2K_1$, $\alpha_{(2K_1)}^{(N)}=\alpha^{(2K_1)}$, so that $f_{\alpha_{(2K_1)}}(\gamma)=0$ if and only if $f_{\alpha_{(2K_1)}^{(N)}}(\gamma)=0$.
In particular, $\kappa_{\alpha}(\gamma)=0$ if and only if $\kappa_{\alpha^{(N)}}(\gamma)=0$.
Hence, we can assume that $f_{\alpha_{(2K_1)}}(\gamma)\neq 0$.
We want to prove that if $N$ is large enough, then the sum $\Sigma$ defined by~(\ref{defsigmatranslateskappa}) is the same for $\alpha$ and for $\alpha^{(N)}$.
By definition of $f$, the fact that $f_{\alpha_{(2K_1)}}(\gamma)\neq 0$ implies that any relative geodesic from $\gamma$ to the right endpoint of $\alpha_{(2K_1)}$ passes within a bounded distance of $\alpha_1$ and any relative geodesic from $\gamma$ to the left endpoint of $\alpha_{(2K_1)}$ passes within a bounded distance of $e$.
Thus, the number of $k$ in the sum defining $\Sigma$ such that $f_{(T^k\alpha)_{(2K_1)}}(\alpha_k^{-1}\gamma)\neq 0$ is finite, with a bound depending only on $K_1$ and $K_2$.
If $N$ is large enough, the same holds replacing $\alpha$ with $\alpha^{(N)}$ and for any such $k$, $f_{(T^k\alpha)_{(2K_1)}}(\alpha_k^{-1}\gamma)=f_{(T^k\alpha^{(N)})_{(2K_1)}}(\alpha_k^{-1}\gamma)$.
This concludes the proof.
\end{proof}

Thus, we do not have to replace $\kappa_{\alpha}$ with $\kappa_{\alpha^{(N)}}$ in the sum~(\ref{defPsi_r}) defining $\Psi_r$, when replacing $\Psi_r(\alpha)$ with $\Psi_r(\alpha^{(N)})$.
This will be very convenient in the following.

\begin{proposition}\label{differencePsietPsitronque}
Let $\epsilon>0$.
Then for $N$ large enough and for $D$ large enough (depending on $N$), for every $r\leq R_\mu$ and for every $n$,
\begin{align*}
    &\left |\sum_{k=0}^{n-1}\sum_{\gamma\in \hat{S}^n}H(e,\gamma|r)\left (\Psi_r \circ T^k([e,\gamma])-\Psi_r^{(D,N)}\circ T^k([e,\gamma])\right )\right |\\
    &\hspace{2cm}\leq \epsilon \sum_{k=0}^{n-1}\sum_{\gamma\in \hat{S}^n}H(e,\gamma|r) \Psi_r \circ T^k([e,\gamma]).
\end{align*}
\end{proposition}

\begin{proof}
We first show that we can replace $\Psi_r$ by $\Psi_r^{(N)}$, that is we prove that if $N$ is large enough, then
\begin{equation}\label{majorationPsietPsi(N)}
\begin{split}
    &\left |\sum_{k=0}^{n-1}\sum_{\gamma\in \hat{S}^n}H(e,\gamma|r)\left (\Psi_r \circ T^k([e,\gamma])-\Psi_r^{(N)}\circ T^k([e,\gamma])\right )\right |\\
    &\hspace{2cm}\lesssim \epsilon \sum_{k=0}^{n-1}\sum_{\gamma\in \hat{S}^n}H(e,\gamma|r) \Psi_r \circ T^k([e,\gamma]).
\end{split}
\end{equation}
Let $n$ and let $k\leq n-1$.
Set $\alpha=T^k([e,\gamma])$.
Then, according to Lemma~\ref{kappaunchanged},
\begin{align*}
    &\sum_{\gamma\in \hat{S}^n}H(e,\gamma|r)\left (\Psi_r \circ T^k([e,\gamma])-\Psi_r^{(N)}\circ T^k([e,\gamma])\right )=\\
    &\frac{1}{I_1(r)}\sum_{\gamma\in \hat{S}^n}H(e,\gamma|r)\sum_{\gamma'\in \Gamma}\kappa_{\alpha}(\gamma')\\
    &\hspace{3cm}\left (\frac{G(\alpha_-,\gamma'|r)G(\gamma',\alpha_+|r)}{G(\alpha_-,\alpha_+|r)}-\frac{G(\alpha_{-N},\gamma'|r)G(\gamma',\alpha_N|r)}{G(\alpha_{-N},\alpha_N|r)}\right ).
\end{align*}
We rewrite
\begin{align*}
    &\left |\frac{G(\alpha_-,\gamma'|r)G(\gamma',\alpha_+|r)}{G(\alpha_-,\alpha_+|r)}-\frac{G(\alpha_{-N},\gamma'|r)G(\gamma',\alpha_N|r)}{G(\alpha_{-N},\alpha_N|r)}\right |\\
    =&\left (\frac{G(\alpha_{-},\gamma'|r)G(\gamma',\alpha_+|r)}{G(\alpha_{-},\alpha_+|r)}\right )
    \left |1-\frac{G(\alpha_{-N},\gamma'|r)G(\gamma',\alpha_N|r)G(\alpha_{-},\alpha_+|r)}{G(\alpha_{-N},\alpha_N|r)G(\alpha_-,\gamma'|r)G(\gamma',\alpha_+|r)}\right |\\
\end{align*}
and
\begin{align*}
    &
    \left |1-\frac{G(\alpha_{-N},\gamma'|r)G(\gamma',\alpha_N|r)G(\alpha_{-},\alpha_+|r)}{G(\alpha_{-N},\alpha_N|r)G(\alpha_-,\gamma'|r)G(\gamma',\alpha_+|r)}\right |\\
    =&
    \left |1-\frac{G(\alpha_{-N},\gamma'|r)G(\alpha_{-},\alpha_+|r)}{G(\alpha_{-N},\alpha_+|r)G(\alpha_-,\gamma'|r)}\frac{G(\alpha_{-N},\alpha_+|r)G(\gamma',\alpha_N|r)}{G(\alpha_{-N},\alpha_N|r)G(\gamma',\alpha_+|r)}\right |
\end{align*}
We now show that
\begin{equation}\label{StrongAnconaarbitrarysmall}
    \left |1-\frac{G(\alpha_{-N},\gamma'|r)G(\alpha_{-},\alpha_+|r)}{G(\alpha_{-N},\alpha_+|r)G(\alpha_-,\gamma'|r)}\frac{G(\alpha_{-N},\alpha_+|r)G(\gamma',\alpha_N|r)}{G(\alpha_{-N},\alpha_N|r)G(\gamma',\alpha_+|r)}\right |.
\end{equation}
is arbitrary small when $N$ is large enough.
Let
$$u_N(\gamma')=\frac{G(\alpha_{-N},\gamma'|r)G(\alpha_{-},\alpha_+|r)}{G(\alpha_{-N},\alpha_+|r)G(\alpha_-,\gamma'|r)}$$
and
$$v_N(\gamma')=\frac{G(\alpha_{-N},\alpha_+|r)G(\gamma',\alpha_N|r)}{G(\alpha_{-N},\alpha_N|r)G(\gamma',\alpha_+|r)}.$$
so that~(\ref{StrongAnconaarbitrarysmall}) can be written as $|1-u_N(\gamma')v_N(\gamma')|$.
Assume that $\kappa_{\alpha}(\gamma')\neq 0$.
Then $f_{\alpha_{(2K_1)}}(\gamma')\neq 0$ and so any relative geodesic from $\gamma'$ to the left endpoint of $\alpha_{(2K_1)}$ passes within $K_1$ of $e$.
This implies that relative geodesics from $\alpha_{-N}$ to $\gamma'$ and from $\alpha_-$ to $\alpha_+$ fellow travel for a time at least $N'$, where $N'$ tends to infinity as $N$ tends to infinity.
Strong relative Ancona inequalities show that
$$\left |1-u_N(\gamma')\right |\leq C\rho^{N'}.$$
Similarly, one can prove that
$$\left |1-v_N(\gamma')\right |\leq C\rho^{N'}.$$
Also, weak relative Ancona inequalities imply that
$v_N(\gamma')$ is bounded.
This yields
$$\left |1-u_N(\gamma')v_N(\gamma')\right |\leq v_N(\gamma')|1-u_N(\gamma')|+|1-v_N(\gamma')|\leq C'\rho^{N'}.$$
Hence,
\begin{align*}
    &\left |\sum_{\gamma\in \hat{S}^n}H(e,\gamma|r)\left (\Psi_r \circ T^k([e,\gamma])-\Psi_r^{(N)}\circ T^k([e,\gamma])\right )\right |\\
    &\hspace{1cm}\leq C'\rho^{N'}\frac{1}{I^{(1)}(r)}\sum_{\gamma\in \hat{S}^n}H(e,\gamma|r)\sum_{\gamma'\in \Gamma}\kappa_{\alpha}(\gamma') \frac{G(\alpha_-,\gamma'|r)G(\gamma',\alpha_+|r)}{G(\alpha_-,\alpha_+|r)}\\
    &\hspace{1cm}=C'\rho^{N'}\sum_{\gamma\in \hat{S}^n}H(e,\gamma|r)\Psi_r \circ T^k([e,\gamma]).
\end{align*}
Thus, if $N$ is large enough, then~(\ref{majorationPsietPsi(N)}) holds.

Let us compare $\Psi_r^{(D,N)} \circ T^k([e,\gamma])$ and $\Psi_r^{(N)}\circ T^k([e,\gamma])$ now.
Let $\alpha=T^k[e,\gamma]$.
Then,
$\Psi_r^{(D,N)} \circ T^k([e,\gamma])-\Psi_r^{(N)}\circ T^k([e,\gamma])$ is non-zero only for elements $\gamma$ such that there exists $j$ between $-N+1$ and $N$ such that $d(\alpha_{j-1},\alpha_{j})\geq D$.
Denote by $\gamma_0=e,\gamma_1,...,\gamma_n=\gamma$ successive elements on $[e,\gamma]$, so that
$\alpha_{j}=\gamma_k^{-1}\gamma_{j+k}$.
Hence, $\Psi_r^{(D,N)} \circ T^k([e,\gamma])-\Psi_r^{(N)}\circ T^k([e,\gamma])$ is non-zero only for elements $\gamma$ such that there exists $j$ between $-N+k+1$ and $N+k$ such that $d(\gamma_{j-1},\gamma_{j})\geq D$.

Let $\hat{S}^n_{\geq D}$ be the set of $\gamma\in \hat{S}^n$ such that one of the increments of the relative geodesic $[e,\gamma]$ between $-N+k+1$ and $N+k$ has length at least $D$.
Also, for a fixed $j$ between $-N+k+1$ and $N+k$, let $\hat{S}^{n,j}_{\geq D}$ be the subset of $\hat{S}^n_{\geq D}$ of elements $\gamma$ such that the first such increment is at step $j$.
Then,
\begin{align*}
    &\sum_{\gamma\in \hat{S}^n}H(e,\gamma|r)\left (\Psi_r^{(N)} \circ T^k([e,\gamma])-\Psi_r^{(D,N)}\circ T^k([e,\gamma])\right )\\
&\hspace{1cm}=\frac{1}{I^{(1)}(r)}\sum_{\gamma\in \hat{S}^n_{\geq D}}H(e,\gamma|r)\sum_{\gamma'\in \Gamma}\kappa_{\alpha}(\gamma')\frac{G(\alpha_{-N},\gamma'|r)G(\gamma',\alpha_N|r)}{G(\alpha_{-N},\alpha_N|r)}\\
&\hspace{1cm}=\frac{1}{I^{(1)}(r)}\sum_{j=-N+k+1}^{N+k}\sum_{\gamma\in \hat{S}^{n,j}_{\geq D}}H(e,\gamma|r)\sum_{\gamma'\in \Gamma}\kappa_{\alpha}(\gamma')\frac{G(\alpha_{-N},\gamma'|r)G(\gamma',\alpha_N|r)}{G(\alpha_{-N},\alpha_N|r)}.
\end{align*}
Fix $j$.
For $\gamma\in \hat{S}^n$, we write $\gamma=\gamma_1\sigma\gamma_2$, where $\gamma_1\in \hat{S}^{j-1}$, $\sigma$ is in a factor $\mathcal{H}_k'$ and $\gamma'\in \hat{S}^{n-j}$.
If $\gamma\in \hat{S}^{n,j}_{\geq D}$, then $d(e,\sigma)\geq D$.
Weak relative Ancona inequalities show that
$$H(e,\gamma|r)\lesssim H(e,\gamma_1|r)H(e,\sigma|r)H(e,\gamma_2|r).$$
Also, using~(\ref{upsilonvspsi}), we can replace $\Psi_r$ with $\Upsilon_r$ in the right member of the sum above.
We get
\begin{equation}\label{differenceNetD_1N}
\begin{split}
    &\sum_{\gamma\in \hat{S}^n}H(e,\gamma|r)\left (\Psi_r^{(N)} \circ T^k([e,\gamma])-\Psi_r^{(D,N)}\circ T^k([e,\gamma])\right )\\
&\hspace{1cm}\lesssim \sum_{j=-N+k+1}^{N+k}\sum_{\gamma\in \hat{S}^{n,j}_{\geq D}} H(e,\gamma_1|r)H(e,\sigma|r)H(e,\gamma_2|r)\Upsilon_r(T^k\alpha).
\end{split}
\end{equation}
Recall that $X_x^1$ is the set of symbols that can precede $x$ in $\Sigma_A$.
More generally, $X_x^m$ is the set of words of length $m$ that can precede $x$.
Decompose the sum over $\gamma$ as follows.
\begin{align*}
    &\sum_{\gamma\in \hat{S}^n}H(e,\gamma|r)\left (\Psi_r^{(N)} \circ T^k([e,\gamma])-\Psi_r^{(D,N)}\circ T^k([e,\gamma])\right )\\
&\hspace{.5cm}\lesssim \sum_{j=-N+k+1}^{N+k}\sum_{\gamma_2\in \hat{S}^{n-j}}\sum_{\underset{d(e,\sigma)\geq D}{\sigma \in X_{\gamma_2}^1}}\sum_{\gamma_1\in X_{\sigma\gamma_2}^{j-1}} H(e,\gamma_1|r)H(e,\sigma|r)H(e,\gamma_2|r)\Upsilon_r(T^k\alpha).
\end{align*}
Note that $\Upsilon_r(T^k\alpha)$ only depends on the $k$th increment of $[e,\gamma]$.
In particular, for $j\neq k$, we can factorize the sum over $\gamma_1,\sigma,\gamma_2$ by $\Upsilon_r(T^k\alpha)$.
Hence, we can bound the terms $j\neq k$ by
$$\Upsilon_r(T^k\alpha)\sum_{j=-N+k+1}^{N+k}\sum_{\gamma_2\in \hat{S}^{n-j}}H(e,\gamma_2|r)\sum_{\underset{d(e,\sigma)\geq D}{\sigma \in X_{\gamma_2}^1}}H(e,\sigma|r)\sum_{\gamma_1\in X_{\sigma\gamma_2}^{j-1}}H(e,\gamma_1|r).$$
Corollary~\ref{derivativeparabolicGreenfinite} shows that
$$\sum_{\sigma\in X_{\gamma_2}^1}H(e,\sigma|r)$$ is uniformly bounded.
Thus, for large enough $D$,
\begin{equation}\label{estimeejneqk}
    \sum_{\underset{d(e,\sigma)\geq D}{\sigma \in X_{\gamma_2}^1}}H(e,\sigma|r)\leq \frac{\epsilon}{2N}\sum_{\sigma\in X_{\gamma_2}^1}H(e,\sigma|r).
\end{equation}
Let us focus on the term $j=k$.
We can still factorize the sum over $\gamma_1$ by $\Upsilon_r(T^k\alpha)$.
We want to bound the sum over $\sigma$.
According to the definition of $\Upsilon_r$, we thus need to bound
\begin{align*}
    &\sum_{\underset{d(e,\sigma)\geq D}{\sigma \in X_{\gamma_2}^1}}H(e,\sigma|r)\sum_{\sigma'\in \mathcal{H}_{\sigma}}\frac{G(e,\sigma'|r)G(\sigma',\sigma|r)}{G(e,\sigma|r)}\\
    &\hspace{3cm}=\sum_{\underset{d(e,\sigma)\geq D}{\sigma \in X_{\gamma_2}^1}}\sum_{\sigma'\in \mathcal{H}_{\sigma}}G(e,\sigma'|r)G(\sigma',\sigma|r)G(\sigma,e|r).
\end{align*}
Since $\mu$ is not spectrally degenerate, the sum
$$\sum_{\sigma \in X_{\gamma_2}^1}\sum_{\sigma'\in \mathcal{H}_{\sigma}}G(e,\sigma'|r)G(\sigma',\sigma|r)G(\sigma,e|r)$$
is uniformly bounded.
Thus, for large enough $D$,
\begin{equation}\label{estimeej=k}
\begin{split}
    &\sum_{\underset{d(e,\sigma)\geq D}{\sigma \in X_{\gamma_2}^1}}\sum_{\sigma'\in \mathcal{H}_{\sigma}}G(e,\sigma'|r)G(\sigma',\sigma|r)G(\sigma,e|r)\\
&\hspace{3cm}\leq \frac{\epsilon}{2N}\sum_{\sigma \in X_{\gamma_2}^1}\sum_{\sigma'\in \mathcal{H}_{\sigma}}G(e,\sigma'|r)G(\sigma',\sigma|r)G(\sigma,e|r).
\end{split}
\end{equation}
When $j$ is fixed, there is a unique way of decomposing $\gamma$ as $\gamma_1\sigma\gamma_2$.
Hence, combining~(\ref{differenceNetD_1N}),~(\ref{estimeejneqk}) and~(\ref{estimeej=k}), we get
\begin{equation}\label{majorationPsi(N)etPsi(D,N)}
\begin{split}
&\sum_{\gamma\in \hat{S}^n}H(e,\gamma|r)\left (\Psi_r^{(N)} \circ T^k([e,\gamma])-\Psi_r^{(D,N)}\circ T^k([e,\gamma])\right )\\
&\hspace{1cm}\lesssim \sum_{j=-N+k+1}^{N+k}\frac{\epsilon}{2N}\sum_{\gamma\in \hat{S}^n}H(e,\gamma|r)\Psi_r^{(N)} \circ T^k([e,\gamma])\\
&\hspace{1cm}\leq \epsilon\sum_{\gamma\in \hat{S}^n}H(e,\gamma|r)\Psi_r^{(N)} \circ T^k([e,\gamma]).
\end{split}
\end{equation}
Finally, combining~(\ref{majorationPsietPsi(N)}) and~(\ref{majorationPsi(N)etPsi(D,N)}), we get the desired inequality.
\end{proof}

Recall that we want to compare $I^{(2)}(r)$ and $I^{(1)}(r)$.
As we saw,
$$I^{(2)}(r)=\sum_{n\geq 0}\sum_{\gamma\in \hat{S}^n}H(e,\gamma|r)\Phi_r(\gamma).$$
Proposition~\ref{propestimeePhietPsi} thus yields
$$I^{(2)}(r)=I^{(1)}(r)\sum_{n\geq 0}\sum_{\gamma\in \hat{S}^n}\sum_{k=0}^{n-1}H(e,\gamma|r)\Psi_r(T^k[e,\gamma])+O\left (I^{(1)}(r)^2\right ).$$
We want to prove that
$$I^{(2)}(r)=\xi I^{(1)}(r)^3+O\left (I^{(1)}(r)^2\right ),$$
so that we only have to deal with
$$I^{(1)}(r)\sum_{n\geq 0}\sum_{\gamma\in \hat{S}^n}\sum_{k=0}^{n-1}H(e,\gamma|r)\Psi_r(T^k[e,\gamma]).$$
In view of Proposition~\ref{differencePsietPsitronque}, we can replace
$\Psi_r$ with $\Psi_r^{(D,N)}$.

\medskip
We now consider the set $\overline{\Sigma}_{A,\Z}$ of (finite or infinite) sequences $x=(x_n)$ indexed by $\Z$ such that $x_n\in \Sigma$ and for every $n$, $x_n$ and $x_{n+1}$ are adjacent edges in the automaton $\mathcal{G}$.
The map $T$ still defines a shift on $\overline{\Sigma}_{A,\Z}$.

Since $\Psi^{(D,N)}(\alpha)$ only depends on the truncated geodesic $\alpha^{(N)}$, $\Psi^{(D,N)}$ can be extended to a function defined on (finite or infinite) relative geodesics.
For any $x\in \overline{\Sigma}_{A,\Z}$, $(...,x_{-n},...,x_0,....x_n,...)$ defines such a relative geodesic, so
$\Psi^{(D,N)}\circ \phi$ is a well defined function on $\overline{\Sigma}_{A,\Z}$.
We will omit the reference to $\phi$ and see $\Psi^{(D,N)}$ as a function on $\overline{\Sigma}_{A,\Z}$ to simplify.
Also, since $\Psi^{(D,N)}(\alpha)$ only depends on the truncated relative geodesic $\alpha^{(N)}$ and vanishes on relative geodesics $\alpha$ whose increments are too long, the induced function on $\overline{\Sigma}_{A,\Z}$ only depends on a finite number of symbols.

For a continuous function $f:\overline{\Sigma}_{A,\Z}\rightarrow \R$, we define
$$\tilde{V}_n(f)=\sup \{|f(x)-f(y)|,x_{-n}=y_{-n},...,x_0=y_0,...,x_n=y_n\}.$$
Letting $0<\rho<1$, we  say that $f$ is $\rho$-locally H\"older if there exists $C\geq 0$ such that
$$\forall n\geq 1,V_n(f)\leq C\rho^n.$$
As before, we do not ask anything on $V_0(f)$ and $f$ can be unbounded.
Say that $f$ is locally H\"older if it is $\rho$-locally H\"older for some $\rho$.
Define the H\"older norm $D_{\rho}$ as
$$D_{\rho}(f)=\sup_n \frac{\tilde{V}_n(f)}{\rho^n}.$$
Also let $H_{\rho}$ be the set of bounded $\rho$-locally H\"older functions and define the norm
$$\|\cdot \|_{\rho}=D_{\rho}+\|\cdot \|_{\infty}$$
on this space.
Then, $(H_{\rho},\|\cdot\|_{\rho})$ is a Banach space.

We want to use Proposition~\ref{propestimeePhietPsi} and apply the transfer operator to $\Psi_r^{(D,N)}$.
To apply this operator, we first need to transform $\Psi_r^{(D,N)}$ into a function only depending on the future, that is a function on $\overline{\Sigma}_A$.
We start proving the following.

\begin{lemma}\label{Lemme3.22Gouezel}
Fix $D$ and $N$.
The functions $\Psi_r^{(D,N)}$ are $\rho$-locally H\"older and are uniformly bounded.
They uniformly converge in $(H_{\rho,\beta},\|\cdot \|_{\rho,\beta})$ to a function $\Psi_{R_{\mu}}^{(D,N)}$, as $r$ tends to $R_{\mu}$.
\end{lemma}

\begin{proof}
We first show that $\Psi_r^{(D,N)}$ is uniformly bounded.
Recall that
$$\Psi_r^{(D,N)}(\alpha)=\frac{1}{I^{(1)}(r)}\sum_{\gamma\in \Gamma}\kappa_{\alpha}(\gamma)\frac{G(\alpha_{-N},\gamma|r)G(\gamma,\alpha_N|r)}{G(\alpha_{-N},\alpha_N|r)}.$$
Denote by $\Gamma_k$ the set of $\gamma$ whose projection on $\alpha^{(N)}$ is on $\alpha_{k+1}$, where we choose the projection which is the closest to $\alpha_N$.
Also let $\mathcal{H}_k$ be the union of parabolic subgroups containing $\alpha_k^{-1}\alpha_{k+1}$.
Then, weak relative Ancona inequalities, together with Lemma~\ref{lemmarelativetripod} show that
$$\sum_{\gamma\in \Gamma_k}\frac{G(\alpha_{-N},\gamma|r)G(\gamma,\alpha_N|r)}{G(\alpha_{-N},\alpha_N|r)}\lesssim I^{(1)}(r)\sum_{\sigma \in \mathcal{H}_k}\frac{G(\alpha_k,\sigma|r)G(\sigma,\alpha_{k+1}|r)}{G(\alpha_{k},\alpha_{k+1}|r)}.$$
Since $\kappa_{\alpha}$ is bounded, we thus have
$$\Psi_r^{(D,N)}(\alpha)\leq C_{\alpha},$$
where $C_\alpha$ only depends on $\alpha$.
Actually, since $\Psi_r^{(D,N)}(\alpha)$ is non-zero for a finite number of $\alpha$ which only depends on $N$ and $D$, $C_\alpha$ also only depends on $D$ and $N$.
Moreover, $\Psi^{(D,N)}(\alpha)$ only depends on $\alpha^{(N)}$, so it is $\rho$-locally H\"older and $\|\Psi^{(D,N)} \|_{\rho,\beta}$ is bounded by some number only depending on $D$ and $N$.

Finally, since $\Psi^{(D,N)}(\alpha)$ only depends on a finite number of symbols, convergence in $(H_{\rho,\beta},\|\cdot \|_{\rho,\beta})$ is equivalent to pointwise convergence.
Let us fix $\alpha$ and prove that $\Psi_r^{(D,N)}(\alpha)$ converges to a function $\Psi_{R_\mu}^{(D,N)}(\alpha)$, as $r$ tends to $R_\mu$.
To do so, we express $\Psi_r^{(D,N)}$ as a sum using the transfer operator.
We introduce a function $\psi_r$ on $\Gamma$ as follows.
We set
$$\psi_r(\gamma)=\kappa_{\alpha}(\gamma)\frac{G(\alpha_{-N},\gamma|r)G(\gamma,\alpha_N|r)}{G(\alpha_{-N},\alpha_N|r)H(e,\gamma|r)}$$
for any relative geodesic $\alpha$ such that $\Psi_r^{(N,D)}(\alpha)\neq 0$.
Otherwise, we set $\psi_r=0$.
Weak relative Ancona inequalities imply that $\gamma \mapsto G(\alpha_{-N},\gamma|r)G(\gamma,\alpha_N|r)$ can be extended to $\partial \hat{\Gamma}$.
Since $\kappa$ is defined on the whole Bowditch compactification, $\psi_r$ can also be extended to $\Gamma \cup \partial \hat{\Gamma}$, so $\psi_r\circ \phi$ is a function on $\overline{\Sigma}_A$.
Note that
\begin{equation}\label{Psienfonctiondepsi}
    \Psi_r^{(D,N)}(\alpha)=\frac{1}{I^{(1)}(r)}\frac{1}{H(e,e|r)}\sum_{n\geq 0}\mathcal{L}_r^n(1_{E_*}\psi_r\circ \phi)(\emptyset).
\end{equation}

We want to apply~\ref{coroKellerLiveranisatisfait} to prove that $\Psi_r$ converges, so we have to transform $\psi_r$ into a locally H\"older function.
First, $\psi_{r}$ is defined using the function $\kappa_{\alpha}$ which is only continuous.
We again have to truncate $\psi_r$ to conclude our proof.

Fix $N'$ and let $\gamma_{N'}$ be the $N'$th element on the relative geodesic $[e,\gamma]$ whenever $\hat{d}(e,\gamma)\geq N'$ and $\gamma_{N'}=\gamma$ otherwise.
Set then
$$\psi_r^{(N')}=\kappa_{\alpha}(\gamma_{N'})\frac{G(\alpha_{-N},\gamma|r)G(\gamma,\alpha_N|r)}{G(\alpha_{-N},\alpha_N|r)H(e,\gamma|r)}.$$
The functions $\psi_r$ and $\psi_r^{(N')}$ implicitly depend on $\alpha$ and on $N$ and $D$.
Actually, Lemma~\ref{kappaunchanged} shows that they do not depend on $\alpha$, but only on $\alpha^{(N)}$.

\begin{lemma}\label{lemmeduN'}
For every $\epsilon>0$, for every $N$ and every $D$, there exists $N'_0$ such that for every $N'\geq N'_0$, for every $\alpha$ and for every $r<R_\mu$,
$$\left |\psi_r-\psi_r^{(N')}\right |\leq \epsilon.$$
\end{lemma}

\begin{proof}
Let $\epsilon>0$.
The function $\kappa_{\alpha}$ is continuous on the Bowditch compactification.
Endow this compactification with any distance $d$.
We can extend the definition of $\gamma_{N'}$ to any infinite relative geodesic $\alpha$ declaring $\alpha_{N'}$ to be the $N'$th point on $\alpha$.
Then $\alpha_{N'}$ uniformly converges to the conical limit point defined by $\alpha$, as $N'$ tends to infinity.
So, for any $\delta>0$, if $N'$ is large enough, then $d(\gamma,\gamma_{N'})\leq \delta$.
Note that this can be easily directly shown if one chooses the shortcut metric on the Bowditch compactification defined in \cite[Définition~2.6]{GerasimovPotyagailo2}.
By compactness, $\kappa_{\alpha}$ is uniformly continuous.
Hence, for $N'$ large enough,
$\left |\kappa_{\alpha}(\gamma_{N'})-\kappa_{\alpha}(\gamma)\right |\leq \epsilon$.
Thus,
$$\left |\psi_r(\gamma)-\psi_r^{(N')}(\gamma)\right |\leq \epsilon\frac{G(\alpha_{-N},\gamma|r)G(\gamma,\alpha_N|r)}{G(\alpha_{-N},\alpha_N|r)H(e,\gamma|r)}\lesssim C_\alpha\epsilon$$
where $C_\alpha$ only depends on $\alpha$.
The integer $N'_0$ a priori depends on $\alpha$, because of $C_\alpha$ in the upper-bounded above and because uniform continuity of $\kappa_{\alpha}$ depends on $\alpha$.
However, $\psi_r$ and $\psi_r^{(N')}$ are the null function except for a finite number of relative geodesics $\alpha$ which only depends on $N$ and $D$.
This concludes the proof.
\end{proof}

To show that $\Psi_r^{(D,N)}(\alpha)$ converges, it is enough to prove it is Cauchy, that is for every $\epsilon>0$, there exists $r_0<R_\mu$ such that for any $r,r'\in [r_0,R_\mu)$,
$$\left |\Psi_r^{(D,N)}(\alpha)-\Psi_{r'}^{(D,N)}(\alpha)\right |\leq \epsilon.$$
Fix $\epsilon>0$.
Let $N'$ be given by Lemma~\ref{lemmeduN'} so that for every $r<R_\mu$,
$$\left |\psi_r-\psi_r^{(N')}\right |\leq \epsilon.$$
According to~(\ref{Psienfonctiondepsi}),
\begin{align*}
    &\left |\Psi_r^{(D,N)}(\alpha)-\Psi_{r'}^{(D,N)}(\alpha)\right |\\
    &\hspace{1cm}\lesssim 2\epsilon+\frac{1}{I^{(1)}(r)}\left |\sum_{n\geq 0}\mathcal{L}_{r}^n(1_{E_*}\psi^{(N')}_{r}\circ \phi)(\emptyset)-\sum_{n\geq 0}\mathcal{L}_{r'}^n(1_{E_*}\psi^{(N')}_{r'}\circ \phi)(\emptyset)\right |.
\end{align*}
We thus only need to prove that
$$\frac{1}{I^{(1)}(r)}\sum_{n\geq 0}\mathcal{L}_{r}^n(1_{E_*}\psi^{(N')}_{r}\circ \phi)(\emptyset)$$
converges, as $r$ tends $R_\mu$.
Note that the functions $\gamma \mapsto \frac{G(\alpha_{-N},\gamma|r)G(\gamma,\alpha_N|r)}{G(\alpha_{-N},\alpha_N|r)}$ and $\gamma\mapsto \kappa_{\alpha_{(2K_1)}}(\gamma_{N'})$ are bounded and locally H\"older, so $\psi_r\circ \phi$ lies in $H_{\rho,\beta}$.

To prove that the above sum, we need to prove that $\psi_r\circ \phi$ uniformly converges to $\psi_{R_\mu}\circ \phi$.
This is not obvious and so we truncate $\psi_r$ as we truncated $\Psi_r$.
Fix another constant $D'$.
For $\gamma\in \Gamma$ let $[e,\gamma]=(e,\gamma_1,...,\gamma_n=\gamma)$ be the relative geodesic from $e$ to $\gamma$ given by the automaton $\mathcal{G}$.
Set then $\psi_r^{(D',N')}(\gamma)=0$ is one of the increments of $[e,\gamma]$ is at least $D'$ and set
$\psi_r^{(D',N')}(\gamma)=\psi_r^{(N')}(\gamma_{N'})$ otherwise.
Since $\psi_r\circ \phi$ is bounded and locally H\"older, the same proof as the proof of Proposition~\ref{differencePsietPsitronque} shows that for every $\eta>0$, for large enough $N'$ and $D'$,
\begin{equation}\label{differencelittlepsiandlittlepsitruncated}
    \frac{1}{I^{(1)}(r)}\sum_{n\geq 0}\mathcal{L}_{r}^n\left (1_{E_*}\left |\psi^{(N')}_{r}\circ \phi-\psi^{(N',D')}\circ \phi\right |\right )(\emptyset)\leq \eta.
\end{equation}

\begin{remark}
It might seem strange that we first had to truncate $\kappa_{\alpha}$ when defining $\psi_r^{(N')}$, before truncating again to define $\psi_r^{(D',N')}$.
However, to apply the same strategy as in Proposition~\ref{differencePsietPsitronque}, we needed to know a priori that our function was locally H\"older.
\end{remark}

Once again, to prove that
$$\frac{1}{I^{(1)}(r)}\sum_{n\geq 0}\mathcal{L}_{r}^n(1_{E_*}\psi^{(N')}_{r}\circ \phi)(\emptyset)$$
converges, it is enough to prove that this quantity is Cauchy, as $r$ tends to $R_\mu$.
In view of~(\ref{differencelittlepsiandlittlepsitruncated}), we thus only need to prove that
$$\frac{1}{I^{(1)}(r)}\sum_{n\geq 0}\mathcal{L}_{r}^n(1_{E_*}\psi^{(N',D')}_{r}\circ \phi)(\emptyset)$$
converges.
The function $\psi^{(N',D')}_{r}\circ \phi$ is bounded and locally H\"older.
Moreover, whenever $x,y,z$ are fixed,
$r\mapsto \frac{G(x,y|r)G(y,z|r)}{G(x,z|r)H(e,y|r)}$ is a continuous function.
It converges to $\frac{G(x,y|R_\mu)G(y,z|R_\mu)}{G(x,z|R_\mu)H(e,y|R_\mu)}$, as $r$ tends to $R_\mu$.
Hence, $\psi^{(N',D')}_{r}\circ \phi$ converges to a function $\psi^{(N',D')}_{R_\mu}\circ \phi$.
Also, $\psi^{(N',D')}_{r}\circ \phi$ only depends on a finite number of symbols, so this convergence also holds in $(H_{\rho,\beta},\|\cdot \|_{\rho,\beta})$.
Now that every parameter is fixed, we set $f=1_{E_*}\psi^{(N',D')}_{R_\mu}\circ \phi$ for convenience.
We are left to proving that
$$\frac{1}{I^{(1)}(r)}\sum_{n\geq 0}\mathcal{L}_{r}^nf(\emptyset)$$
converges, as $r$ tends to $R_\mu$, which is a direct consequence of~(\ref{estimeesommeoperateurtransfert}).
\end{proof}

\subsection{From the double sided to the one-sided shift}
As announced, to study
$$\sum_{n\geq 0}\sum_{\gamma\in \hat{S}^n}\sum_{k=0}^{n-1}H(e,\gamma|r)\Psi_r^{(D,N)}(T^k[e,\gamma]),$$
we will express this sum with the transfer operator and then use Theorem~\ref{GouezelSarig}, exactly like in the proof of Lemma~\ref{Lemme3.22Gouezel}.
However, we cannot apply the transfer operator to the function $\Psi_r^{(D,N)}$, which depends both on past and future.

We use the following trick.

\begin{lemma}\label{passagededoubleshiftashiftsimple}
Let $f$ be a $\rho$-locally H\"older function on $\overline{\Sigma}_{A,\Z}$.
Then, there exist $\rho^{1/2}$-locally H\"older functions $g$ and $u$ on $\overline{\Sigma}_{A,\Z}$ such that
$$f=g+u-u\circ T.$$
Moreover, $g(x)=g(y)$ as soon as $x_n=y_n$ for every non-negative $n$, so that $g$ induces a function on $\overline{\Sigma}_A$.
Also, if $f$ is bounded, then $g$ and $u$ also are bounded and the maps
$$f\in (H_{\rho},\|\cdot\|_{\rho})\mapsto g\in (H_{\rho^{1/2}},\|\cdot\|_{\rho^{1/2}}), \hspace{.5cm} f\in (H_{\rho},\|\cdot\|_{\rho})\mapsto u\in (H_{\rho^{1/2}},\|\cdot\|_{\rho^{1/2}})$$
are continuous.
\end{lemma}

This is proved in \cite[Proposition~1.2]{ParryPollicott} for finite-type shifts.
However, the proof does not use that the set of symbols is finite.

According to Lemma~\ref{Lemme3.22Gouezel}, the functions $\Psi_r^{(D,N)}$ are bounded and locally H\"older on $\overline{\Sigma}_{A,\Z}$ and they converge in $(H_{\rho,\beta},\|\cdot \|_{\rho,\beta})$ to a function $\Psi_{R_\mu}^{(D,N)}$.
We thus obtain from Lemma~\ref{passagededoubleshiftashiftsimple} functions $\tilde{\Psi}_r^{(D,N)}$, $r\leq R_\mu$ defined on $\overline{\Sigma}_A$ and functions $u_r^{(D,N)}$ defined on $\overline{\Sigma}_{A,\Z}$ such that
$$\Psi_r^{(D,N)}=\tilde{\Psi}_r^{(D,N)}+u_r^{(D,N)}-u_r^{(D,N)}\circ T.$$
For any $x\in \overline{\Sigma}_A$ of length $n$,
$$\sum_{k=0}^{n-1}\Psi_r^{(D,N)}(T^kx)=\sum_{k=0}^{n-1}\tilde{\Psi}_r^{(D,N)}(T^kx)+u_r^{(D,N)}(x)-u_r^{(D,N)}(T^nx).$$
The functions $u_r^{(D,N)}$ are bounded by some number that only depends on $D$ and $N$, so
$$\sum_{k=0}^{n-1}\Psi_r^{(D,N)}(T^kx)=\sum_{k=0}^{n-1}\tilde{\Psi}_r^{(D,N)}(T^kx)+O_{D,N}(1).$$
Thus,
\begin{equation}\label{dePsiatildePsiCVfinale}
\begin{split}
    &\sum_{n\geq 0}\sum_{\gamma\in \hat{S}^n}\sum_{k=0}^{n-1}H(e,\gamma|r)\Psi_r^{(D,N)}(T^k[e,\gamma])\\
    &\hspace{2cm}=\sum_{n\geq 0}\sum_{\gamma\in \hat{S}^n}\sum_{k=0}^{n-1}H(e,\gamma|r)\tilde{\Psi}_r^{(D,N)}(T^k[e,\gamma])+O_{D,N}\left (I^{(1)}(r)\right ).
\end{split}
\end{equation}
Since $\tilde{\Psi}_r^{(D,N)}$ only depends on the future, we rewrite this as
\begin{equation}\label{reecrituretildePsiCVfinale}
\begin{split}
    &\sum_{n\geq 0}\sum_{\gamma\in \hat{S}^n}\sum_{k=0}^{n-1}H(e,\gamma|r)\tilde{\Psi}_r^{(D,N)}(T^k[e,\gamma])\\
    &\hspace{2cm}=H(e,e|r)\sum_{n\geq 0}\mathcal{L}_r^n\left (1_{E_*}\sum_{k=0}^{n-1}\tilde{\Psi}_r^{(D,N)}\circ T^k\right )(\emptyset).
\end{split}
\end{equation}

\subsection{Proof of Proposition~\ref{propositioncomparaisonI_2etI_1} : convergence of $I^{(2)}(r)/I^{(1)}(r)^3$}
We first prove that the quantity~(\ref{reecrituretildePsiCVfinale}) is asymptotic to $\xi_{D,N}I^{(1)}(r)^2$, as $r$ tends to $R_\mu$, where $\xi_{D,N}$ is some number only depending on $D$ and $N$.
Since $\Psi_r^{(D,N)}$ converges in $(H_{\rho,\beta},\|\cdot \|_{\rho,\beta})$ to $\Psi_{R_\mu}^{(D,N)}$, we deduce from Lemma~\ref{Lemme3.22Gouezel}, up to changing $\rho$, that $\tilde{\Psi}_r^{(D,N)}$ converges in $(H_{\rho,\beta},\|\cdot \|_{\rho,\beta})$ to $\tilde{\Psi}_{R_\mu}^{(D,N)}$.
We thus only need to prove that
$$\sum_{n\geq 0}\mathcal{L}_r^n\left (1_{E_*}\sum_{k=0}^{n-1}\tilde{\Psi}_{R_\mu}^{(D,N)}\circ T^k\right )(\emptyset)$$
is asymptotic to $\xi_{D,N}I^{(1)}(r)^2$.
Recall that $\mathcal{L}_r(u\cdot v\circ T)=v\mathcal{L}_r(u)$, so that
\begin{equation}\label{CVfinale00}
\begin{split}
\sum_{n\geq 0}\mathcal{L}_r^n \left (1_{E_*}\sum_{k=0}^{n-1}\tilde{\Psi}_{R_\mu}^{(D,N)}\circ T^k\right )&=\sum_{n\geq 0}\sum_{k=1}^n\mathcal{L}_r^k(\tilde{\Psi}_{R_\mu}^{(D,N)}\mathcal{L}_r^{n-k}1_{E_*})\\
&=\sum_{k\geq 1}\mathcal{L}_r^k\left (\tilde{\Psi}_{R_\mu}^{(D,N)}\sum_{n\geq 0}\mathcal{L}_r^{n}1_{E_*}\right ).
\end{split}
\end{equation}

From Corollary~\ref{coroKellerLiveranisatisfait}, we deduce that for any $r$ close enough to $R_{\mu}$ and for any $x\in \overline{\Sigma}_A$,
\begin{equation}\label{CVfinale01}
    \sum_{n\geq 0}\mathcal{L}_r^n1_{E_*}(x)=\sum_{j=1}^k\frac{1}{\tilde{P}_j(r)}\sum_{i=1}^{p_j}\tilde{h}_{j,r}^{(i)}(x)\int 1_{E_*} d\tilde{\nu}_{j,r}^{((i-n) \text{ mod } p_j)}+O(1).
\end{equation}
Let
$$\alpha_{i,j,r}=\int 1_{E_*} d\tilde{\nu}_{j,r}^{((i-n) \text{ mod } p_j)},$$
so that $\alpha_{i,j,r}$ converges to $\alpha_{i,j}$, as $r$ tends to $R_\mu$.
We now estimate
$$\sum_{k\geq 1}\mathcal{L}_r^k\left (\tilde{\Psi}_{R_\mu}^{(D,N)}\sum_{j=1}^k\frac{1}{\tilde{P}_j(r)}\sum_{i=1}^{p_j}\alpha_{i,j,r}\tilde{h}_{j,r}^{(i)}\right )(\emptyset).$$
According to~(\ref{tildehtendsimplementversh}), $\tilde{h}_{j,r}^{(i)}(\emptyset)$ converges to $h_j^{(i)}(\emptyset)$, as $r$ tends to $R_\mu$, so we can start the above sum at $k=0$.
Fix $j$ and let $1\leq i \leq p_j$.
We use again Corollary~\ref{coroKellerLiveranisatisfait} to obtain
\begin{equation}\label{CVfinale02}
\begin{split}
    &\sum_{k\geq 0}\mathcal{L}_r^k\left (\tilde{\Psi}_{R_\mu}^{(D,N)}\tilde{h}_{j,r}^{(i)}\right )(\emptyset)\\
    &=\sum_{j'}\frac{1}{|\tilde{P}_{j'}(r)|}\sum_{i'=1}^{p_{j'}}\tilde{h}_{j',r}^{(i')}(\emptyset)\int \tilde{\Psi}_{R_\mu}^{(D,N)}\tilde{h}_{j,r}^{(i)} d\tilde{\nu}_{j',r}^{((i'-n) \text{ mod } p_{j'})}+O(1).
\end{split}
\end{equation}
We show that for every $j',i'$,
$$\int \tilde{\Psi}_{R_\mu}^{(D,N)}\tilde{h}_{j',r}^{(i')} d\tilde{\nu}_{j',r}^{((i'-n) \text{ mod } p_{j'})}$$
converges, as $r$ tends to $R_\mu$.
Write
\begin{align*}
    \int \tilde{\Psi}_{R_\mu}^{(D,N)}\tilde{h}_{j',r}^{(i')} d\tilde{\nu}_{j',r}^{((i'-n) \text{ mod } p_{j'})}=&\int \tilde{\Psi}_{R_\mu}^{(D,N)}\left (\tilde{h}_{j',r}^{(i')}-h_j^{(i)}\right )d\tilde{\nu}_{j',r}^{((i'-n) \text{ mod } p_{j'})}\\
    &+\int \tilde{\Psi}_{R_\mu}^{(D,N)}h_j^{(i)}d\tilde{\nu}_{j',r}^{((i'-n) \text{ mod } p_{j'})}.
\end{align*}
Corollary~\ref{coroKellerLiveranisatisfait} shows that $\tilde{\nu}_{j',r}^{((i'-n) \text{ mod } p_{j'})}$ weakly converges to $\nu_{j'}^{((i'-n) \text{ mod } p_{j'})}$, so that the second integral in the right-hand term converges.
We show that the first one converges to 0.
Let $m_{j',r}$ be the measure defined by $dm_{j',r}=\frac{1}{p_{j'}}\sum_{i'=1}^{p_{j'}}h_{j',r}^{(i)}d\nu_{j',r}^{(i')}$.
According to \cite[Proposition~4]{Sarig1}, $m_{j',r}$ is Gibbs and according to \cite[Proposition~2]{Sarig1}, the functions $h_{j',r}^{(i)}$ are bounded away from zero and infinity on the support of $\nu_{j',r}^{(i)}$, so that
$$\nu_{j',r}^{(i')}([x_1...x_n]) \leq C H(e,x_1...x_n|r)\leq C H(e,x_1...x_n|R_\mu).$$
Using~(\ref{equivalencemesuredeGibbs}), we see that the measure $\nu_{j',r}^{(i')}$ is dominated by the measure $m$ on cylinders.
Since $\tilde{\Psi}_{R_\mu}^{(D,N)}\left (\tilde{h}_{j',r}^{(i')}-h_j^{(i)}\right )$ is locally H\"older, we have
$$\left |\int \tilde{\Psi}_{R_\mu}^{(D,N)}\left (\tilde{h}_{j',r}^{(i')}-h_j^{(i)}\right )d\tilde{\nu}_{j',r}^{((i'-n) \text{ mod } p_{j'})}\right |\lesssim \int \left | \tilde{\Psi}_{R_\mu}^{(D,N)}\left (\tilde{h}_{j',r}^{(i')}-h_j^{(i)}\right )\right |dm.$$
Finally, since $\tilde{\Psi}_{R_\mu}^{(D,N)}$ is bounded, we have
$$\left |\int \tilde{\Psi}_{R_\mu}^{(D,N)}\left (\tilde{h}_{j',r}^{(i')}-h_j^{(i)}\right )d\tilde{\nu}_{j',r}^{((i'-n) \text{ mod } p_{j'})}\right |\lesssim \int |\tilde{h}_{j,r}^{(i)}-h_j^{(i)}|dm.$$
According to Corollary~\ref{coroKellerLiveranisatisfait}, this last quantity converges to 0.

Now,~(\ref{tildehtendsimplementversh}) shows that $\tilde{h}_{j',r}^{(i')}(\emptyset)$ converges and so we deduce from~(\ref{CVfinale02}) that
$$\sum_{k\geq 0}\mathcal{L}_r^k\left (\tilde{\Psi}_{R_\mu}^{(D,N)}\tilde{h}_{j,r}^{(i)}\right )(\emptyset)=\sum_{j'}\frac{\xi^{i,j}_{j',D,N,r}}{|\tilde{P}_{j'}(r)|},$$
where $\xi^{i,j}_{j',D,N,r}$ converges, as $r$ tends to $R_\mu$.
Also recall that we proved in Proposition~\ref{pressureindependentcomponents} that $\tilde{P}_{j}(r)\sim P(r)$ for every $j$, so~(\ref{3.6Gouezel}) yields
$$O\left (\sum_{j=1}^k\frac{1}{\tilde{P}_j(r)}\right )=O\left (I^{(1)}(r)\right ).$$
Finally, we get from~(\ref{CVfinale00}) and~(\ref{CVfinale01}) that
$$\sum_{n\geq 0}\mathcal{L}_r^n \left (1_{E_*}\sum_{k=0}^{n-1}\tilde{\Psi}_{R_\mu}^{(D,N)}\circ T^k\right )=\sum_{j,j'}\frac{\xi^{j,j'}_{D,N,r}}{|\tilde{P}_{j}(r)||\tilde{P}_{j'}(r)|}+O\left (I^{(1)}(r)\right ),$$
where $\xi^{j,j'}_{D,N,r}$ converges.
Consequently,
$$\sum_{n\geq 0}\mathcal{L}_r^n \left (1_{E_*}\sum_{k=0}^{n-1}\tilde{\Psi}_{R_\mu}^{(D,N)}\circ T^k\right )=\frac{\xi_{D,N,r}}{P(r)^2}+O\left (I^{(1)}(r)\right ),$$
where $\xi_{D,N,r}$ converges to some $\xi_{D,N}$.
Therefore,
$$\sum_{n\geq 0}\mathcal{L}_r^n\left (1_{E_*}\sum_{k=0}^{n-1}\tilde{\Psi}_{R_\mu}^{(D,N)}\circ T^k\right )(\emptyset)$$
is asymptotic to $\xi_{D,N}I^{(1)}(r)^2$, as $r$ tends to $R_\mu$.

We thus deduce from~(\ref{dePsiatildePsiCVfinale}) and~(\ref{reecrituretildePsiCVfinale}) that
\begin{equation}\label{CVsommefinaleavecDN}
    \sum_{n\geq 0}\sum_{\gamma\in \hat{S}^n}\sum_{k=0}^{n-1}H(e,\gamma|r)\Psi_r^{(D,N)}(T^k[e,\gamma])=\xi_{D,N}I^{(1)}(r)^2+o_{D,N}\left (I^{(1)}(r)^2\right ).
\end{equation}
Also note that we deduce from~(\ref{CVfinale00}) that
$$\sum_{n\geq 0}\mathcal{L}_r^n \left (1_{E_*}\sum_{k=0}^{n-1}\tilde{\Psi}_{R_\mu}^{(D,N)}\circ T^k\right )=\sum_{k\geq 1}\mathcal{L}_r^k\left (\tilde{\Psi}_{R_\mu}^{(D,N)}\sum_{n\geq 0}\mathcal{L}_r^{n}1_{E_*}\right ).$$
and so according to~(\ref{3.6Gouezel}), we have
$$\sum_{n\geq 0}\mathcal{L}_r^n \left (1_{E_*}\sum_{k=0}^{n-1}\tilde{\Psi}_{R_\mu}^{(D,N)}\circ T^k\right )\lesssim I^{(1)}(r)\sum_{k\geq 1}\mathcal{L}_r^k \tilde{\Psi}_{R_\mu}^{(D,N)}.$$
Thus,~(\ref{Upsilonuniformementbornee}) and~(\ref{upsilonvspsi}) show that
$$\sum_{n\geq 0}\mathcal{L}_r^n \left (1_{E_*}\sum_{k=0}^{n-1}\tilde{\Psi}_{R_\mu}^{(D,N)}\circ T^k\right )\lesssim I^{(1)}(r)^2.$$
Hence,
\begin{equation}\label{xiDNbornee}
    \xi_{D,N}\lesssim 1.
\end{equation}

We finally conclude the proof of Proposition~\ref{propositioncomparaisonI_2etI_1}.

\begin{proof}
Recall that
$$I^{(2)}(r)=\sum_{\gamma\in \Gamma}H(e,\gamma|r)\Phi_r(\gamma)=\sum_{n\geq0}\sum_{\gamma\in \hat{S}^n}\Phi_r(\gamma).$$
According to Proposition~\ref{propestimeePhietPsi},
$$I^{(2)}(r)=I^{(1)}(r)\sum_{n\geq0}\sum_{\gamma\in \hat{S}^n}\sum_{k=0}^{n-1}\Psi_r(T^k[e,\gamma])+O(I^{(1)}(r)).$$
We need to prove that $I^{(2)}(r)/I^{(1)}(r)^3$ converges, as $r$ tends to $R_\mu$.
It is thus enough to show that
$$\frac{1}{I^{(1)}(r)^2}\sum_{n\geq0}\sum_{\gamma\in \hat{S}^n}\sum_{k=0}^{n-1}\Psi_r(T^k[e,\gamma])$$
converges.

Fix $\epsilon>0$.
Choose sequences $D_l$ and $N_l$ that tend to infinity, as $l$ tends to infinity.
Since we want to apply Proposition~\ref{differencePsietPsitronque}, the sequence $D_l$ will actually depend on the sequence $N_l$.
According to~(\ref{xiDNbornee}), we can assume, up to taking a sub-sequence, that $\xi_{D_l,N_l}$ converges to some constant $\xi$.
We will show that the above sum also converges to $\xi$.
According to Proposition~\ref{differencePsietPsitronque}, we can choose $N_l$ and $D_l$ so that for any $l$ large enough,
\begin{align*}
    &\frac{1}{I^{(1)}(r)^2}\left |\sum_{n\geq0}\sum_{\gamma\in \hat{S}^n}\sum_{k=0}^{n-1}\Psi_r(T^k[e,\gamma])-\sum_{n\geq0}\sum_{\gamma\in \hat{S}^n}\sum_{k=0}^{n-1}\Psi_r^{(D_l,N_l)}(T^k[e,\gamma])\right |\\
    &\hspace{2cm}\leq \epsilon \frac{1}{I^{(1)}(r)^2}\sum_{n\geq0}\sum_{\gamma\in \hat{S}^n}\sum_{k=0}^{n-1}\Psi_r(T^k[e,\gamma]).
\end{align*}
Fix a large enough $l$ so that this inequality is satisfied and so that $|\xi_{D_l,N_l}-\xi|\leq \epsilon$.
Now that $l$ is fixed, we set $D=D_l$ and $N=N_l$.
We thus have
\begin{align*}
    &\frac{1}{I^{(1)}(r)^2}\sum_{n\geq0}\sum_{\gamma\in \hat{S}^n}\sum_{k=0}^{n-1}\Psi_r(T^k[e,\gamma])-\xi\\
    &\hspace{2cm}\leq \frac{1}{1-\epsilon}\frac{1}{I^{(1)}(r)^2}\sum_{n\geq0}\sum_{\gamma\in \hat{S}^n}\sum_{k=0}^{n-1}\Psi_r^{(D,N)}(T^k[e,\gamma])-\xi.
\end{align*}
Hence,~(\ref{CVsommefinaleavecDN}), shows that whenever $r$ is close enough to $R_\mu$,
$$\frac{1}{I^{(1)}(r)^2}\sum_{n\geq0}\sum_{\gamma\in \hat{S}^n}\sum_{k=0}^{n-1}\Psi_r(T^k[e,\gamma])-\xi\leq \frac{1}{1-\epsilon}\xi_{D,N}-\xi\leq \frac{\epsilon}{1-\epsilon}+\xi(\frac{1}{1-\epsilon}-1).$$
Similarly
\begin{align*}
    &\xi-\frac{1}{I^{(1)}(r)^2}\sum_{n\geq0}\sum_{\gamma\in \hat{S}^n}\sum_{k=0}^{n-1}\Psi_r(T^k[e,\gamma])\\
&\hspace{2cm}\leq \xi-\frac{1}{1+\epsilon}\frac{1}{I^{(1)}(r)^2}\sum_{n\geq0}\sum_{\gamma\in \hat{S}^n}\sum_{k=0}^{n-1}\Psi_r^{(D,N)}(T^k[e,\gamma])
\end{align*}
and so whenever $r$ is close enough to $R_\mu$,
$$\xi-\frac{1}{I^{(1)}(r)^2}\sum_{n\geq0}\sum_{\gamma\in \hat{S}^n}\sum_{k=0}^{n-1}\Psi_r(T^k[e,\gamma])
\leq \xi (1-\frac{1}{1+\epsilon})-\frac{\epsilon}{1+\epsilon}.$$
Since $\epsilon$ is arbitrary, this shows that
$$\frac{1}{I^{(1)}(r)^2}\sum_{n\geq0}\sum_{\gamma\in \hat{S}^n}\sum_{k=0}^{n-1}\Psi_r(T^k[e,\gamma])\underset{r\to R_\mu}{\longrightarrow}\xi.$$
Finally, we already know that $I^{(2)}(r)/I^{(1)}(r)^3$ is bounded away from 0, independently of $r$, so that $\xi\neq 0$.
This concludes the proof.
\end{proof}

Theorem~\ref{preciseequadiff} is a direct consequence of Proposition~\ref{propositioncomparaisonI_2etI_1}.

\section{From the Green asymptotics to the local limit theorem}\label{SectionFromGreentoLLT}
We can finally prove Theorem~\ref{maintheorem}.
We first deduce from Theorem~\ref{preciseequadiff} the following.

\begin{corollary}\label{equadiffprecise2}
Let $\Gamma$ be a non-elementary relatively hyperbolic group.
Let $\mu$ be a finitely supported, admissible and symmetric probability measure on $\Gamma$.
Assume that the corresponding random walk is non-spectrally degenerate along parabolic subgroups.
Then, for every $\gamma_1,\gamma_2$,
there exists $C_{\gamma_1,\gamma_2}>0$ such that
$$\frac{d}{dr}\left (G(\gamma_1,\gamma_2|r)\right )\underset{r\to R_{\mu}}{\sim} C_{\gamma_1,\gamma_2}\frac{1}{\sqrt{R_\mu-r}}.$$
\end{corollary}

\begin{proof}
For $\gamma_1=\gamma_2=e$, this is a direct consequence of Theorem~\ref{preciseequadiff}, combined with \cite[Lemma~3.2]{DussauleLLT1} which relates the derivatives of the Green function with the sums $I^{(k)}(r)$.
Note that by equivariance, we only need to prove the result with $\gamma_2=e$.
According to Lemma~\ref{lemmafirstderivative}, an asymptotic of $\frac{d}{dr}\left (G(\gamma,e|r)\right )$ is given by an asymptotic of
$$\sum_{\gamma'\in \Gamma}G(\gamma,\gamma'|r)G(\gamma',e|r).$$
Consider $\gamma\in \Gamma$ and set
$$f_r(\gamma')=\frac{G(\gamma,\gamma'|r)}{G(e,\gamma'|r)}.$$
Let $\tilde{f}_r=f_r\circ \phi$.
Then,
$$\sum_{\gamma'\in \Gamma}G(\gamma,\gamma'|r)G(\gamma',e|r)=H(e,e|r)\sum_{n\geq 0}\mathcal{L}_r^n \tilde{f}_r(\emptyset).$$
Since $\gamma$ is fixed, $\tilde{f}_r$ is uniformly bounded.
Strong relative Ancona inequalities also imply that $\tilde{f}_r$ can be extended to a function on $\overline{\Sigma}_A$ which lie in $H_{\rho,\beta}$.
If $\tilde{f}_r$ were uniformly converging to a function $\tilde{f}$, as $r$ tends to $R_\mu$, then we could directly conclude the proof, using~(\ref{3.6Gouezel}).
However, exactly like for $\Psi_r$, this uniform convergence does not necessarily hold and we have to truncate $\tilde{f}_r$.
We can apply the same strategy as for the proof of Proposition~\ref{propositioncomparaisonI_2etI_1} to conclude.
\end{proof}

Theorem~\ref{maintheorem} follows directly from Corollary~\ref{equadiffprecise2} and \cite[Theorem~9.1]{GouezelLalley}.
Corollary~\ref{maincorollary} thus follows from \cite[Proposition~4.1]{Gouezel1}.
Beware that the symmetry assumption on the measure $\mu$ is needed here, see the remarks in \cite[Section~4]{Gouezel1}. \qed

\bibliographystyle{plain}
\bibliography{LLT_2}

\end{document}